\documentclass[11pt]{amsart}

\usepackage{amssymb}
\usepackage{amsthm}
\usepackage{amsmath}
\usepackage{graphicx}
\usepackage{amsaddr}
\usepackage{comment}
\usepackage[usenames,dvipsnames]{xcolor}
\usepackage[margin = 1in]{geometry}
\usepackage{enumerate}
\usepackage[toc,page]{appendix}
\usepackage{lineno}
\usepackage{color}
\usepackage{bbm}
\usepackage{hyperref}
\usepackage{mhchem}
\usepackage{siunitx}
\newcommand{\red}{}
\newcommand{\blue}{}
\definecolor{mygreen}{rgb}{0.1,0.75,0.2}

\providecommand{\bbs}[1]{\left(#1\right)}

 \newtheorem{thm}{Theorem}[section]
 \newtheorem{cor}[thm]{Corollary}
 \newtheorem{lem}[thm]{Lemma}
 \newtheorem{prop}[thm]{Proposition}

 \theoremstyle{definition}

 \theoremstyle{remark}
 \newtheorem{rem}[thm]{Remark}

 \numberwithin{equation}{section}

\DeclareMathOperator{\KL}{KL}
\DeclareMathOperator{\ran}{Ran}
\DeclareMathOperator{\kk}{Ker}
\DeclareMathOperator{\argmin}{argmin}

\newcommand{\lra}{\longrightarrow}

\newcommand{\peq}{\vec{x}^{\text{s}}}
\newcommand{\xss}{x^{\xs}}
\newcommand{\la}{\langle}
\newcommand{\ra}{\rangle}
\newcommand{\pt}{\partial}

\newcommand{\eps}{\varepsilon}

\newcommand{\ud}{\,\mathrm{d}}
\newcommand{\8}{\infty}

\newcommand{\bR}{\mathbb{R}}
\newcommand{\bZ}{\mathbb{Z}}
\newcommand{\bE}{\mathbb{E}}
\newcommand{\bP}{\mathbb{P}}
\newcommand{\vs}{\vec{s}}
\newcommand{\vx}{\vec{x}}
\newcommand{\V}{\scriptscriptstyle{\text{v}}}
\newcommand{\A}{\scriptscriptstyle{\text{A}}}
\newcommand{\B}{\scriptscriptstyle{\text{B}}}
\newcommand{\tot}{\scriptscriptstyle{\text{tot}}}
\newcommand{\C}{\scriptscriptstyle{\text{C}}}
\newcommand{\R}{\scriptscriptstyle{\text{R}}}
\newcommand{\xs}{\scriptscriptstyle{\text{S}}}

\newcommand{\mr}{\scriptscriptstyle{\text{MR}}}
\newcommand{\vxv}{\vec{x}_{\V}}
\newcommand{\vyv}{\vec{y}_{\V}}
\newcommand{\pe}{\vec{x}^{\text{e}}}
\newcommand{\vxr}{\vec{x}^{\R}}
\newcommand{\vpr}{\vec{p}^{\R}}
\newcommand{\vprr}{\vec{p}^{\mr}}
\newcommand{\vp}{\vec{p}}
\newcommand{\xa}{\vec{x}^{\A}}

\newcommand{\xb}{\vec{x}^{\B}}
\newcommand{\xc}{\vec{x}^{\C}}
\newcommand{\va}{\vec{\alpha}}
\newcommand{\vb}{\vec{\beta}}
\newcommand{\vr}{\vec{r}}
\newcommand{\vm}{\vec{m}}
\newcommand{\vR}{\vec{R}}
\newcommand{\vq}{\vec{q}}
\newcommand{\vy}{\vec{y}}
\newcommand{\ac}{\ce{Act}}

\newcommand{\cv}{C^{\scriptscriptstyle{\text{v}}}}

\begin{document}

\title[Chemical reactions from a Hamiltonian viewpoint]{Revisit of   macroscopic dynamics for some non-equilibrium chemical reactions from a Hamiltonian viewpoint}

\author[Y. Gao]{Yuan Gao}
\address{Department of Mathematics, Purdue University, West Lafayette, IN}
\email{gao662@purdue.edu}

\author[J.-G. Liu]{Jian-Guo Liu}
\address{Department of Mathematics and Department of
  Physics, Duke University, Durham, NC}
\email{jliu@math.duke.edu}

\keywords{Conservative-dissipative decomposition, positive entropy production rate, time reversal,
transition path with energy barrier, thermodynamic limit, large deviation principle}

\subjclass[2010]{80A30, 35F21, 70H33, 49N99}
%\keywords{Large deviation, time reversal, transition path, Freidlin-Wentzell theory, Pontryagin maximum principle}

\date{\today}

\maketitle

\begin{abstract}
 Most biochemical reactions in living cells are  open systems interacting with environment through chemostats to exchange both energy and materials. At a mesoscopic scale, the number of each species in those biochemical reactions can be modeled by a random time-changed Poisson processes. To characterize   macroscopic behaviors in the large number limit, the law of large numbers in the path space determines a mean-field limit nonlinear reaction rate equation describing the dynamics of the concentration of species, while the WKB expansion for the chemical master equation yields a Hamilton-Jacobi equation  and the Legendre transform  of the corresponding Hamiltonian  gives the good rate function (action functional) in the large deviation principle.   In this paper, we decompose a general macroscopic reaction rate equation into a conservative part and a dissipative part in terms of the stationary solution to the Hamilton-Jacobi equation.  This stationary solution is used to determine the energy landscape and thermodynamics for general chemical reactions, which particularly maintains a positive entropy production rate at a non-equilibrium steady state. The associated energy dissipation law   at both the mesoscopic and macroscopic  levels is proved together with a passage from the mesoscopic to macroscopic one. A non-convex energy landscape emerges from the convex mesoscopic relative entropy functional in the large number limit, which picks up the non-equilibrium features.   The existence of this stationary solution is ensured by the optimal control representation at an undetermined time horizon for the weak KAM solution to the stationary Hamilton-Jacobi equation. 
  Furthermore, we use a symmetric Hamiltonian to study a   class of non-equilibrium    enzyme reactions, which leads to nonconvex energy landscape due to flux grouping degeneracy and reduces the conservative-dissipative decomposition to an Onsager-type strong gradient flow.  
This symmetric Hamiltonian implies that the transition paths between multiple   steady states (rare events in biochemical reactions) is  a modified time reversed least action path  with associated   path affinities and energy barriers.  We illustrate this idea through a bistable catalysis  reaction and compute the energy barrier for the transition path connecting two   steady states via its energy landscape.   
\end{abstract}

\section{Introduction}

At a mesoscopic scale, chemical or biochemical reactions can be understood from a probabilistic viewpoint. A convenient way to stochastically  describe chemical reactions is via random time-changed Poisson processes; c.f. \cite{Kurtz15}. Based on this, one can observe `statistical properties' of chemical reactions in the thermodynamic limit. For instance, the law of large numbers gives the `mean path' of a chemical reaction while the large deviation principle can give  rough estimates on  the probability of the occurrence in a vicinity of any path, particularly the  transition path (the most probable path) between two stable states; c.f.  \cite{Freidlin_Wentzell_2012}.
   {\blue In this paper, we investigate various macroscopic behaviors for general chemical reactions, including the conservative-dissipative decomposition for macroscopic dynamics, the passage  from mesoscopic to macroscopic   free energy dissipation relations and   symmetric structures brought by Markov chain detailed balance  in some enzyme reactions.  The studies for  non-equilibrium thermodynamics and metastability in biochemical oscillations was pioneered by Prigogine \cite{prigogine1967introduction}. The coexistence of multiple stable steady states breaks the chemical version detailed balance \eqref{DB} or complex balance \eqref{cDB} properties for chemical reactions  and leads to bifurcations and transition paths. However, some enzyme reactions,  most important non-equilibrium reactions in an open system to maintain metabolite concentrations,  can still be characterized by a process with a Markov chain detailed balance \eqref{master_db_n}.   This mesoscopic   reaction process yields the energy landscape $\psi^{ss}$ of the chemical reaction and the corresponding macroscopic Hamiltonian is symmetric w.r.t $\nabla \psi^{ss}$; see \eqref{newS}. This enables us to study   transition paths, energy barriers and gradient flow structures for a class of non-equilibrium dynamics with multiple steady states.   Before we introduce the main results, we first review some backgrounds for the macroscopic limiting ODE from the large number limit of  the mesoscopic stochastic processes and   backgrounds for  the Wentzel–Kramers–Brillouin (WKB) expansion, a corresponding  Hamiltonian $H$ and the good rate function in an associated large deviation principle.}

\subsection*{Background for large number process and its macroscopic limiting  ODE}
Chemical reaction  with $i=1,\cdots,N$ species $X_i$ and $j=1,\cdots,M$ reactions can be kinematically described as
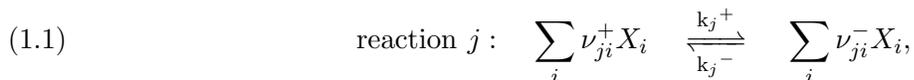
\begin{equation}\label{CRCR}
\text{ reaction }j: \quad \sum_{i} \nu_{ji}^+ X_{i} \quad \ce{<=>[k_j^+][k_j^-]} \quad   \sum_i \nu_{ji}^- X_i,
\end{equation}
where nonnegative integers $\nu_{ji}^\pm \geq 0$ are stoichiometric coefficients and $k_j^\pm\geq 0$ are reaction rates for the $j$-th forward/backward reactions. Denote $\nu_{ji} :=  (\nu_{ji}^- - \nu_{ji}^+)$ as the net change in molecular numbers for species $X_i$ in the $j$-th forward reaction. The $M\times N$ matrix $\nu:= \bbs{\nu_{ji}}, j=1,\cdots,M, \, i=1,\cdots, N$ as the Wegscheider matrix and $\nu^T$ is referred as the stoichiometric matrix \cite{mcquarrie1997physical}. The column vector $\vec{\nu}_j:= \vec{\nu}_j^- - \vec{\nu}_j^+ := \bbs{\nu_{ji}^- - \nu_{ji}^+}_{i=1:N}\in \mathbb{Z}^N$ is  called the reaction vector for  the $j$-th reaction. In this paper,  all vectors  $\vec{X}= \bbs{X_i}_{i=1:N} \in \bR^N$ and    $ \bbs{\varphi_j}_{j=1:M},\, \bbs{k_j}_{j=1:M}\in \bR^M$ are  column vectors.    We remark the   description \eqref{CRCR} includes reactions both in a closed system and an open system. In an open system, $\vec{\nu}_j^+=\vec{0}$ represents birth/death reaction with the corresponding reactant on the right hand side  called chemostats and birth rate $k_j^+$, death rate $k_j^-$; see a bistable example in Section \ref{subsub_sch} for  materials/energy exchange with environment.

 Let  the space of natural numbers $\mathbb{N}$ be the state space of the counting process $X_i(t)$ and let $X_i(t) \in \mathbb{N}$ for $i=1,\cdots,N$ be  the number of each species in those biochemical reactions.
The random time-changed Poisson representation for chemical reactions \eqref{CRCR} is, c.f. \cite{kurtz1980representations, Kurtz15}, 
\begin{equation}\label{CR}
\begin{aligned}
&\vec{X}(t) = \vec{X}(0) + \sum_{j=1}^M \vec{\nu}_{j}\bbs{ \mathbbm{1}_{\{\vec{X}(t_-) + \vec{\nu}_j \geq 0 \}}  Y^+_j \bbs{\tilde{t}_j^+}- \mathbbm{1}_{\{\vec{X}(t_-) - \vec{\nu}_j \geq 0 \}} Y^-_j \bbs{\tilde{t}_j^-}},\qquad \tilde{t}^\pm_j(t) := \int_0^t \lambda_j^\pm(s)\ud s,
\end{aligned}
\end{equation}
where for the $j$-th reaction channel,   $Y^\pm_{j}(t)$ are i.i.d.   unit rate Poisson processes and $\mathbbm{1}$ is the characteristic function indicating that there is no reaction if the next state $\vec{X}(t_-)\pm\vec{\nu}_j $ is negative for some component. Here and in the following a vector $\vec{x}\geq 0$ is understood as componentwisely nonnegative.  The existence and uniqueness to the stochastic equation \eqref{CR} was proved by \cite{kurtz1980representations, Kurtz15} in terms of the corresponding martingale problem.
In \eqref{CR},  the intensity function $\lambda_j^\pm(s)=\varphi_j^\pm(\vec{X}(s))$ for the time clock $\tilde{t}^\pm_j(t)$ is usually chosen as the mesoscopic law of mass action (LMA)
 \begin{equation}\label{newR}
  {\varphi}_j^\pm(\vec{X}) = k_j^\pm V  \prod_{\ell=1}^{N} \frac{X_\ell !}{V^{\nu_{j\ell}^\pm} \bbs{X_\ell - \nu_{j\ell}^\pm}!}. 
\end{equation}
 Here $V\gg 1$ is the volume for species in the chemical reaction in a container. Because we assume chemical reactions in a container is independent of molecule position and the molecular number is proportional to the container volume, we call the limit for the large number of molecules as thermodynamic limit or macroscopic limit.   Rescale the process \eqref{CR} as  $\cv _i:=\frac{X_i}{V}$ and denote the rescaled discrete state variable as  $\vxv:=\frac{\vec{n}}{V},\, \vec{n}\in \mathbb{N}^N$.
Denote the forward and backward rescaled fluxes  as
$
\tilde{\Phi}^\pm_j(\vxv):= \frac{\varphi_j(\vec{n})}{V}.
$
Then the   large number process $\cv $ satisfies
\begin{equation} \label{Csde}
\begin{aligned}
\cv(t) = \cv(0) + \sum_{j=1}^M  \frac{\vec{\nu}_{j} }{V} \Bigg(\mathbbm{1}_{\{\cv(t_-)+\frac{\vec{\nu}_j}{V}  \geq 0\}} Y^+_j  \bbs{V\int_0^t \tilde{\Phi}^+_j(\cv(s))\ud s}\\
-\mathbbm{1}_{\{\cv(t_-)-\frac{\vec{\nu}_j} {V} \geq 0\}}Y^-_j  \bbs{V\int_0^t \tilde{\Phi}^-_j(\cv(s))\ud s}\Bigg).
\end{aligned}
\end{equation}
{\red The `no reaction' constraints $\cv(t^-)\pm \vec{\nu}_{j} h\geq 0$ is to ensure that there is actually no jump if  the number of some species will be negative in the container. This `no reaction' correction to process \eqref{Csde} was also noticed in \cite[eq(28)]{anderson2019constrained}, where a very similar `no reaction' constraint was imposed near the relative boundary of the positive orthant.}

For a chemical reaction modeled by \eqref{CR}, denote the   counting probability of $\cv (t)$ as $p(\vxv,t)=\bE(\mathbbm{1}_{\vxv}(\cv (t))$, where $\mathbbm{1}_{\vxv}$ is the indicator function. Then $p(\vxv,t)$  satisfies the  chemical master equation (CME), c.f. \cite{Kurtz15}
\begin{equation}\label{rp_eq}
\begin{aligned}
 \frac{\ud }{\ud t} p(\vxv, t) = (Q_{\V}^* p)(\vxv,t)=  &V\sum_{j=1, \vxv- \frac{\vec{\nu}_{j}}{V}\geq 0}^M   \bbs{\tilde{\Phi}^+_j(\vxv-\frac{\vec{\nu}_{j}}{V})     p(\vxv-\frac{\vec{\nu}_{j}}{V},t) - \tilde{\Phi}^-_j(\vxv)     p(\vxv,t) } \\
  &+    V\sum_{j=1, \vxv+ \frac{\vec{\nu}_{j}}{V}\geq 0}^M  \bbs{    \tilde{\Phi}^-_j(\vxv+\frac{\vec{\nu}_{j}}{V})     p(\vxv+\frac{\vec{\nu}_{j}}{V},t) - \tilde{\Phi}^-_j(\vxv) p(\vxv,t)}.
 \end{aligned}
\end{equation}
Here $Q^*_{\V}$ is the transpose of the generator $Q_{\V}$ of process $\cv (t)$. {\blue Here we call the constraint  $\vxv\pm \vec{\nu}_{j} h\geq 0$  the `no reaction' boundary condition for CME which  inherits from \eqref{Csde}; see \cite{GL_vis}. In \cite{Kurtz71, Kurtz15}, this `no reaction' restriction was omitted in the process $\cv$, thus to derive the master equation and generator including this `no reaction' constraint, 
we give a pedagogical derivation in Appendix \ref{app:master}.} We refer to \cite{Gauckler14} for the existence and regularity of solutions to CME.

The   mesoscopic jumping process $\cv $ in \eqref{Csde} can be regarded as a large number interacting particle system.  
In the large number limit (thermodynamic limit), this interacting particle system can be approximately described by a mean field equation, i.e., a macroscopic nonlinear chemical reaction-rate equation. If  the law of large numbers in the mean field limit holds, i.e., $p(\vxv, t)\to \delta_{\vx(t)}$ for some $\vx (t)$, then the limit $\vx (t)$ describes the   dynamics of the concentration of $N$ species in the continuous state space $\mathbb{R}_+^N:=\{\vx \in \bR^N; x_i \geq 0\}$ and
  is given by the following reaction rate equation (RRE), also known as chemical kinetic rate equation,
\begin{equation}\label{odex}
\frac{\ud}{\ud t} \vec{x} =  \sum_{j=1}^M \vec{\nu}_j \bbs{\Phi^+_j(\vec{x}) - \Phi^-_j(\vec{x})}.
\end{equation}
  Here the macroscopic fluxes $\Phi_j^\pm$ satisfy the macroscopic LMA
\begin{equation}\label{lma}
\Phi^\pm_j(\vxv)= k^\pm_j  \prod_{\ell=1}^{N} \bbs{x_\ell}^{\nu^\pm_{j\ell}},
\end{equation}
which can be viewed as a large number approximation for the mesoscopic LMA \eqref{newR}. This RRE with LMA were first proposed by \textsc{Guldberg \& Waage} in 1864.   The limiting macroscopic LMA in RRE \eqref{odex} is same as long as the  mesoscopic LMA satisfies  $\frac{\varphi_j(\vec{n})}{V}\approx \Phi^\pm_j(\vxv)= k^\pm_j  \prod_{\ell=1}^{N} \bbs{x_\ell}^{\nu^\pm_{j\ell}}$. 
Indeed, Kurtz \cite{Kurtz71} proved the law of large numbers for the large number process $\cv (t)$; c.f.,  \cite[Theorem 4.1]{Kurtz15}. Suppose $\Phi^\pm_j$ is local Lipschitz.   If $\cv (0)\to \vec{x}(0)$ as $V\to+\8$, then for any $\eps>0, t>0$,
\begin{equation}\label{lln}
\lim_{V\to +\8} \bP\{ \sup_{0\leq s\leq t}|\cv (s)-\vec{x}(s)|\geq \eps \}=0.
\end{equation}
Thus we will also call the large number limiting ODE \eqref{odex} as the macroscopic RRE. This gives a passage from the mesoscopic LMA \eqref{newR} to the macroscopic one \eqref{lma}.   {\blue In Appendix \ref{app_meanfield}, we   give a  pedagogical derivation for this mean field limit result \eqref{lln} to include   `no reaction' boundary condition.} We also refer to recent results in \cite{{maas2020modeling}} which proves the evolutionary $\Gamma$-convergence from   CME to the Liouville equation and thus starting from a deterministic state $\vx_0$, \cite[Theorem 4.7]{{maas2020modeling}} recovers Kurtz's  results on the mean field limit of CME.

If there exists a positive vector $\vec{m}\in\bR^N_+$ (for instance due to the conservation of mass for each reaction $j$ in a closed system) such that 
\begin{equation}\label{mb}
\vec{\nu}_j \cdot  \vec{m}=0, \quad j=1,2, \cdots,M,
\end{equation}
 where $\vec{m}=(m_i)_{i=1:N}$ and $m_i$ represents the  molecular weight for the $i$-th species, then
  the  Wegscheider matrix $\nu$   has a nonzero kernel, i.e., 
$
\dim \bbs{\kk(\nu)} \geq 1
$
and we have a direct decomposition for the species space
\begin{equation}\label{direct}
\bR^N = \ran(\nu^T)\oplus \kk(\nu).
\end{equation}
Denote the stoichiometric space $G:=\ran(\nu^T)$. Given an initial state $\vec{X}_0 \in \vq+G$, $\vq\in\kk(\nu)$, the dynamics of both mesoscopic \eqref{Csde} and macroscopic \eqref{odex} states stay in the same space $G_q:=\vq+G$, called a stoichiometric compatibility class. We will see later the corresponding Hamiltonian/Lagrangian from WKB expansion for \eqref{rp_eq} are strictly convex in $G$ while degenerate in $\kk(\nu)$. Below we discuss the uniqueness of steady states within one stoichiometric compatibility class for the RRE detailed/complex balance case.

Denote a steady state to RRE \eqref{odex} as $\peq$, which satisfies
\begin{equation} 
 \sum_{j=1}^M \vec{\nu}_j \bbs{\Phi^+_j(\peq) - \Phi^-_j(\peq)} =0.
\end{equation}
The detailed balance condition for RRE \eqref{odex} is defined by \textsc{Wegscheider 1901, Lewis 1925} as: 
(i) there exists a  $\peq>0$ (componentwise); and (ii) $\peq$ satisfies
\begin{equation}\label{DB}
\Phi^+_j(\peq) - \Phi^-_j(\peq) = 0, \quad \forall j.
\end{equation}
We call RRE \eqref{odex}  detailed balanced if there exists such a detailed balanced state $\peq$. 
This immediately gives a necessary condition that both $k_{j}^\pm>0$, i.e., the reaction is reversible. 
This concept of detailed balance for RRE \eqref{DB} is commonly used in chemistry and biology, while it is different from the Markov chain detailed balance condition \eqref{master_db_n} for the mesoscopic jump process $C^V$.  {\blue The latter is a more proper mathematical definition for the detailed balance condition and includes some non-equilibrium reactions which can not be characterized via the more constrained chemical version detailed balance \eqref{DB}; see Section \ref{sec_reversal}. We will use both concepts in this paper, so we call \eqref{DB} the detailed balance for RRE and \eqref{master_db_n} the Markov chain detailed balance for CME, respectively.}

%From the macroscopic detailed balanced state $\peq$, one can construct a mesoscopic detailed balanced invariant measure $\pi_{\V}$ via product Poisson distribution for CME; see \cite{Anderson_Craciun_Kurtz_2010, Kurtz15}  and Appendix \ref{app_psi}.

Under RRE detailed balance condition \eqref{DB}, all the positive steady solutions to \eqref{odex} are detailed balanced and are characterized by
$ \peq  e^{\vq}>0$ for some $\vq \in \kk(\nu)$. This is also true for a weaker condition called the complex balance condition \eqref{cDB}.  We summarize the well-known result on the uniqueness of the RRE detailed/complex balanced steady state discovered in \cite[Theorem 6A]{Horn72} (\cite[Theorem 3.5]{Kurtz15}) as Lemma \ref{lem:equi} and the deficiency zero theorem proved by \textsc{Horn, Feinberg} \cite{feinberg1972chemical, Horn72} is revisited in Section \ref{subsec_de}.
  Lemma \ref{lem:equi} says that   for  each stoichiometric compatibility class,   there is only one equilibrium steady state for a detailed/complex balanced RRE system.  
 
   On the contrary,  non-equilibrium chemical reaction system has coexistent steady states and nonzero steady fluxes, so how to find transition paths between different non-equilibrium steady states and to compute the corresponding energy barriers for the macroscopic RRE \eqref{odex} are the main goals of this paper. Particularly, in many biochemical reactions, such as heterogeneous catalytic oxidations and some enzyme reactions, the RRE detailed/complex balance conditions do not hold. We will use a symmetric Hamiltonian  \eqref{newS} to study some non-equilibrium reaction dynamics,    which enables us to explore  Onsager's strong form   gradient flow structure (see \eqref{strong_GF})  and to compute the  explicit transition path formula with the associated  path affinity (see Proposition \ref{thm2}).
Before describing our main results, let us further review some important properties for a Hamiltonian raising from WKB expansion of $p(\vxv, t)$ below.

\subsection*{Background for WKB expansion, Hamiltonian and large deviation principle for chemical reactions}
Besides the macroscopic trajectory $\vx(t)$ given by the law of large numbers,     WKB expansion for $p(\vxv, t)$ in CME \eqref{rp_eq} is another standard  method   \cite{Kubo73, Hu87, Dykman_Mori_Ross_Hunt_1994, SWbook95,   QianGe17}, which   builds up a more informative bridge between the mesoscopic dynamics and the macroscopic behaviors. We remark the WKB expansion has different names in different fields, such as the eikonal approximation, or the instanton technique, or the nonlinear semigroup, or the Cole-Hopf transformation.

To characterize the exponential asymptotic behavior, we assume there exists a continuous function $\psi(\vx,t)$ such that $p(\vxv,t)$ has a WKB reformulation
\begin{equation}\label{1.3}
 p(\vxv,t) =  e^{-V\psi(\vxv,t)}, \quad p(\vxv,0)=p_0(\vxv).
\end{equation}
The fluctuation on path space, i.e., the large deviation principle, can be computed through WKB expansion, and the good rate function for the large number process in a chemical reaction is   rigorously proved by \textsc{Agazzi} et.al  in  \cite[Theorem 1.6]{Dembo18}; see explanations below.
We know $\psi$ satisfies
\begin{equation}\label{reHJE}
\pt_t \psi(\vxv,t) = -\frac{1}{V} e^{V\psi(\vxv,t)} Q^*_V e^{-V \psi(\vxv,t)}=: - \frac{1}{V} H^*_V (V\psi), \quad \psi(\vxv,0) = -\frac{1}{V} \log p_0(\vxv).
\end{equation}
 By  Taylor's expansion of $\psi(\vx\pm \frac{\vec{\nu}_j}{V},t)$ in \eqref{rp_eq} and \eqref{reHJE}, 
 we obtain the following  Hamilton-Jacobi equation (HJE) for the rescaled master equation \eqref{rp_eq} for $\psi$
 \begin{equation}\label{HJEpsi}
  \pt_t \psi(\vec{x}, t) =   -\sum_{j=1}^M     \bbs{ \Phi^+_j(\vec{x})\bbs{e^{\vec{\nu_j} \cdot \nabla \psi(\vec{x},t)}   -  1} +  \Phi^-_j(\vec{x})\bbs{e^{-\vec{\nu_j} \cdot \nabla \psi(\vec{x},t)}   - 1 }};
 \end{equation}
 see also derivations for \eqref{tempHJE} later.
Define
 Hamiltonian $H(\vp,\vx)$ on $\bR^N\times \bR^N$ as
\begin{equation}\label{H}
H(\vec{p},\vec{x}):= \sum_{j=1}^M     \bbs{ \Phi^+_j(\vec{x})e^{\vec{\nu_j} \cdot \vec{p}}   -  \Phi^+_j(\vec{x}) +  \Phi^-_j(\vec{x})e^{-\vec{\nu_j} \cdot\vec{p}}   -  \Phi^-_j(\vec{x}) } .
\end{equation}
Then the HJE for $\psi(\vx,t)$ can be recast as 
\begin{equation}\label{HJE2psi}
\pt_t \psi + H(\nabla \psi, \vx) =0.
\end{equation}
The WKB analysis above defines a Hamiltonian  $H(\vp,\vx)$, which contains almost all the information for the macroscopic dynamics.
  We remark this kind of WKB expansion was first used by \textsc{Kubo} et.al. \cite{Kubo73} for  master equations for general Markov processes and later was applied to CME by  \textsc{Hu} in \cite{Hu87}. In \cite{Dykman_Mori_Ross_Hunt_1994}, \textsc{Dykman} et.al. first derived the HJE \eqref{HJE2psi} with the associated Hamiltonian $H$ in \eqref{H} and reviewed the symmetry of Hamiltonian $H$ in the RRE detailed balance case $H(\vp, \vx) = H(\log \frac{\vx}{\peq} - \vp, \vx).$

Equivalent to WKB reformulation for CME, one can define    Varadhan's nonlinear semigroup \cite{Varadhan_1966, feng2006large} for process $\cv (t)$ via the WKB reformulation for the backward equation
 \begin{equation}\label{semigroup}
u(\vxv,t)= \frac{1}{V} \log \bE^{\vxv}\bbs{e^{V u_0(\cv _t)}} =:\bbs{S_t u_0}(\vxv)
\end{equation}
and as $V\to+\8$, with the same Hamiltonian $H$, $u(\vx,t)$ satisfies
\begin{equation}\label{HJEu}
\pt_t u - H(\nabla u, \vx) =0.
\end{equation}
{\blue Comparing with \eqref{HJE2psi}, the limiting HJE after WKB expansion for forward and backward equation only has a sign difference in the time derivative.
The rigorous convergence from the Varadhan's nonlinear semigroup \eqref{semigroup} to the viscosity solution of   HJE \eqref{HJEu} was proved in \cite{GL_vis} by reformulating \eqref{semigroup} as a monotone scheme to HJE \eqref{HJEu}. Two difficulties  brought by the `no reaction' boundary condition and the polynomial growth rate for the coefficients $\Phi(\vx)^\pm_j$ in Hamiltonian when constructing unique viscosity solution were overcame in \cite{GL_vis} by constructing upper/semicontinuous envelopes which inherit the `no reaction' constraint and by constructing barriers to control far field values.   Based on this convergence and the Lax-Oleinik's representation for the viscosity solution to HJE \eqref{HJEu} 
\begin{equation}\label{LO}
u(\vx,t) = \sup_{\vy\in \mathbb{R}^N} \bbs{u_0(\vy) - I_{\vx,t}(\vy)}, \quad I_{\vx,t}(\vy) = \inf_{\gamma(0)=\vx,\gamma(t)=\vy} \int_0^t L(\dot{\gamma}(s),\gamma(s)) \ud s,
\end{equation}
\cite{GL_vis} verified the Varadhan's inverse lemma for the large deviation principle \cite{Bryc_1990}. Here $L$ is the convex conjugate of $H$ and $I_{\vx,t}(\vy)$ is the least action between fixed initial point $\vx$ and ending point $\vy$ at time $t$.  The Lax-Oleinik's representation for $u(\vx,t)$ can be interpreted as  a deterministic optimal control problem with terminal profit $u_0$ at $t$ and running cost given by the least action $I_{\vx,t}$. Combining this convergence  with the exponential tightness of $\cv(t)$ at single times,    the large deviation principle for the random variable $\cv (t)$ at any time $t$   with good rate function $I_{\vx_0,t}(\cdot)$   was then proved.  The sample path large deviation principle, which requires further the exponential tightness in the path space, is more involved and we refer to \textsc{Agazzi} et.al   \cite[Theorem 1.6]{Dembo18}.}
The relation between the Hamiltonian $H$ and the rate function in the large deviation principle for a general Markov process was introduced by \textsc{Fleming and Sheu} \cite{fleming1983optimal}; see also \cite{feng2006large}. 
In  Section \ref{subsec_wkb} and Section \ref{sec4.0}, we   summarize key properties for  $H(\vp,\vx)$ and its convex conjugate $L(\vs,\vx)$ and their relations  to the macroscopic RRE, the HJE for the phase variable $\psi$, and also the good rate function   in the large deviation principle.  
%Precisely, $H(\vp,\vx)$ vanishes for $\vp\in \kk(\nu)$ while $H(\vp,\vx)$ is strictly convex for  $\vp\in G:=\ran(\nu^T)$. The convex conjugate of $H(\vp,\vx)$, the Lagrangian $L(\vs,\vx)$ for $\vs\in \bR^N$,   is also strictly convex  for $\vs\in G$ while $L(\vs,\vx)=+\8$ for $\vs \notin G$. 
The solution $\vx(t)$ to RRE \eqref{odex} is shown to be  a least action curve with zero action cost $\ac(\vx(\cdot))=0$; see Lemma \ref{lem:least}. Indeed, $\vx(t)$ is a curve following the Hamiltonian dynamics with zero momentum $\vp=0$.
% while in general, it is well known that the Hamiltonian dynamics with Hamiltonian $H$ corresponds to a least action curve minimizing action functional $\ac(\vx(\cdot))=\int_0^T L(\dot{\vx},\vx)\ud t$.
%At the large deviation regime with positive action cost, it is rigorously proved by \textsc{Agazzi} et.al  in  \cite[Theorem 1.6]{Dembo18}
%that the Lagrangian $L(\vs,\vx)$ actually gives  the good rate function $\ac(\vx(\cdot))=\int_0^T L(\dot{\vx},\vx) \ud t$ for the large derivation principle in the path space  for process $\cv$. 

\subsection*{Main results}
  In Section \ref{sec3}, we utilize  the dynamic and stationary solutions to HJE \eqref{HJE2psi} to study the characterization and decomposition of RRE. We first recast the RRE \eqref{odex} as a bi-characteristic of HJE
\begin{equation}
\frac{\ud}{\ud t} \vx = \nabla_p H(\vec{0}, \vx), \quad \vp\equiv \vec{0}.
\end{equation}
This directly gives the characterization of the macroscopic RRE trajectory $\vx(t)$, i.e., 
\begin{equation}
\vec{x}(t) =  \argmin_{\vec{x}} \, \psi(\vec{x},t), \quad \text{ for all } t\in[0,T].
\end{equation}

Second, the stationary solution $\psi^{ss}(\vx)$ to HJE \eqref{HJE2psi} plays the role of a free energy, by which, we decompose the RRE as a conservative part and a dissipation part
\begin{equation}
\begin{aligned}
&\dot{\vx} =W(\vx) \,\, - \,  K(\vx) \nabla \psi^{ss}(\vx),\\
&W(\vx):=\int_0^1 \nabla_p H(\theta \nabla \psi^{ss}(\vx),\vx) \ud \theta, \quad K(\vx):=   \int_0^1 (1-\theta) \nabla^2_{pp} H(\theta \nabla \psi^{ss}(\vx), \vx) \ud \theta.
\end{aligned} 
\end{equation}
Here the conservative part is orthogonal to $\nabla \psi^{ss}$, i.e. $\la W(\vx), \nabla\psi^{ss}(\vx) \ra=0$ 
 and the   dissipation part is expressed using Onsager's nonnegative definite response operator $K(\vx)$ for $\dot{\vx}$ w.r.t generalized force $\nabla \psi^{ss}(\vx)$; see details in Theorem  \ref{thm_decom}. This orthogonal decomposition yields that any increasing function of the energy landscape $\phi(\psi^{ss})$ serves as the Lyapunov function of RRE; see \eqref{Ly_n}.   GENERIC formalism and anti-symmetric structures for RRE with additional mass conservation law are also discussed in Section \ref{sec_3decom}.  
 
Third, in Section \ref{subsec_thermo}, we use  the above conservative-dissipative decomposition for RRE to derive the thermodynamic relations for general   chemical reactions, i.e., we express the total entropy production rate as  the adiabatic  and nonadiabated entropy production rate
\begin{equation}
\begin{aligned}
T \dot{S}_{\tot}&=T \dot{S}_{na} + T \dot{S}_{a} \geq 0,\\
T \dot{S}_{na}&=  k_{\B}T \la K(\vx)   \nabla  \psi^{ss}(\vx),\,  \nabla \psi^{ss}(\vx) \ra\geq 0,
\\
T \dot{S}_{a}&=k_{\B}T \sum_j \bbs{ \KL(\Phi^+_j (\vec{x}(t))||\Phi^-_j (\vec{x}(t))e^{-\vec{\nu}_j \cdot \nabla \psi^{ss}}) + \KL(\Phi^-_j (\vec{x}(t))||\Phi^+_j (\vec{x}(t))e^{\vec{\nu}_j \cdot \nabla \psi^{ss}})} \geq 0,
\end{aligned}
\end{equation}
where $\KL(\vec{x}|| \vx^*) := \sum_i \bbs{  x_i \ln \frac{x_i}{x^{*}_i} - x_i + x^*_i}$ is the relative entropy; see Proposition \ref{prop_Stot}.
Particularly, as $t\to+\8$ and $\vx(t)$ goes to a non-equilibrium steady state (NESS) $\peq$, remaining a strictly positive entropy production rate is an important feature of a non-equilibrium chemical reaction \cite{kondepudi2014modern}
\begin{equation}
T \dot{S}_{a} \to k_{\B}T \sum_j \bbs{\Phi^+_j(\peq) - \Phi^-_j(\peq)}\log\frac{\Phi^+_j(\peq)}{\Phi^-_j(\peq)}>0.
\end{equation}

Fourth,  for general non-equilibrium RRE and the corresponding CME \eqref{rp_eq},  we also derive a  $\phi$-divergence energy dissipation law  based on the $Q_V$-matrix structure and a Bregman's divergence in Proposition \ref{prop_meso_limit}. Particularly, if there exists a positive invariant measure $\pi_{\V}$ for mesoscopic CME,  taking $\phi(p)=p\log \frac{p}{\pi_V}$, in the large number limit, the corresponding mesoscopic energy dissipation relation converges to the macroscopic energy dissipation relation \eqref{Ly} in terms of the energy landscape $\psi^{ss}(\vx)$.
We emphasis that the mesoscopic 
 energy functional  $F(p)=\sum_{\vxv}  \phi \bbs{\frac{p(\vxv)}{\pi(\vxv)}} \pi(\vxv)$ is always convex w.r.t. $p$. However, since $\phi(u)$ is convex, the nonlinear weight $\pi_{\V}(\vxv)$ in $F(p)$ drastically pick up the complicated non-convex energy landscape for chemical reactions from $\pi_{\V}(\vxv)\approx e^{-V \psi^{ss}(\vx)}$ in the large number limit. After the concentration of the measure in the large number limit, a non-convex energy landscape $\psi^{ss}$ emerges. Notice  there is no such a transition from convex functional to a nonconvex function under the RRE detailed balance assumption because the corresponding probability flux is only monomial. {\blue We point out the above emerged polynomial grouped probability flux \eqref{group_flux} and the non-convex energy landscape  are only linked to non-equilibrium in the specific context of chemical reactions. For  general equilibrium models  in statistical physics, non-convex energy landscape is common, for instance the Lagenvin dynamics with non-convex potential and Ising model of ferromagnetism.} 
 %It also inspires the definition of  grouped probability flux \eqref{group_flux} including polynomials, which distinct the equilibrium and non-equilibrium reactions.
%{\red We remark the evolutionary $\Gamma$-convergence from the mesoscopic CME to the macroscopic RRE was proved in \cite{Liero, maas2020modeling} for RRE detailed balance case. }

{\blue In terms of the mesoscopic chemical reaction jumping process, the proper mathematical definition for detailed balance is  there exists a positive  invariant measure $\pi_{\V}$ satisfying the Markov chain detailed balance \eqref{master_db_n}, which naturally includes the above grouped probability fluxed \eqref{group_flux}.  In the large number limit $V\to +\8$, this Markov chain detailed balance gives raise to a  symmetry for the macroscopic Hamiltonian 
\begin{equation} \label{newS}
H(\vp, \vx) = H(\nabla\psi^{ss}(\vx)-\vp, \vx), \quad \forall \vx, \vp;
\end{equation}
see Proposition \ref{prop_mcdb}.
  Taking $\vp=\vec{0}$, we know $\psi^{ss}(\vx)$ is the stationary solution to HJE \eqref{HJE2psi}. Applying \eqref{1.3} and \eqref{HJE2psi}, we formally have $\psi^{ss}(\vx)=-\lim_{V\to+\8} \frac{\log \pi(\vxv)}{V}$. Rigorously, an upper semicontinuous viscosity solution $\psi^{ss}$  to the stationary HJE was constructed from $\pi(\vxv)$ in \cite{GL_vis} in the Barron-Jensen’s sense \cite{Barron_Jensen_1990}. }
  
   The first consequence of this symmetric Hamiltonian \eqref{newS} is the RRE becomes an Onsager's type strong gradient flow in terms of $\psi^{ss}(\vx)$. That is to say the conservative part $W(\vx)$ vanishes in the RRE decomposition \eqref{odeDD}; see Proposition \ref{prop_sgf}. 

The second consequence of this symmetric Hamiltonian is the time reversal symmetry in terms of the Lagrangian upto a null Lagrangian  
 \begin{equation}\label{LLL}
L(\vs,\vx) - L(-\vs, \vx) =   \vs \cdot \nabla \psi^{ss}(\vx), \quad \forall \vx, \vs.
\end{equation}
{\blue This symmetric relation was first dated back to \textsc{Morpurgo} et.al. in \cite{morpurgo1954time}   for Hamiltonian dynamics in classical mechanics with a   quadratic Hamiltonian. The quadratic form Hamiltonian  $H(\vp,\vx)=\vp\cdot (\vp-\nabla U)$ from the WKB expansion of the Langevin dynamics is also symmetric w.r.t $\vp=\frac12\nabla U$, so the classical Freidlin-Wentzell theory \cite{Freidlin_Wentzell_2012} shows  the most probable path  connecting two steady states $\xa,\xb$ of $U$ (assumed to be double well with Morse index $1$) is piesewisely given by an 'uphill' least action curve which starts from $\xa$, passes through a saddle point $\xc$  and then matches with a 'downhill' least action curve from $\xc$ to $\xb$.  The 'uphill'  least action curve with nonzero action is exactly the time reversal of the zero-cost least action curve from $\xc$ to $\xa$. This symmetric relation \eqref{LLL} was   systematically studied in \cite{Mielke_Renger_Peletier_2014}, which established the relation between generalized gradient flow and the large deviation principle. The symmetry in the  Hamiltonian \eqref{newS}  was also  used    in \cite{bertini2002macroscopic} for the macroscopic fluctuation theory; see recent developments in \cite{Kraaij_Lazarescu_Maes_Peletier_2020}  for the fluctuation symmetry and the associated GENERIC formalism.}
 In Proposition \ref{thm2}, with the symmetric Hamiltonian condition \eqref{newS}, (i) the 'uphill' least action curve connecting    a stable steady state and a saddle point is still the time reversal $\vxr(t)$ of the 'downhill' curve $\vx(t)$ while the corresponding momentum is reversed with an additional control force $\nabla\psi^{ss}$; (ii) the difference of Lagrangians between the forward and reversed curve   is a null Lagrangian \eqref{LLL}, so   the difference of the action cost between the path and the time revered path depends only on the starting/end positions; (iii) the steady solution $\psi^{ss}(\vx)$ to HJE \eqref{HJE2psi} defines the energy landscape for the chemical reaction and the path affinity is given by the difference between the values of $\psi^{ss}$ at the starting/end positions
\begin{equation}
\ac(\vxr(\cdot))  - \ac(\vx(\cdot)) = \psi^{ss}(\vxr_T) - \psi^{ss} (\vxr_0 ).
\end{equation}
 The globally defined energy landscape $\psi^{ss}$ coincides with the quasipotential \cite{Freidlin_Wentzell_2012} upto a constant if the least action curve stays within a stable basin of attraction of a steady state.

The third consequence is we can use the symmetric Hamiltonian to study a class of non-equilibrium enzyme reactions. 
%Only with catalysis of enzymes, the biochemical reactions can happen in a living cell within a reasonable time interval.  
  {\blue Notice the symmetric Hamiltonian condition brought by the mathematical definition of Markov chain detailed balance \eqref{master_db_n} is more general than the constrained chemical version detailed balance condition \eqref{DB}. Although the Markov chain detailed balance is a basic mathematical concept, it includes the grouped probability flux representing a nonzero steady flux in each reaction channel.} The resulting symmetric Hamiltonian can be used to describe a   class of non-equilibrium reactions including enzyme catalyzed reactions.    In Section \ref{subsub_sch}, a simplified Schl\"ogl catalysis reaction was studied in detail, where  three   features for non-equilibrium chemical reactions: multiple steady states, nonzero steady state fluxes and positive entropy production rates at non-equilibrium steady states (NESS) are shown.

The stationary solution $\psi^{ss}(\vx)$ to HJE serves as the energy landscape of chemical reactions, facilitates the conservative-dissipative decomposition for RRE, and also  determines both the energy barrier and thermodynamics of chemical reactions.  For a detailed/complex balanced RRE, we simply have 
 $\psi^{ss}(\vx)=\KL(\vx||\peq)$; see Lemma \ref{lem_cDB}. For general chemical reactions, we first discuss  viscosity solutions to stationary HJE  by the dynamic programming method \cite[Theorem 2.41]{Tran21}. This is equivalent to an  optimal control interpretation in an undefined time horizon; see Section \ref{subsec_oc}.   By Maupertuis's principle for an undefined time horizon (see \eqref{Ebb}), we regard
 $\vp$ as a control variable, then in terms of the Hamiltonian,  the most probable path is solved by a constrained  optimal control problem (see \eqref{ocn})
 \begin{equation} 
 \begin{aligned}
 &v(\vy; \, \xa, c)=\inf_{T,\vp} \int_0^T \bbs{ \vp \cdot \nabla_p H(\vp,\vx) - H(\vp,\vx) + c  }\ud t,\\
 & \text{s.t.}  \,\, \dot{\vx} = \nabla_p H(\vp,\vx),\,\, t\in(0,T),\quad \vx_0=\xa,\,\, \vx_T =\vy.
 \end{aligned}
 \end{equation}
 Here and afterwards, we use notation $\nabla_p H$ as  the vector $\{\pt_{p_i}H\}_{i=1:N}.$
 The critical energy level $c$ is zero for the Hamiltonian in chemical reactions.   Let $\xa$ be a   steady state of RRE, then the corresponding critical ma\~n\'e potential $v(\vy; \, \xa, 0)$ gives a viscosity solution to the steady HJE \cite{ishii2020vanishing}. However, to obtain a unique viscosity solution given by the energy landscape $\psi^{ss}$ in a chemical reaction,  some notation of selection principle in weak KAM solutions needs to be imposed \cite{Gao_2022}.

 In general,  a standard diffusion approximation can be obtained   via the Kramers-Moyal expansion for the CME, which is equivalent to the quadratic approximations near $\vp=0$
 for the Hamiltonian; see Section \ref{subsec_diffusion}. However,
 this diffusion approximation only valid for a transition   near the   'downhill' solution to the RRE. The 'uphill' transition path starting from a stable steady state ending at a saddle point is  a rare transition in the large deviation regime and the energy barrier can not be computed by the above diffusion approximation. We  refer to \cite{Doering_Sargsyan_Sander_2005} for quantified analysis of the failure of the diffusion approximation via the Kramers-Moyal expansion (a.k.a `system size expansion' by \textsc{van Kampen} \cite{van1992stochastic}). Based on the strong gradient formulation \eqref{strong_GF}, another    drift-diffusion approximation \eqref{528} is proposed as a good quadratic approximation near not only the 'downhill' solution to the macroscopic RRE but also the 'uphill' least action curve. This  diffusion approximation shares the same energy landscape and same symmetric Hamiltonian structure w.r.t.  $\nabla\psi^{ss}$ and satisfies a fluctuation-dissipation relation with an invariant measure $\pi=e^{-V\psi^{ss}}$.

\subsection*{State of the art}
The WKB expansion  is a classical and powerful tool to understand the exponential asymptotics that quantifies the fluctuations in  many physical problems; see \textsc{Kubo} et.al \cite{Kubo73} for a WKB expansion of a general stochastic process and see Doi-Peliti formalism \cite{Doi_1976, Peliti_1985}.  WKB analysis also initials   physical studies of the  large deviation(fluctuation) behaviors for stochastic models of chemical reactions  from a Hamiltonian  viewpoint, c.f. \cite{ Hu87, Dykman_Mori_Ross_Hunt_1994, SWbook95, Spohn99, Assaf_Meerson_2017}.
Particularly,  in the large number limit of chemical reactions modeled by the CME, \textsc{Dykman} et.al. \cite{Dykman_Mori_Ross_Hunt_1994} first derived   HJE \eqref{HJE2psi} with the associated Hamiltonian $H$ and studied the symmetry of the Hamiltonian   in a  detailed balanced chemical reaction system. We also refer to a recent review article \cite{Assaf_Meerson_2017} using WKB approximations to study  various large deviation behaviors such as population extinction/fixation, genetic  switches and biological invasions.

 The concept and the exponential asymptotics for    the reaction rate in terms of the activation energy (energy barrier) for transitions between two states in a chemical reaction was pioneered by Arrhenius 1889 while the celebrated work by Kramers explicitly estimated it for a Langevin dynamics. 
At the mathematical analysis level, 
the large deviation principle with the associated Lagrangian/Hamiltonian  for general  stochastic processes and     the transition path  (the most probable path) connecting two stable states were   pioneered by \textsc{Freidlin and Wentzell} in late 60s, c.f. \cite{Freidlin_Wentzell_2012}.  The central idea of the Freidlin-Wentzell theory is that the steady solution $\psi^{ss}(\vs)$ to the HJE defines a quasipotential which quantifies   the maximum probability or the energy barrier for a transition, i.e., an exit problem in the basin of attraction.  We also refer to \cite{fleming1983optimal, fleming06, feng2006large, Kraaij_2016, Kraaij_2020} for the optimal control and nonlinear semigroup viewpoint, which connect   least action problems with   HJEs.
For chemical reactions with RRE detailed balance \eqref{DB}, the quasipotential is given by $\psi^{ss}(\vx)=\KL(\vx||\peq)$ \cite{Dykman_Mori_Ross_Hunt_1994}, while for general large number process including   non-equilibrium   dynamics, \cite{Dembo18} proved the large deviation principle for $\cv$ with the associated good rate function. In \cite{GL_vis}, the large deviation principle at single times was proved via the convergence from the Varadhan's nonlinear semigroup to the  Lax-Oleinik representation of the viscosity solution to HJE. Moreover, an upper semicontinuous viscosity solution in the Barron-Jensen’s sense \cite{Barron_Jensen_1990} to the stationary HJE was also obtained in \cite{GL_vis} by using a positive detailed balanced invariant measure to $\cv$.

With the  RRE detailed/complex balance condition,   characterization and uniqueness of all  steady states for the macroscopic RRE was proved in \cite{Horn72, feinberg1972chemical}. On the contrary,  thermodynamic relations,  dissipation structures and  computations for transition paths in non-equilibrium stochastic  dynamics are challenging problems due to coexistent steady states and nonzero steady fluxes   sustained by environment, whose studies  were pioneered by  \textsc{Prigogine} \cite{prigogine1967introduction} from the Brussels School.   We refer to \cite{Ruelle_2003, qian2006open, kondepudi2014modern, Rao_Esposito_2016, QianGe17, qian_book}  and the references therein for thermodynamics relations, particularly the adiabatic/nonadiabatic decomposition for the nonzero entropy production rate  in non-equilibrium biochemical reactions.
  In  \cite{Lazarescu_Cossetto_Falasco_Esposito_2019}, \textsc{Lazarescu} et.al. used  a biased Hamiltonian $H$ for chemical reaction based on time-averaged observations  to study the first order phase transitions,  particularly for metastable models in an open system with non-equilibrium steady states.  However, it is not clear whether the biased Hamiltonian provides the most probable path (the least action path for the original Hamiltonian). Using a linear response relation with a susceptibility $\chi(\rho)$ between the current and the external field generating the fluctuation,   a comprehensive review by \textsc{Bertin} et.al \cite{Bertini15} discussed the macroscopic fluctuation theory including  the time reversal, symmetry of Hamiltonians, fluctuation theorems at a macroscopic scale for various physical models.  {\blue The macroscopic fluctuation theory was first developed by \textsc{Bertini}, et.al \cite{bertini2002macroscopic}; see further mathematical analysis and variational structure including the density-flux pair large deviation principle in \cite{Renger_2018, Patterson_Renger_2019, Patterson_Renger_Sharma_2021}.} Without the quadratic approximation of the Hamiltonian, the calculations of transition paths,   and the symmetry for fully nonlinear  Hamiltonians in non-equilibrium reactions were not discussed in \cite{bertini2002macroscopic, Bertini15}. Indeed, there were many studies for the failure in computing the correct energy barrier of transition paths using a simple diffusion approximation from the Kramers-Moyal expansion of the CME; c.f.  \textsc{Doering} et.al  \cite{Doering_Sargsyan_Sander_2005}  for  the extinction problem in a birth-death stochastic population model. 

Under the RRE detailed balance assumption,  the macroscopic RRE has several gradient flow structures in terms of free energy $\KL(\vx||\peq)$ \cite{Onsager,  maas2020modeling}.   Particularly, a De Giorgi type generalized gradient flow structure brought by the symmetry in the Hamiltonian   is closely related to the good rate function in the large deviation principle; see systematical  studies in \cite{Mielke_Renger_Peletier_2014}.  Recently, \cite{Liero, maas2020modeling}  recovered the macroscopic RRE  for chemical reactions  via the evolutionary $\Gamma$-convergence techniques in \cite{Serfaty04, mielke2016evolutionary} in the gradient flow regime.  {\blue The symmetric Hamiltonian was also  used in \textsc{Kraaij} et.al. \cite{Kraaij_Lazarescu_Maes_Peletier_2020}   to study the fluctuation symmetry. This symmetry criteria  reduces a pre-GENERIC system to a GENERIC formalism \cite{Kraaij_Lazarescu_Maes_Peletier_2020}. In general, the energetic decomposition for a dynamics is not unique and has   different gradient flow structures with associated fluctuation estimates; c.f.,  \cite{Peletier_Redig_Vafayi_2014}.}

 The remaining part of this  paper is organized as follows. 
 In Section \ref{sec2}, we provide   preliminaries for the RRE, WKB expansion and properties for the Hamiltonian and the Lagarangian. In Section \ref{sec3}, we study    dynamic solutions,  steady solution $\psi^{ss}(\vx)$ to the HJE. Using the stationary solution, we propose a conservative-dissipative decomposition for general non-equilibrium RRE (see Section \ref{sec_3decom}) and also give a decomposition for its thermodynamics (see Section \ref{subsec_thermo}). The associated energy dissipation laws at both mesoscopic and macroscopic level with the passage from one to another is given in Section \ref{sec_meso2}.   In Section \ref{sec_reversal}, we use a symmetric Hamiltonian to study a  class of non-equilibrium enzyme reactions,  which yields (i) an Onsager-type strong  form of gradient flow and (ii)  a modified  time reversed curve serves as the   transition paths between coexistent steady states. Bistable Schl\"ogl example is discussed in Section \ref{subsub_sch}.   In Section  \ref{sec_control}, we clarify the existence of the stationary solution to HJE via an optimal control representation in an undefined time horizon and give a   diffusion approximation for transition path computations that satisfies the fluctuation-dissipation relation and the same symmetric Hamiltonian.  
 Pedagogical derivations for   CME, the generator and the mean field limit RRE after
  including   `no reaction' boundary condition are given in Appendix.

%\begin{enumerate}[(I)]
%\item The solution to RRE \eqref{odex} is a least action with action functional $\ac(\vx(\cdot))=\int_0^T L(\dot{\vx},\vx)\ud t$ equals zero;
%\item Detailed balance condition for a steady state $\peq$ to RRE \eqref{odex} is equivalent to the even-symmetry for Hamiltonian $H$, i.e., 
%$
%H(\vp, \vx) = H(\log \frac{\vx}{\peq} - \vp, \vx);
%$
%\item Under the detailed balance condition, the time revered solution to Kurtz RRE (and in general to a Hamiltonian dynamics)  is  still a least action curve which switches two ending states;
%the time reversed solution is also an optimal controlled curve minimizing the action functional $\ac(\vx(\cdot))$ with a nonzero  action cost;
%\item For a chemical reaction with a detailed balance steady state $\peq$, the most probable path, in the sense of large deviation principle, can NOT be given by piecewise least action curve;
%\item RRE \eqref{odex} has several gradient flow formulations w.r.t the relative entropy $\KL(\vx|\peq|)$ and among them, the De Giorgi $(\psi,\psi^*)$-type generalized gradient flow is the zero level-set of the rate function $L$.
%\end{enumerate}

\section{Preliminaries: macroscopic RRE, WKB expansion and large deviation}\label{sec2}
As a preparation for our main results, in this section, we review some terminologies  for the large number limiting RRE \eqref{odex} and collect some preliminary lemmas for existence, uniqueness, characterization of steady states in a  detailed/complex balanced RRE system. The associated Hamiltonian $H(\vp,\vx)$ and HJE from the WKB expansion are also revisited. Moreover, the convex conjugate $L(\vs,\vx)$ of $H(\vp,\vx)$ gives the rate function in the large deviation principle for the large number process, which allows us to study the transition path for a non-equilibrium system in later sections.  Most of the results in this section was known while we provide brief proofs for completeness.

\subsection{Terminologies for the macroscopic RRE and RRE detailed/complex balance conditions}
Recall the forward and backward fluxes $\Phi_j^\pm$ satisfying LMA \eqref{lma} and RRE \eqref{odex}.
Using the $N\times M$ stoichiometric matrix $\nu^T$ and the reaction rate vector 
\begin{equation}
\vr(\vx) = \bbs{ r_j(\vx)}_{j=1:M} = (\Phi^+_j(\vec{x}) - \Phi^-_j(\vec{x}))_{j=1:M},
\end{equation}
we represent RRE \eqref{odex} in a matrix form
\begin{equation}\label{RRR}
\frac{\ud}{\ud t} \vx = \nu^T \vr =: \vR(\vx),
\end{equation}
where $\vR(\vx)$ is called the production rate.
Denote the range of matrix $\nu^T$ as $\ran(\nu^T)$, i.e., the span of the column vectors $\{\vec{\nu}_j\}$ of $\nu^T$. Then we know the production rate 
\begin{equation}
\dot{\vx}(t) \in \ran (\nu^T)\subset \bR^N.
\end{equation}
Motivated by this, we define the subspace
$
G=\{ \vx \in \bR^N; \, \vx
\in \ran(\nu^T) \}
$ which is known as the stoichiometric space.
Recall \eqref{mb}, i.e.,
$
\nu \vec{m}=\vec{0},
$
 which implies the conservation of total mass, i.e., $\frac{\ud}{\ud t} \bbs{ \vec{m} \cdot \vx } = \vec{m}\cdot \nu^T \vr=0$.
Therefore the  Wegscheider matrix $\nu$ always has a nonzero kernel, i.e., 
$
\dim \bbs{\kk(\nu)} \geq 1.
$

We have the following lemma on existence and uniqueness of dynamic solution to \eqref{odex}.
\begin{lem}\label{lem:ode}
Assume $\nu\in \bR^{M\times N}$ is the Wegscheider matrix satisfying \eqref{mb}. Consider RRE \eqref{odex} with flux $\Phi_j^\pm$ satisfying \eqref{lma}. We have
\begin{enumerate}[(i)]
\item The region $R^N_+:=\{\vx \in \bR^N; x_i \geq 0\}$ is an invariant region;
\item For any initial data $\vx_0\geq 0$, there exists a unique global-in-time bounded solution to \eqref{odex} satisfying
\begin{equation}
\frac{\ud }{\ud t} \vx(t) \cdot \vec{m} =0.
\end{equation}
\end{enumerate}
\end{lem}
The statement (i) can be directly verified by proving $\frac{\ud}{\ud t} x_i \geq 0$ at any $x_i=0$ using case by case arguments. The statement (ii) is a consequence of \eqref{mb} and the standard ODE theory.

\subsubsection{Detailed balance  and complex  balance for the macroscopic RRE}
Recall the macroscopic RRE \eqref{odex} and the RRE detailed balance condition \eqref{DB} is equivalent to
\begin{equation}
\log k_j^+ -  \log k_j^- =   \vec{\nu}_{j} \cdot \log \peq
\end{equation}
due to  LMA \eqref{lma}.

Denote the complex space as the collection of  distinct reaction vectors
$
\mathcal{C}:= \bbs{\vec{\nu}_j^\pm}_{j=1:M}.
$
Then the complex balance condition means for each complex $\vec{\eta}\in \mathcal{C}$, all the reactant contributions in the flux equals all the product contributions in flux. Precisely, a strictly positive (componentwisely) state $\peq_c>0$ is called complex balanced \cite{Horn72} if
\begin{equation}\label{cDB}
\sum_{j, \vec{\nu}_j^+ =\vec{\eta}} \bbs{\Phi_j^-(\peq_c) - \Phi_j^+(\peq_c) }+ \sum_{j, \vec{\nu}_j^- =\vec{\eta}} \bbs{ \Phi_j^+(\peq_c)  - \Phi_j^-(\peq_c) }   = 0.
\end{equation}
One can directly verify state $\peq_c>0$ satisfying \eqref{cDB} is a steady state to RRE \eqref{odex}. Indeed, at $\peq_c$, recast the RHS of \eqref{odex} as flux difference
\begin{equation}\label{KF}
\sum_j \vec{\nu}_j  \bbs{\Phi^+_j(\peq_c) - \Phi^-_j(\peq_c)} =  \sum_j  \vec{\nu}^+_j  \bbs{\Phi^-_j(\peq_c) - \Phi^+_j(\peq_c) }   +  \sum_j  \vec{\nu}^-_j  \bbs{\Phi^+_j(\peq_c) - \Phi^-_j(\peq_c) }.
\end{equation}
The first term in the summation represents that for the reactant (aka substrate) complex $\vec{\nu}^+_j$  in the $j$th-forward reaction,  the net flux   is  $\Phi^-_j(\peq_c) - \Phi^+_j(\peq_c)$. So we can re-classify this summation w.r.t distinct reactant complex  $\vec{\nu}^+_j=\vec{\eta}, \vec{\eta}\in \mathcal{C}$
\begin{equation} 
\sum_j \vec{\nu}^+_j  \bbs{\Phi^-_j(\peq_c) - \Phi^+_j(\peq_c)} =  \sum_{ \vec{\eta}\in \mathcal{C}} \vec{\eta}   \sum_{j: \vec{\nu}^+_j=\vec{\eta}}  \bbs{\Phi^-_j(\peq_c) - \Phi^+_j(\peq_c) }.
\end{equation}
 Similarly, the second term in the summation represents that the reactant complex $\vec{\nu}^-_j$ in the $j$th-backward reaction,  the net flux is  $\Phi^+_j(\peq_c) - \Phi^-_j(\peq_c)$. Therefore we can choose to re-classify this summation w.r.t distinct reactant complex  $\vec{\nu}^-_j=\vec{\eta}, \vec{\eta}\in \mathcal{C}$
 \begin{equation} 
\sum_j \vec{\nu}^-_j  \bbs{\Phi^+_j(\peq_c) - \Phi^-_j(\peq_c)} =  \sum_{ \vec{\eta}\in \mathcal{C}} \vec{\eta}   \sum_{ j:  \vec{\nu}^-_j =\vec{\eta}}  \bbs{\Phi^+_j(\peq_c) - \Phi^-_j(\peq_c) }.
\end{equation} 
Combining the above two ways of rearrangements for the 
 summation in chemical channel $j$,  \eqref{KF} becomes
\begin{equation}\label{Kflux}
\sum_j \vec{\nu}_j  \bbs{\Phi^+_j(\peq_c) - \Phi^-_j(\peq_c)} =  \sum_{ \vec{\eta}\in \mathcal{C}} \vec{\eta} \bbs{  \sum_{j: \vec{\nu}^+_j=\vec{\eta}}  \bbs{\Phi^-_j(\peq_c) - \Phi^+_j(\peq_c) }   +  \sum_{ j:  \vec{\nu}^-_j =\vec{\eta}}  \bbs{\Phi^+_j(\peq_c) - \Phi^-_j(\peq_c) }  }=\vec{0},
\end{equation} 
where we used the complex  balance condition \eqref{cDB}.

\subsection{Characterization of RRE steady state for the detailed/complex balance case}

Now we investigate all the steady states of RRE \eqref{odex}, i.e., 
\begin{equation}
S_e:=\{\pe \in \bR^N_+;\, \vec{R}(\pe)= \sum_j \vec{\nu}_j \bbs{\Phi^+_j(\pe) - \Phi^-_j(\pe) } = \vec{0}\}.
\end{equation}

First, we show  uniqueness of positive steady states for the detailed balanced RRE.   
From \eqref{lma}, we have
 \begin{equation}\label{complexIND}
   \vec{\nu}_j \cdot \log \frac{\vx }{\peq}  = \log \bbs{ \prod_{i=1}^{N} \bbs{\frac{ x_i }{ \xss_i }  }^{\nu_{ji}}  }= \log \bbs{  \frac{\Phi_j^-(\vx)}{\Phi_j^+(\vx)} \frac{\Phi_j^+(\peq)}{\Phi_j^-(\peq)}}.
 \end{equation} 
If $\pe >0$, then from \eqref{complexIND}, we have
\begin{equation}
0 = \log \frac{\pe}{\peq} \cdot \vec{R}(\pe) =   \sum_j \bbs{\Phi^+_j (\pe)- \Phi^-_j(\pe)}  \log \bbs{\frac{\Phi^-_j(\pe)}{\Phi^+_j(\pe)}\frac{\Phi^+_j(\peq)}{\Phi^-_j(\peq)}}.
\end{equation}
Under RRE detailed balance condition \eqref{DB}, the above equation implies $\Phi^+_j (\pe)= \Phi^-_j(\pe)$ and thus $\pe$ also satisfies RRE detailed balance. Notice \eqref{lma} and \eqref{DB}  implies identity
\begin{equation}\label{tempPhi}
 \vec{\nu}_j \cdot \log \frac{\vx }{\peq}=\log \bbs{  \frac{\Phi_j^-(\vx)}{\Phi_j^+(\vx)}}.
\end{equation}
We know
$
\log \frac{\pe}{\peq} \in \kk(\nu).
$
Given the stoichiometric space $G$ and $\vq \in \kk(\nu)$, $\vq+G$ is called one stoichiometric compatibility class. Then it is easy to verify that if $\pe$ and $\peq$ are in the same stoichiometric compatibility class, then $\pe=\peq$. Indeed, from $\log \frac{\pe}{\peq} \in \kk(\nu)$ and $\pe-\peq \in G$, we know 
\begin{equation}
\log \frac{\pe}{\peq} \cdot \bbs{\pe -\peq} =0,
\end{equation}
which implies $\pe = \peq$.

This uniqueness of steady states in one stoichiometric compatibility class still holds for the complex balanced system, with a slight modification of the above proof. We conclude the
 following  well-known result on the uniqueness of steady state; c.f., \cite[Theorem 6A]{Horn72}, \cite[Theorem 3.5]{Kurtz15}.

\begin{lem}\label{lem:equi}
Assume there exists a strictly positive steady state $\peq_1$ satisfying complex  balance \eqref{cDB}.
Then for any $\vq\in \kk(\nu)$, there exists a unique steady states $\peq_*$ in the space $ \{\vx \in \vq+G;\,  \vx>0\}$. Moreover, $\peq_*$  
 satisfies complex balance condition \eqref{cDB}, and  is  characterized by
\begin{equation}\label{qq_equi}
(\peq_*)_i =( \peq_1)_i e^{q_i}>0.
\end{equation}
\end{lem}
As a consequence, if $\peq_1$ satisfies RRE detailed balance \eqref{DB}, thus it also satisfies \eqref{cDB}. So \eqref{qq_equi} still holds and this unique steady state $\peq_*$ in    the space $ \{\vx \in \vq+G;\,  \vx>0\}$ is  RRE detailed balanced. 
  For both the detailed/complex balanced RRE system, 
 $\peq$  constructs a Lyapunov function for \eqref{odex}, known as the relative entropy 
$
\KL(\vec{x}|| \peq) = \sum_i \bbs{  x_i \ln \frac{x_i}{x^{\xs}_i} - x_i + \xss_i}.
$
  Since $\KL(\vx||\peq)$ is strictly convex for $ \{\vx \in \vq+G;\,  \vx>0\}$, so one also have  local stability of the RRE detailed/complex balanced steady state $\peq_*$.   To obtain   global stability of $\peq_*$, a necessary condition (see \cite{sontag2001structure}) is that there shall be no  equilibrium on the boundary of $\bR^N_+$ for the positive stoichiometric compatibility class, i.e., $\{\vx \in \vq+G;\,  \vx>0\}$. We refer to \cite{anderson2008global} for more detailed conditions to obtain global stability.

\subsubsection{Deficiency zero theorem}\label{subsec_de}
The complex balance condition is an important property for balance between the product complex and the reactant complex. It also motivates a more important index theorem based only on the graph structure of the reaction networks. Recall the complex space 
$
\mathcal{C}= \{\vec{\nu}_j^\pm\}_{j=1:M}
$ 
and species ${X}=\{X_i\}_{i=1:N}$.
A reaction network, denoted as $({X}, \mathcal{C}, \mathcal{R})$, is a  directed graph with nodes given by the complexes $\mathcal{C}$ and directed edges given by reactions $\mathcal{R}=\{\vec{\nu}_j^+ \to \vec{\nu}_j^-\}$. Each connected  subgraph (regarded as undirected subgraph)  is called a linkage class and denote the total number of the linkage classes of the reaction graph as $\ell$.  Denote the total number of distinct complex as $n_c$ and denote the rank of $\nu$ as $s$. Then the deficiency of the reaction network is $\delta:= n_c - \ell-s\geq 0.$ In \cite{feinberg1972chemical}, \textsc{Feinberg} proved that a deficiency zero network, i.e., $\delta=0$ is equivalent condition for that  the equilibrium for \eqref{odex} is complex balanced. Therefore, the equilibrium of the RRE \eqref{odex} can be characterized using the deficiency zero theorem, which relies only on the network structure of $({X}, \mathcal{C}, \mathcal{R})$. More precisely,  we call the reaction network is weakly reversible if for any path connecting from complex $\mathcal{C}_i$ to complex $\mathcal{C}_j$, one can always find a path connecting from complex $\mathcal{C}_j$ to complex $\mathcal{C}_i$. Then the deficiency zero theorem proved by \textsc{Horn, Feinberg} \cite{feinberg1972chemical, Horn72} states that if a chemical reaction network with LMA satisfies (i) $\delta=0$ and (ii) weakly reversibility, then there is a unique positive steady state in each  stoichiometric compatibility class; see also \cite[Theorem 7.1.1]{Feinberg_2019}.

\subsection{WKB expansion and Hamilton-Jacobi equation}\label{subsec_wkb}
In this section, we  use the WKB analysis of the CME for $p(\vxv,t)$  to study the exponential asymptotic behavior. We will investigate some good properties of the resulting HJE and the associated Hamiltonian $H(\vp,\vx)$ defined in \eqref{H}.
Recall 
the large number process $\cv (t)$ in \eqref{Csde}, which is also denoted as $C_t$ for simplicity.
For  fixed $V$, recall $Q^*_{\V}$ defined in \eqref{rp_eq}. Then for any continuous test function $f(\vxv)$, we have
\begin{equation} \label{generator}
\begin{aligned}
 \frac{\ud}{\ud t}\sum_{\vxv}f(\vxv) p(\vxv, t)   
= \sum_{\vxv} (Q_{\V} f)(\vxv)   p(\vxv,t).
\end{aligned}
\end{equation}
{\blue Here $Q_{\V}$ is the duality of $Q^*_{\V}$, see explicit definition in  \eqref{p_eq} after including `no reaction' boundary condition.}

Denote 
\begin{equation}
w(\vxv, t) = \bE^{\vxv}\bbs{f(C_t)},
\end{equation}
then $w(\vxv,t)$ satisfies the backward equation
\begin{equation}
\pt_t w = Q_{\V} w, \quad w(\vxv, 0)=f(\vxv).
\end{equation}
{\blue We refer to \cite{GL_vis} for  the well-posedness of the backward equation after including `no reaction' boundary condition.}

Assume there exists a smooth enough function $u(\vx,t)$ such that at $\vx=\vx_{\V}$, we have WKB reformulation
\begin{equation}
w(\vxv, t) = e^{V u(\vxv,t)}.
\end{equation} 
We obtain
\begin{equation}
\pt_t u(\vxv,t) = \frac{1}{V} e^{-V u(\vxv,t)} Q_{\V} e^{V u(\vxv,t)} =: \frac{1}{V} H_{V}(V u), \quad u(\vxv,0) = \frac{1}{V}\log f(\vxv).
\end{equation}
In summary, 
\begin{equation}
u(\vxv,t) =\frac{1}{V} \log w(\vxv,t) =  \frac{1}{V} \log \bE^{\vxv}\bbs{f(C_t)} = \frac{1}{V} \log \bE^{\vxv}\bbs{e^{V u_0(C_t)}} =:\bbs{S_t u_0}(\vxv)
\end{equation}
is the so-called Varadhan's nonlinear semigroup \cite{Varadhan_1966, feng2006large} for process $C_t$.

For any $\vx\in \bR^N_+$, let $\vx_V=\frac{\vec{n}}{V}\to \vx$ as $V\to +\8$. Then the after WKB reformulation at $\vxv\geq 0$ gives
\begin{align*}
Q_{\V} e^{V u(\vxv, t)} =& V \sum_{j=1, \vxv+\frac{\vec{\nu}_j}{V} \geq 0}^M    \Phi^+_j(\vxv)\bbs{ e^{Vu(\vxv+\frac{\vec{\nu_{j}}}{V})} - e^{V u(\vxv)}}    +  V \sum_{j=1,  \vxv-\frac{\vec{\nu}_j}{V} \geq 0}^M \Phi_j^-(\vxv)\bbs{ e^{V u(\vxv-\frac{\vec{\nu_{j}}}{V})} - e^{Vu(\vxv)}}
\end{align*}
For $\vxv\notin\bR^N_+ $, one can define a zero extension for $\tilde{\Phi}_j^\pm(\vxv)$; see \cite{GL_vis}.
Using Taylor's expansion w.r.t $\frac{\vec{\nu}_j}{V}$, we obtain HJE for $u$
\begin{equation}\label{tempHJE}
\pt_t u(\vx,t) =  \sum_{j=1}^M \bbs{ \Phi^+_j(\vec{x})\bbs{e^{\vec{\nu}_j\cdot \nabla u(\vx, t)}-1 }     +   \Phi_j^-(\vec{x})\bbs{ e^{-\vec{\nu}_j\cdot \nabla u(\vx, t)}-1 }  } .
\end{equation}
Similarly, starting from the froward equation \eqref{rp_eq}, one can obtain the HJE \eqref{HJE2psi} for $\psi(\vx,t)$.

%\begin{rem}
%From the definition of the backward solution $w$ and the forward solution $p$, we have
%\begin{equation}
%\begin{aligned}
%w(\vx, t ) = \la w_0(\vec{y}), p(\vec{y},t; \vx) \ra_y,\quad 
%p(\vec{y},t) = \la p_0(\vx), p(\vec{y},t; \vx) \ra_x. 
%\end{aligned}
%\end{equation}
%Thus using the Chapman-Kolmogorov equation, we have
%\begin{equation}
%\frac{\ud}{\ud s}  \la p(\vx,s) w(\vx,t-s)\ra_x = 0.
%\end{equation}
%Similarly, given $T>0$, the time reversal $\tilde{u}(\vx,s)=u(\vx, T-s)$ starting from $\tilde{u}(\vx,0)=\psi(\vx,0)$ satisfies
%\begin{equation}
%\tilde{u}(\vx,s)=u(\vx,T-s)=\psi(\vx,s).
%\end{equation}
%\end{rem}

%\begin{proof}
%Let $P_{ij}(t)$ be the transition probability function for the $C$-process with generator $Q$. Then from \eqref{nonS}, we have
%\begin{equation}\label{nonS}
%u(\vx, t):= \frac1V \log \bE^{\vx} \bbs{e^{V u_0(C_t)} } = \frac{1}{V} \log \bbs{ \sum_j e^{V u_j} P_{ij}(t)}.
%\end{equation}
%Then 
%\begin{align*}
%n e^{n u_i(t)} \dot{u}_i(t) = \frac{\ud}{\ud t} e^{n u_i(t)} = \sum_j e^{n f_j} \dot{P}_{ij}(t) = \sum_{j,k} Q_{ik}P_{kj}e^{n f_j} = \sum_{k} Q_{ik} e^{nu_k(t)}.
%\end{align*}
%Therefore, we conclude
%\begin{equation}
%\frac{\ud}{\ud t} u_i(t) =\frac1n \sum_k Q_{ik} e^{n u_k(t)-n u_i(t)}= \frac1n \bbs{H(n u)}_i.
%\end{equation}
%\end{proof}

\subsubsection{Properties of Hamiltonian $H$}
Recall the matrix form of the macroscopic RRE
\begin{equation}
\frac{\ud}{\ud t} \vx = \nu^T \vr =: \vR(\vx),
\end{equation}
where $\nu\in\bR^{M\times N}$ is a constant matrix. Recall  the mass conservation law of chemical reactions  \eqref{mb} and   direct decomposition \eqref{direct}, which always satisfies
\begin{equation}
\dim \bbs{\ran(\nu^T)} <N.
\end{equation}
  It motivates that for  the WKB expansion and the corresponding relations with the rate function $L$ in the large deviation principle, we will see $L$ make sense in a `more accurate' subspace $G$.

\begin{lem}\label{lem_Hdege}
Hamiltonian $H(\vp,\vx)$ in \eqref{H} is degenerate in the sense that
\begin{equation}\label{Hdege}
H(\vp, \vx) = H(\vp_1, \vx),
\end{equation}
where $\vp_1\in \ran(\nu^T)$ is the direct decomposition of $\vp$ such that
\begin{equation}\label{decom}
\vp = \vp_1 + \vp_2, \quad \vp_1\in \ran(\nu^T),\,\, \vp_2 \in \kk(\nu).
\end{equation} 
\end{lem}
\begin{proof}
From the direct decomposition \eqref{direct}, we have \eqref{decom}. Thus $  0 =\vec{\nu}_j \cdot \vp_2$, which implies \eqref{Hdege}.
\end{proof}

\begin{lem}\label{Hconvex}
$H(\vp,\vx)$ defined in \eqref{H} is strictly convex for  $\vp \in G$.
\end{lem}
\begin{proof}
We compute the Hessian of $H$ in $G\times \bR^N$.  For any $\vec{\alpha}\in G$,
\begin{align*}
 \frac{\ud^2}{\ud \eps^2}\Big|_{\eps=0} H(\vp+\eps \vec{\alpha},\vx)  = \sum_{j=1}^M \bbs{\vec{\nu}_j \cdot \vec{\alpha}}^2 \bbs{\Phi^+_j(\vec{x})e^{\vec{\nu_j} \cdot \vec{p}} + \Phi^-_j(\vec{x})e^{-\vec{\nu_j} \cdot \vec{p}} }\geq 0
\end{align*}
and the equality holds if and only if $\nu \va =\vec{0}$. Since $\va\in G=\ran(\nu^T)$, there exists a  vector $\vb \in \bR^M$ such that $\va = \nu^T \vb$. Thus the equality above holds if and only if 
\begin{equation}
0=\vb^T \nu \va= \vb^T\nu \nu^T \vb,
\end{equation}
which implies $\vec{\alpha} = \vec{0}\in G.$
\end{proof}

\subsection{The convex conjugate $L(\vs,\vx)$ gives the rate function in large deviation principle}\label{sec4.0}
Let us first introduce the convex conjugate function $L$ and   the associated action functional.
Since $H$ defined in \eqref{H} is convex w.r.t $\vp$,  we compute the convex conjugate of $H$ via the Legendre transform. For any $\vs\in \bR^N$, define
\begin{equation}\label{L}
L(\vs,\vec{x}) := \sup_{\vec{p}\in \bR^N} \bbs{ \la \vp ,  \vs \ra - H(\vp, \vx)} =  \la \vp^* ,  \vs \ra - H(\vp^*, \vx)
\end{equation}
where $\vp^*(\vs, \vx)$ solves
\begin{equation}\label{ts1}
\vs = \nabla_p H(\vp^*, \vx) = \sum_j \vec{\nu}_j \bbs{\Phi_j^+ e^{\vec{\nu}_j \cdot \vp^*} - \Phi_j^- e^{-\vec{\nu}_j \cdot \vp^*} }.
\end{equation}
  Recall here  notation  $\nabla_p H$ is a vector $\bbs{\pt_{p_i}H}_{i=1:N}$.
Thus 
\begin{equation}\label{ts2}
L(\vs, \vx) = \vs \cdot \vp^*(\vs,\vx) - H(\vp^*(\vs,\vx),\vx).
\end{equation}
Define the action functional as
\begin{equation}\label{A}
\ac(\vx(\cdot)) = \int_0^T L(\dot{\vx}(t), \vx(t)) \ud t.
\end{equation}
Then we have the following lemma.
\begin{lem}\label{lem:least}
For $L$ function defined in \eqref{L}, we know
\begin{enumerate}[(i)]
\item $L(\vs,\vx)\geq 0$ and 
\begin{equation}\label{LL}
L(\vs, \vx) = \left\{ \begin{array}{cc}
\max_{\vp \in G} \{ \vs \cdot \vp - H(\vp, \vx) \}, & \vs\in G\\
+\8, & \vs \notin G;
\end{array}  \right.
\end{equation}
moreover, $L$ is strictly convex in $G$;
\item For  the action functional $\ac(\vx(\cdot))$ in \eqref{A}
the least action $\vx(t)$ satisfies
the Euler-Lagrange equation 
\begin{equation}\label{EL}
\frac{\ud}{\ud t} \bbs{\frac{\pt L}{\pt \dot{\vx}}(\dot{\vx}(t), \vx(t))  }= \frac{\pt L}{\pt \vx}(\dot{\vx}(t), \vx(t)),
\end{equation}
which is equivalent to the Hamiltonian dynamics with $H$ defined in \eqref{H}
\begin{equation}\label{HM}
\frac{\ud}{\ud t} \vx = \nabla_p H(\vp, \vx), \quad \frac{\ud}{\ud t} \vp = - \nabla_x H(\vp , \vx);
\end{equation}
\item $\vx(t)$ is the solution to RRE \eqref{odex} if and only if $\ac(\vx(\cdot))=0$.
\end{enumerate} 
\end{lem}
\begin{proof}
(i) First, from \eqref{Hf}, $H(\vec{0},\vx)\equiv 0$ thus we know $L(\vs,\vx)\geq 0$. 

Second, from Lemma \ref{lem_Hdege}, we know for $\vs\in \bR^N$,
\begin{equation}
\begin{aligned}
L(\vs,\vx)=&\sup_{\vp\in \bR^N} \bbs{\la (\vp, \vs \ra - H(\vp, \vx)}\\
=& \sup_{\vp\in \bR^N}   \bbs{\la \vp_1, \vs \ra+ \la \vp_2, \vs \ra - H(\vp_1, \vx)},
\end{aligned}
\end{equation}
where $\vp_1\in G$ and $\vp_2\in \kk(\nu)$ are direct decomposition of $\vp$. Therefore, for $\vs\notin G$,
\begin{equation}\label{sG}
L(\vs,\vx)\geq \sup_{\vp_2\in \kk(\nu), \vp_1=\vec{0}}   \bbs{\la \vp_1, \vs \ra+ \la \vp_2, \vs \ra - H(\vp_1, \vx)} = \sup_{\vp_2\in \kk(\nu)} \la \vp_2, \vs \ra =+\8.
\end{equation}
On the other hand, for $\vs\in G$,
\begin{equation}\label{chL}
L(\vs,\vx) = \sup_{\vp_1\in G}   \bbs{\la \vp_1, \vs \ra - H(\vp_1, \vx)}.
\end{equation}
 From the definition of $H$, we know
$H$ has a lower bound and is exponentially coercive. Indeed,
\begin{equation}
\lim_{|\vp|\to+\8} H(\vp,\vx) \geq \lim_{|\vp|\to+\8} \sum_j\bbs{ \min(\Phi^+_j, \Phi_j^-) e^{(\vec{\nu}_j \cdot \hat{\vp})|\vp|} - \Phi_j^+(\vx)-\Phi_j^-(\vx)} = +\8.
\end{equation}
Therefore, the sup in \eqref{sG} can be achieved and we conclude \eqref{LL}. 

Third, we show that strict convexity of $H$ in $G$ implies   strict convexity of $L$ in $G$. For any $\vs_1, \vs_2\in G$, from \eqref{chL} above, there exist $\vp_1,\vp_2\in G$ such that $\vs_1=\nabla_p H(\vp_1, \vx), \, \vs_2=\nabla_p H(\vp_2, \vx)$. Then we have 
\begin{equation}
(\vs_1 - \vs_2) \cdot \bbs{\nabla_s L(\vs_1, \vx) - \nabla_s L(\vs_2, \vx) } = \bbs{\nabla_p H(\vp_1, \vx) - \nabla_p H(\vp_2, \vx)} \cdot (\vp_1 - \vp_2)>0
\end{equation}
due to the strict convexity of $H$ in $G$.

(ii) Let $\vx(t)$ be the least action such that 
\begin{equation}\label{least_ac}
\vx(\cdot) = \text{arg}\min_{\vx(0)=\vx_0, \vx(T)=\vec{b}} \ac(\vx(\cdot)). 
\end{equation}
 Then $\vx(t)$ satisfies the Euler-Lagrange equation \eqref{EL}.
From   \eqref{ts1}, \eqref{ts2}, we know given $\vs, \vx$ 
\begin{equation}
\frac{\pt L}{\pt \vx}(\vs,\vx) = - \nabla_x H(\vp^*(\vs,\vx), \vx).
\end{equation}
Thus  for $\vp = \frac{\pt L}{\pt \vs}$, the Hamiltonian dynamics \eqref{HM} holds. 

(iii) First,  let $\vx(t)$ is the solution to   RRE \eqref{odex} with initial data $\vx_0$ and set $\vp(0)=\vec{0}$. Then
\begin{equation}
\vp\equiv \vec{0}, \quad \frac{\ud}{\ud t} \vx = \nabla_p H(\vec{0}, \vx).
\end{equation}
This corresponds to a least action $\vx(t)$ such that $\ac(\vx(\cdot))=0$.

Second, assume $\vec{x}$ is a least action such that $\ac(\vx(\cdot))=0$, then $L(\dot{\vx}(t),\vx(t))\equiv 0$ for all $t\in[0,T]$ and the Hamiltonian dynamics \eqref{HM} holds. It is sufficient to prove the following two cases. \\
Case(I), if there exists $t^*$ such that $\vp(t^*)=\vq$ for some $\vq\in \kk(\nu)$, then $\vp(t)\equiv \vq\in\kk(\nu)$ because $\dot{\vp}=-\nabla_x H(\vq,\vx)=\vec{0}$. Thus
\begin{equation}
\frac{\ud}{\ud t}\vec{x} =  \nabla_p H(\vq, \vec{x}) = \nabla_p H(\vec{0}, \vec{x}) = \sum_{j=1}^M \vec{\nu}_j \bbs{\Phi^+_j(\vec{x}) - \Phi^-_j(\vec{x})}
\end{equation}
implies $\vx(t)$ is the solution to RRE \eqref{odex}. \\
Case (II), if $\vp(t) \notin \kk(\nu)$ for all $t\in[0,T]$, then we know $\vec{\nu}_j \cdot \vp \neq 0$. Then from \eqref{ts1} and Lemma \ref{Hconvex}, we have
\begin{equation}
\dot{\vx}(t) = \vs = \nabla_p H(\vp(t),\vx(t)) \neq \nabla_p H(\vec{0},\vx(t)).
\end{equation}
However, from the strict convexity of $L$, $L(\vs,\vx)>0$ for $\vs\neq \nabla_p H(\vec{0},\vx(t))$, which contradicts with $\ac(\vx(\cdot))=0.$

 Thus we conclude (iii).
\end{proof}

As mentioned in the introduction,     Lax-Oleinik's representation \eqref{LO}, to which the Varadhan's nonlinear semigroup converges, shows that the   function $L$ defined in \eqref{L}  actually gives the good rate function for the large derivation principle  for the large number process $\cv (t)$ at single times. Precisely,
\begin{thm}[\cite{GL_vis}]
Let $\cv $ be the large number process   defined in \eqref{Csde} with generator $Q_{\V}$. Assume $\cv (0)=\vx_0^{\V}$ satisfying $\vx_0^{\V} \to \vx_0$ in $\bR^N$.  Then at each time $t$, the random variable $\cv (t)$  satisfies the large deviation principle in $\bR^N_+$ with the good rate function $I_{x_0,t}(\vy)$ defined in \eqref{LO}.
That is, for any open set $\mathcal{O}\subset  \bR^N_+$, it holds
\begin{equation}
\liminf_{V\to +\8} \frac{1}{V} \log \bP_{\vx_0^V} \{\cv (t) \in  \mathcal{O}\} \geq - \inf_{\vx\in \mathcal{O}} I_{\vx_0, t}(\vx)
\end{equation}
while for any closed set $\mathcal{C}\subset \bR^N_+$, it holds
\begin{align}
\limsup_{V\to +\8} \frac{1}{V} \log \bP_{\vx_0^V} \{\cv (t) \in\mathcal{C}\} \leq - \inf_{\vx\in \mathcal{C}} I_{\vx_0, t}(\vx).
\end{align}
\label{LD11}
\end{thm}
 %This large deviation principle was derived by \cite{SWbook95, QianGe17} via a WKB approximation. 
 
The    sample path large deviation principle  in the path space $D([0,T];\bR^N_+)$, i.e., the space of c\`adl\`ag functions, which is proved by \textsc{Agazzi} et.al. \cite{Dembo18}, is more difficult and significant. 
 Under some mild assumptions ensuring the existence of solution $\cv (t)$, \cite{Dembo18}   proved the sample path large derivation principle for $\{\cv (t)\}$, as restated   below. 
\begin{thm}[Theorem 1.6, \cite{Dembo18}]
Let $\cv $ be the large number process   defined in \eqref{Csde} with generator $Q_{\V}$ defined in \eqref{generator}. Assume $\cv (0)=\vx_0^V$ satisfying $\vx_0^V \to \vx_0$ in $\bR^N$.  Then the sample path $\cv (t)$, $t\in [0,T]$ satisfies the large deviation principle in $D([0,T];\bR^N_+)$ with the good rate function
\begin{equation}
A_{x_0, T}(\vx(\cdot)) := \left\{ 
\begin{array}{cc}
\int_0^T L(\dot{\vx}(t),\vx(t)) \ud t & \text{ if } \vx(0)=\vx_0,\,\, \vx(\cdot)\in AC([0,T];\bR^N),\\
+\8 & \text{ otherwise.}
\end{array}
\right. 
\end{equation} 
That is, for any open set $\mathcal{E}\subset D([0,T];\bR^N_+)$, it holds
\begin{equation}
\liminf_{V\to +\8} \frac{1}{V} \log \bP_{\vx_0^V} \{\cv (t) \in  \mathcal{E}\} \geq - \inf_{\vx\in \mathcal{E}} A_{\vx_0, T}(\vx(\cdot)),
\end{equation}
while for any closed set $\mathcal{G}\subset D([0,T];\bR^N_+)$, it holds
\begin{align}
\limsup_{V\to +\8} \frac{1}{V} \log \bP_{\vx_0^V} \{\cv (t) \in \mathcal{G}\} \leq - \inf_{\vx\in \mathcal{G}} A_{\vx_0, T}(\vx(\cdot)),
\end{align}
where $D([0,T];\bR^N_+)$ is Skorokhod space and $AC([0,T];\bR^N)$ is space of absolute continuous curves.
\label{LD22}
\end{thm}

The sample path large deviation principle Theorem \ref{LD22}  covers the above single time result in Theorem \ref{LD11}. Indeed, for any fixed open set $\mathcal{O}\subset \bR^N_+$, one takes special open set $\mathcal{E}\subset D([0,T];\bR^N_+) $ as $\mathcal{E} = \{\vx(\cdot)\in D([0,T];\bR^N_+); \, \vx(t) \in \mathcal{O}\}$.  Then
$$\inf_{\vx\in \mathcal{E}} A_{\vx_0, T}(\vx(\cdot)) = \inf_{\vy \in \mathcal{O}} \bbs{\inf_{\vx(\cdot)\in D([0,T];\bR^N_+),\, \vx(0)=\vx_0,\, \vx(t) = \vy } \int_0^T L(\dot{\vx}(s),\vx(s)) \ud s} = \inf_{\vy \in \mathcal{O}} I_{\vx_0, t}(\vy). $$
Here in the last equality, the least action from $0$ to $T$ is the combination of  the least action from $0$ to $t$ and a zero-cost action for $t$ to $T$.
 However, \cite{GL_vis} gives an alternative proof for the simple case in Theorem \ref{LD11} using semigroup approach.

\section{The dynamic and steady solution to Hamilton-Jacobi equation}\label{sec3}
  {\blue In this section, we first study  a   general Hamiltonian $H(\vp, \vx)$ and its HJE. (i) As a result of  the law of large numbers $p(\vxv,t)\approx \delta_{\vx(t)}$ in the large number limit, the minimizer of the dynamic  solution to the HJE gives the deterministic macroscopic path, which is a solution to the corresponding large number limiting ODE; see Proposition \ref{prop1}. (ii) The steady solution $\psi^{ss}(\vx)$ to the HJE gives a Lyapunov function and a conservative-dissipative decomposition for the macroscopic RRE; see Theorem \ref{thm_decom}. The thermodynamics for detailed/complex balanced RRE and also for general RRE    will be discussed in Section \ref{subsec_thermo} at the mesoscopic scale and in Section \ref{sec_meso2} after the passage from mesoscopic scale to macroscopic scale.}
\subsection{Kurtz's limiting ODE as the minimizer of HJE solution $\psi$ and its Lyapunov function}
  We further observe the following special properties for $H$
\begin{equation}\label{Hf}
H(\vec{0},\vec{x})\equiv 0, \quad \text{ and thus } \nabla_x H(\vec{0}, \vec{x}) \equiv 0.
\end{equation}
 In the following theorem, we will show the solution to the Kurtz limiting ODE \eqref{odex} is the minimizer of HJE solution $\psi$. Indeed,  without any symmetry assumptions, we will prove a general theorem that the HJE solution $\psi(\vx,t)$ from the WKB expansion can always characterize the ODE path given by the law of large numbers, and the steady solution $\psi^{ss}(\vx)$ yields a Lyapunov function to RRE \eqref{odex}. 
\begin{prop}\label{prop1}
Let $\psi_0(\vx)$ be the initial data to HJE \eqref{HJE2psi} with a generic Hamiltonian $H(\vp,\vx)$ satisfying \eqref{Hf}. Assume $\min_{\vx} \psi_0(\vx)=0$ and assume $\psi_0(\vx)$ is smooth,  strictly convex   with  a linear growth at the far field, i.e. 
\begin{equation}
c_1 |\vx| \leq \psi_0(\vx) \leq c_2 |\vx|, \quad \text{ as } |\vx|\to +\8.
\end{equation}
 Then
\begin{enumerate}[(i)]
\item  there exists a unique local-in-time strictly convex classical solution $\psi(\vx, t), \, t\in[0,T]$ to \eqref{HJE2psi};
\item the global viscosity solution to \eqref{HJE2psi} is given by the Lax-Oleinik semigroup (a.k.a the optimal control formulation)
\begin{equation}\label{LO}
\psi(\vx,t) = \inf_{\gamma(\cdot)\in AC([0,t]),\, \gamma(t)=\vx} \int_0^t L(\dot{\gamma}(\tau), \gamma(\tau)) \ud \tau + \psi_0(\gamma(0)), \quad t\in[0,+\8);    
\end{equation}
\item the solution to ODE $\frac{\ud}{\ud t}\vec{x} = \nabla_p H(\vec{0}, \vec{x})$ with initial data $\vx_0=\argmin_{\vx} \psi_0(\vx) $ is the minimizer of $\psi(\vx,t)$, i.e., 
\begin{equation}
\vec{x}^*(t) =  \argmin_{\vec{x}} \, \psi(\vec{x},t), \quad \text{ for all } t\in[0,T].
\end{equation}
Along the ODE solution, $\psi(\vx(t),t)\equiv 0.$
That is to say the trajectory of the corresponding Hamiltonian dynamics with $\vp\equiv \vec{0}$ gives the  mean path in the sense of the weak law of large numbers
\begin{equation}
\lim_{V\to+\8} \bE (\varphi(C^V_t)) = \varphi(\vx^*(t)).
\end{equation}
\end{enumerate}
\end{prop}
{\blue We remark the rigorous proof for the large deviation principle of process $\cv$ shall be done via the convergence of the WKB reformulation for backward equation, i.e., the Varahdan's nonlinear semigroup, to the viscosity solution to the corresponding HJE \eqref{HJEu} \cite{GL_vis}. Then the concentration of measure gives the mean field limit equation $\frac{\ud}{\ud t}\vec{x} = \nabla_p H(\vec{0}, \vec{x})$.  However, in this proposition, we use the WKB reformulation for the forward equation because the   forward equation is more intuitive for computing the probability.  Then the formal convergence from this WKB reformulation for the forward equation  to the Lax-Oleinik semigroup representation of the viscosity solution to \eqref{HJE2psi} yields the mean field limit equation.}
\begin{proof}
Step 1. Using the definition of the Hamiltonian in \eqref{H}, we solve the following HJE by the characteristic method
\begin{equation}
\pt_t \psi  + H(\nabla \psi , \vec{x})=0, \quad \psi(\vec{x}, 0)=\psi_0(\vec{x}).
\end{equation}
Then constructing the characteristics $\vx(t), \vp(t)$ 
\begin{equation}\label{cc}
\begin{aligned}
\dot{\vec{x}} = \nabla_p H(\vec{p}, \vec{x}), \quad \vec{x}(0) = \vec{x}_0,\\
\dot{\vec{p}} = -\nabla_x H(\vec{p}, \vec{x}), \quad \vec{p}(0) = \nabla \psi_0(\vec{x}_0).
\end{aligned}
\end{equation}
From the assumptions on $\psi_0$, we know there exists $T$ such that the characteristics $\vx_1(t)$, $\vx_2(t)$ starting from any initial data $(\vx_1(0), \vp_1(0))$, $(\vx_2(0), \vp_2(0))$ do not intersect. Thus  upto $t\in[0,T]$, $\vx(t),\vp(t)$ can be uniquely solved from \eqref{cc}.   For any $t\in[0,T]$, we also know $ \psi(\vx,t)$ is convex.  Then along characteristics,  with $\vec{p}(t) = \nabla_x\psi( \vec{x}(t), t )$,  we know $z(t)=\psi(\vx(t),t)$ satisfies
\begin{equation}
\begin{aligned}
\dot{z} =\nabla_x \psi(\vx(t),t) \cdot \dot{\vx} + \pt_t \psi(\vx(t),t) =\vec{p} \cdot \nabla_p H(\vec{p},\vec{x}) -H(\vec{p}, \vec{x}), \quad z(0) = \psi_0(\vec{x}_0).
\end{aligned}
\end{equation} 
Hence we can solve for $z(t)=\psi(\vx(t),t).$
Then we know along the characteristic
\begin{equation}
\begin{aligned}
\frac{\ud}{\ud t} H( \vec{x}(t), \vec{p}(t)) = \dot{\vx} \cdot \nabla_x H + \dot{\vp} \cdot \nabla_p H = 0,\\
\dot{z} - \vp(t)\cdot \nabla_p H(\vec{p}, \vec{x}) = - H(\vec{p},\vec{x}) = - H(\vec{p}_0, \vec{x}_0).
\end{aligned}
\end{equation}
 The Lax-Oleinik formula for the global viscosity solution in conclusion (ii) is a direct application of the dynamic program principle (semigroup property of \eqref{LO}); see \cite[Theorem 2.22]{Tran21}.

Step 2. Particularly, taking $ \vec{x}_0$ as the minimizer of $\psi_0$ such that $\nabla_x \psi_0(\vec{x}_0)=\vec{0}$ and thus $\vec{p}(0) = \vec{0}.$ Then from \eqref{Hf}, we have 
\begin{equation}
\vp(t) \equiv \vec{0}, \quad \frac{\ud}{\ud t} \psi(\vec{x}(t), t)=\dot{z} \equiv -H(\vp_0,\vx_0)= 0
\end{equation}
and we obtain ODE
\begin{equation}
\frac{\ud}{\ud t}\vec{x}(t) = \nabla_p H(\vec{0}, \vec{x}(t)).
\end{equation}
%Moreover, we know $\nabla_x \psi(\vec{x}(t),t) =\vp(t) \equiv \vec{0},$ which means Kurtz limit ODE \eqref{odex} can be recovered from the minimizer of  $\psi(\vec{x}, t)$ for $t\in[0,T]$ as
\begin{equation}\label{minXX}
\vec{x}(t) = \text{argmin}_{\vec{x}} \, \psi(\vec{x},t).
\end{equation}
 
  Now we prove the trajectory in  \eqref{minXX} is the mean path in the sense of a weak formulation of the law of large numbers. 
 Recall the WKB expansion for the law of large number process $C^V$, i.e., $p(\vxv, t)=e^{-V \psi(\vxv, t)}$ for any $t>0$. Then for any test function $\varphi(\vx)$, the expectation of $\varphi$ satisfies
 \begin{equation}
 \bE(\varphi(C^V_t))=   \sum_{\vxv} \varphi(\vxv) p(\vxv, t)  = \frac {  \sum_{\vxv} \varphi(\vxv)  e^{-V \psi(\vxv,t)}}{ \sum_{\vxv} e^{-V \psi(\vxv,t)}}, \quad t>0.
 \end{equation}
 Then by the Laplace principle, we have for any  $t\in[0,T]$,
 \begin{equation}\label{Claw}
 \lim_{V\to+\8} \bE (\varphi(C^V_t)) = \varphi(\vx^*(t)), \quad \vx^*(t) = \argmin_{\vx} \psi(\vx, t).
 \end{equation}
 In other words, the process time marginal $C^V(t)$ converges to $\vx^*(t)$ in law for any $t$.
 %This argument is a result of the concentration of measure, a.k.a. the method of steepest descent.

\end{proof}

\subsection{Conservative-dissipative decomposition for the macroscopic RRE}\label{sec_3decom}
 In this section, we study finer properties of the macroscopic RRE based on   any smooth enough stationary solutions  $\psi^{ss}(\vx)$ to HJE.
 We will first decompose the RRE as a conservative part and a dissipative part in Theorem \ref{thm_decom}. Then using another conservation law for the total mass in special chemical reactions, we explore the GENERIC formalism and bi-anti-symmetric structures in the decomposition.
 
\begin{thm}\label{thm_decom}
Let $H(\vp,\vx)$ be a convex Hamiltonian satisfying $H(\vec{0},\vx)=0$. Assume  $\psi^{ss}(\vx)$ is a steady solution to the corresponding HJE satisfying $H(\nabla \psi^{ss}(\vx), \vx)=0$. Then   we have the following conservative-dissipative decomposition for the macroscopic RRE $\frac{\ud}{\ud t}\vec{x} = \nabla_p H(\vec{0}, \vec{x})$
\begin{equation}\label{odeDD}
\begin{aligned}
&\dot{\vx} =W(\vx) \,\, - \,  K(\vx) \nabla \psi^{ss}(\vx),\\
&W(\vx):=\int_0^1 \nabla_p H(\theta \nabla \psi^{ss}(\vx),\vx) \ud \theta, \quad K(\vx):=   \int_0^1 (1-\theta) \nabla^2_{pp} H(\theta \nabla \psi^{ss}(\vx), \vx) \ud \theta.
\end{aligned} 
\end{equation}
(i) $W(\vx)$ is the conservative part satisfying
\begin{equation}
\la W(\vx),\nabla \psi^{ss}(\vx) \ra =0.
\end{equation}
Thus  we recast conservative part as $W(\vx)=\mathcal{A}(\vx) \nabla \psi^{ss}(\vx)$, where $\mathcal{A}(\vx) :=\frac{W\otimes \nabla \psi^{ss}(\vx) - \nabla \psi^{ss}(\vx)\otimes W}{|\nabla \psi^{ss}(\vx)|^2}$ 
is an anti-symmetric operator.
 \\
 (ii) $-K(\vx)\nabla \psi^{ss}(\vx)$ is the dissipative part with a nonnegative definite operator $K(\vx)$. 
Thus  any increasing function $\phi(\cdot)$ of $\psi^{ss}(\vx)$   is a Lyapunov function for the  ODE with  energy dissipation relation
\begin{equation}\label{Ly_n}
\frac{\ud}{\ud t} \phi(\psi^{ss}(\vec{x})) =\big\la  \dot{\vec{x}} \, , \,  \phi'(\psi^{ss}(\vx))\nabla \psi^{ss}(\vx) \big\ra=-\big\la \phi'(\psi^{ss}(\vx))K(\vx)   \nabla  \psi^{ss}(\vx)   ,\,  \nabla \psi^{ss}(\vx)\big\ra\leq 0.
\end{equation} 
%\begin{equation}\label{Ly}
%\begin{aligned}
%\frac{\ud}{\ud t} \psi^{ss}(\vec{x}) =&\la  \dot{\vec{x}} , \, \nabla \psi^{ss} \ra=-\la K(\vx)   \nabla  \psi^{ss}(\vx)   ,\,  \nabla \psi^{ss}(\vx)\ra\leq 0.
%\end{aligned}
%\end{equation}
\end{thm}
\begin{proof}
   Notice $\psi^{ss}(\vx)$ is a steady solution to the HJE satisfying $H(\vec{0},\vx)=H(\nabla \psi^{ss}(\vx), \vx)=0$.

We first recast the right-hand-side of the RRE as
\begin{equation}
\begin{aligned}
\nabla_p H(\vec{0}, \vx)  =& - \int_0^1 \pt_\theta\bbs{(1-\theta) \nabla_p H(\theta \nabla \psi^{ss},\vx)} \ud \theta\\
=& -\int_0^1 (1-\theta) \nabla^2_{pp} H(\theta \nabla \psi^{ss}(\vx),\vx) \ud \theta\nabla \psi^{ss}(\vx)   + \int_0^1 \nabla_p H(\theta \nabla \psi^{ss},\vx) \ud \theta.
\end{aligned}
\end{equation}
This gives the right-hand-side of  \eqref{odeDD} with the definitions $K(\vx), W(\vx)$.
%Denote 
%\begin{equation}
%K(\vx):= \la \int_0^1 (1-\theta) \nabla^2_{pp} H(\theta \nabla \psi^{ss}(\vx), \vx) \ud \theta.
%\end{equation}
%Therefore, RRE \eqref{odex} can be decomposed as the following gradient flow part and Hamiltonian flow part
%\begin{equation}
%\dot{\vx} =\int_0^1 \nabla_p H(\theta \nabla \psi^{ss}(\vx),\vx) \ud \theta \,\, - \, K(\vx) \nabla \psi^{ss}(\vx).
%\end{equation}
%Indeed, for the second term, we have
%\begin{equation}
%\la - K(\vx) \nabla \psi^{ss}(\vx), \nabla \psi^{ss}(\vx) \ra \leq 0,
%\end{equation}
%while
 For the first term in \eqref{odeDD}, taking inner product with $\nabla \psi^{ss}$ gives 
\begin{equation}\label{conser}
\la W(\vx), % \int_0^1 \nabla_p H(\theta \nabla \psi^{ss}(\vx),\vx) \ud \theta, 
\nabla \psi^{ss}(\vx) \ra = \int_0^1 \pt_\theta H (\theta \nabla \psi^{ss}(\vx), \vx) \ud \theta =H(\nabla \psi^{ss}(\vx),\vx)-H(0,\vx)=0
\end{equation}
due to  $H(\vec{0},\vx)=H(\nabla \psi^{ss}(\vx), \vx)=0$.  
From Lemma \ref{lem_Hdege} and Lemma \ref{Hconvex}, we know $\la \nabla^2_{pp} H(\vp,\vx) \vp, \vp\ra \geq 0$ and strictly positive in $G$, so $K$ is nonnegative definite operator. 

Second, from $\phi'\geq 0$ and the orthogonality \eqref{conser}, we conclude   the energy dissipation \eqref{Ly_n}. Particularly,
\begin{equation}\label{Ly}
\begin{aligned}
\frac{\ud}{\ud t} \psi^{ss}(\vec{x}) =&\la  \dot{\vec{x}} , \, \nabla \psi^{ss} \ra=-\la K(\vx)   \nabla  \psi^{ss}(\vx)   ,\,  \nabla \psi^{ss}(\vx)\ra\leq 0.
\end{aligned}
\end{equation}
\end{proof}
\begin{rem}
{\blue We remark that in the chemical  Langevin approximation, see  for instance \cite{GG00} and \eqref{clr}, the   Lagrangian and Hamiltonian are both quadratic and thus the above decomposition becomes transparent. We point out the choice of the Hamiltonian for the RRE decomposition is not unique, which leads to different interpretations in energetics and kinetics; c.f.,  \cite{Peletier_Redig_Vafayi_2014}. }
\end{rem}
\begin{rem}
Although we have a family of Lyapunov functions $\phi(\psi^{ss}(\vx(t)))$, we will see only the stationary solution $\psi^{ss}$ is the energy landscape of the chemical reactions later in Section \ref{sec_reversal}.
  The energy dissipation \eqref{Ly} can be regarded as the large number limit of the energy dissipation law for the mesoscopic master equation in terms of the natural relative entropy $\rho\log\frac{\rho}{\pi}-\rho+1$; see Proposition \ref{prop_meso_limit} for a passage from mesoscopic to macroscopic in the large number limit.
In other words, the RRE can be decomposed as an Onsager-type strong  gradient flow  
 in the direction of $\nabla \psi^{ss}$,   and a conservative flow  in the orthogonal direction of $\nabla \psi^{ss}$; both with the same free energy $\psi^{ss}(\vx)$. Thus \eqref{odeDD} can be regarded as a conservative dynamics coupling with a dissipation structure. Due to the competition in the chemical reaction represented by Hamiltonian $H$ between conservative force and dissipation in terms of a nonconvex energy landscape $\psi^{ss}(\vx)$,  this system can exhibit complicated dynamic patterns such as limit cycles, oscillations, chaotic attractors and multi-stability, etc.. 
The decomposition \eqref{odeDD} is  a ``pre-GENERIC'' formalism. The concept   ``pre-GENERIC'' was proposed by \textsc{Kraaij} et.al. \cite{Kraaij_Lazarescu_Maes_Peletier_2020} which replaces the Hamiltonian part $\mathcal{A}\nabla E$ in the original GENERIC formalism by  a general orthogonal term $W$ such that $\la W, \nabla\psi^{ss} \ra 
=0$.
\end{rem}
 
Below, we explore further the GENERIC formalism and   two anti-symmetric structures for RRE by  utilizing the additional mass conservation law for chemical reaction, i.e., 
 for any $\vm\in \kk(\nu)$,  we have $\frac{\ud }{\ud t} \vx(t) \cdot \vec{m} =0$.
 
 \subsubsection{GENERIC formalism for RRE}
 Denote the conservative part in the decomposition \eqref{odeDD} as
\begin{equation}
W(\vx):= \int_0^1 \nabla_p H(\theta \nabla \psi^{ss}(\vx),\vx) \ud \theta.
\end{equation}

 For any $\vm \in \kk(\nu)$, we use the conserved mass as the role of the conserved energy functional in GENERIC  formalism
 \begin{equation}
 E(\vx) := \vm \cdot \vx.
 \end{equation}
 Since $\nabla_p H(\vp, \vx) = \sum_j \vec{\nu}_j \bbs{\Phi_j^+ e^{\vec{\nu}_j \cdot \vp} - \Phi_j^- e^{-\vec{\nu}_j \cdot \vp} } \in G$,  the conservative part $W$ is orthogonal to $\nabla E = \vm$
\begin{equation}\label{conser_m}
\la W(\vx), \vm \ra = 0, \quad \forall \vm\in \kk(\nu).
\end{equation} 
Therefore, $W$ can be recast as
\begin{equation}\label{Wr1}
W(\vx) = \frac{(W\otimes \vm - \vm \otimes W)\vm}{|\vm|^2} =: \mathcal{A}_1(\vx) \nabla E,
\end{equation}
where $\mathcal{A}_1(\vx):= \frac{(W\otimes \vm - \vm \otimes W)}{|\vm|^2}$ is an anti-symmetric matrix satisfying
\begin{equation}
\mathcal{A}_1(\vx)\nabla \psi^{ss}(\vx)=\vec{0}.
\end{equation}
Here we remark that since $H$ is degenerate in $\kk(\nu)$, without loss of generality, $\psi^{ss}(\vx)$ can be chosen within $G$ or we only require $\la W(\vx),\nabla \psi^{ss}(\vx) \ra=0.$

On the other hand, by Lemma \ref{lem_Hdege}, we know $\vm^T K(\vx)\vm = 0$. Since $K(\vx)$ is nonnegative definite, we have
\begin{equation}
K(\vx) \nabla E=0.
\end{equation}
In summary, for special Hamiltonian $H$ in chemical reactions,  the decomposition \eqref{odeDD} for RRE can be recast as GENERIC formalism
\begin{equation}\label{decom_gen}
\dot{\vx} = \mathcal{A}_1(\vx)\nabla E(\vx) - K(\vx) \nabla \psi^{ss}(\vx).
\end{equation}

\subsubsection{Two anti-symmetric  structures for the conservative part}
Apart from the full decomposition, we study  two anti-symmetric structures of the conservative part $W(\vx)=\int_0^1 \nabla_p H(\theta \nabla \psi^{ss}(\vx),\vx) \ud \theta$ in \eqref{odeDD}  for RRE.

First, we observe
  there are two conservation laws for $W$: one is \eqref{conser_m} and the other one is  
\begin{equation}
\la W(\vx), \nabla \psi^{ss}(\vx) \ra =0
\end{equation}
due to \eqref{conser}.
 Following the same idea for constructing the anti-symmetric operator in \eqref{Wr1}, given any conservation laws $\la W, \vec{a} \ra =0$, $W(\vx)$ can always be recast as anti-symmetric matrix $\mathcal{A}$ using
 \begin{equation}
 W = \frac{\bbs{W\otimes \vec{a} - \vec{a}\otimes W}\vec{a}}{|\vec{a}|^2}=:\mathcal{A}(\vx) \vec{a}.
 \end{equation}
 Particularly, using the formula for $H$ in \eqref{H},
 \begin{equation}\label{Wr2}
 \begin{aligned}
 W(\vx) =& \sum_j \vec{\nu}_j \bbs{ \Phi^+_j \int_0^1 e^{\theta \vec{\nu}_j \cdot \nabla \psi^{ss} } \ud \theta -  \Phi^-_j \int_0^1 e^{-\theta \vec{\nu}_j \cdot \nabla \psi^{ss} } \ud \theta }\\
 =& \sum_j  \bbs{ \Phi^+_j \bbs{{ e^{\vec{\nu_j} \cdot \nabla \psi^{ss} }}   -  1 } +  \Phi^-_j \bbs{{ e^{-\vec{\nu_j} \cdot \nabla \psi^{ss} } }  - 1  }} \frac{\bbs{\vec{\nu}_j \otimes \nabla\psi^{ss}}\nabla\psi^{ss}}{ \vec{\nu}_j \cdot \nabla\psi^{ss} |\nabla \psi^{ss}|^2}\\
 =& \sum_j  \bbs{ \Phi^+_j \bbs{{ e^{\vec{\nu_j} \cdot \nabla \psi^{ss} }}   -  1 } +  \Phi^-_j \bbs{{ e^{-\vec{\nu_j} \cdot \nabla \psi^{ss} } }  - 1  }} \frac{\bbs{\vec{\nu}_j \otimes \nabla\psi^{ss} -\nabla\psi^{ss}\otimes \vec{\nu}_j }}{ \vec{\nu}_j \cdot \nabla\psi^{ss} |\nabla \psi^{ss}|^2}\nabla\psi^{ss}=:\mathcal{A}_2(\vx)\nabla\psi^{ss}.
 \end{aligned}
 \end{equation}
 That is to say, for chemical reactions with $H$ in \eqref{H}, the decomposition \eqref{odeDD} can be recast as 
\begin{equation}\label{RREdecom}
\frac{\ud}{\ud t}\vec{x} = \nabla_p H(\vec{0}, \vec{x}) = \sum_{j=1}^M \vec{\nu}_j \bbs{\Phi^+_j(\vec{x}) - \Phi^-_j(\vec{x})}=: \mathcal{A}_2(\vx) \nabla \psi^{ss}(\vx) -K(\vx)\nabla \psi^{ss}(\vx)
\end{equation}
with explicit formulas 
\begin{equation}\label{K12}
\begin{aligned}
K(\vx):= \sum_{j=1}^M     \bbs{ \Phi^+_j \bbs{{ e^{\vec{\nu_j} \cdot \nabla \psi^{ss} }}   -  1 - \vec{\nu_j} \cdot \nabla \psi^{ss}  } +  \Phi^-_j \bbs{{ e^{-\vec{\nu_j} \cdot \nabla \psi^{ss} } }  - 1 + \vec{\nu_j} \cdot \nabla \psi^{ss} }}\frac{\vec{\nu}_j\otimes \vec{\nu}_j}{|\vec{\nu_j} \cdot \nabla \psi^{ss}|^2}\\
\mathcal{A}_2(\vx):= \sum_{j=1}^M     \bbs{ \Phi^+_j \bbs{{ e^{\vec{\nu_j} \cdot \nabla \psi^{ss} }}   -  1 } +  \Phi^-_j \bbs{{ e^{-\vec{\nu_j} \cdot \nabla \psi^{ss} } }  - 1  }}\frac{\bbs{\vec{\nu}_j \otimes \nabla\psi^{ss} -\nabla\psi^{ss}\otimes \vec{\nu}_j }}{ \vec{\nu}_j \cdot \nabla\psi^{ss} |\nabla \psi^{ss}|^2}.
\end{aligned}
\end{equation}
It is easy to verify the positive symmetry of $K$, the anti-symmetry of $\mathcal{A}$ and
$$\la K(\vx) \nabla \psi^{ss}(\vx), \nabla \psi^{ss}(\vx) \ra \geq 0, \quad \la \mathcal{A}_2(\vx) \nabla \psi^{ss}(\vx), \nabla \psi^{ss}(\vx)   \ra=0.$$
We point out the decomposition \eqref{RREdecom} is  analogous to the  Landau-Lifshitz equation instead of the GENERIC formalism \eqref{decom_gen}.

In summary, combining \eqref{Wr1} and \eqref{Wr2},  two conservation laws leads to two anti-symmetric structures for the conservative part $W(\vx)$
\begin{equation}
W(\vx) =\mathcal{A}_1(\vx)\nabla E(\vx)= \mathcal{A}_2(\vx)\nabla\psi^{ss}(\vx).
\end{equation}
%In terms of only the conservative part $W$, as long as one can verify the associated bi-Hamiltonian pair is non-degenerate, then it induces a hierarchy of commuting Hamiltonian flows and thus a completely integrable system \cite{biH}.

 With the above  characterization for relations between solutions to HJE \eqref{HJE2psi} and  RRE trajectory in Proposition \ref{prop1}, we summarize the corresponding relations and the decompositions for RRE in the following Corollary. 
\begin{cor}
For    Hamiltonian $H(\vp,\vx)$ in \eqref{H} for a chemical reaction, the minimizer of the dynamic solution $\psi(\vx,t)$ gives the unique solution $\vx(t)$ to the RRE \eqref{odex} with initial data $\vx_0=\min_{\vx} \psi_0$. The steady solution $\psi^{ss}(\vx)$ gives a Lyapunov function to RRE \eqref{odex}. The   GENERIC formalism for RRE reads as \eqref{decom_gen}.  Another decomposition for RRE reads as \eqref{RREdecom} with   $K$ and $\mathcal{A}_2$ defined in \eqref{K12}. 
\end{cor}

As an example, we show that if   RRE satisfies the detailed/complex balance condition, then it is well-known that a closed formula solution for the Lyapunov function is $\psi^{ss}(\vx)=\KL(\vx||\peq);$ c.f. \cite{sontag2001structure}.
In the RRE detailed balanced case,
using \eqref{tempPhi},
 we have dissipation relation
$
\frac{\ud}{\ud t} \KL(\vec{x}(t)||\peq) = \sum_i \dot{x}_i \log \frac{x_i}{\xss_i} 
  =- \sum_j \bbs{\Phi^-_j (\vec{x})- \Phi^+_j(\vec{x})}  \log \bbs{\frac{\Phi^-_j(\vec{x})}{\Phi^+_j(\vec{x})}} \leq 0$.
For the complex balance case,
we summarize the following equivalent characterization for complex balance condition \eqref{cDB}. The proof will be given in Appendix \ref{app_lem_cDB}
\begin{lem}\label{lem_cDB}
Assume $\peq>0$ is a positive steady state to RRE \eqref{odex}. Then the following statements are equivalent:
\begin{enumerate}[(i)]
\item $\peq$ satisfies complex balance condition \eqref{cDB};
\item The relative entropy $ \psi^{ss}(\vx)=\KL(\vx||\peq)$  is a steady solution to HJE \eqref{HJE2psi};
\item The product of Poisson distribution $\pi_V(\vxv)=\Pi_{i=1}^N\frac{(V \xss_i)^{V x_i}}{(V x_i)!}e^{-V \xss_i}$  is a positive invariant measure $\pi_V(\vxv)$ to the mesoscopic CME \eqref{rp_eq}.
\end{enumerate}
\end{lem}

\subsection{ Thermodynamics of chemical reaction systems}\label{subsec_thermo}

Suppose chemical reactions is in a large reservoir that has a constant temperature $T$ and a constant pressure. 
In this section, we devote to study the thermodynamics  for non-equilibrium  chemical reactions. 
For non-equilibrium chemical reactions,  the reaction affinity along  $j$-th reaction pathway is introduced by \textsc{Kondepudi, Prigogine} \cite{kondepudi2014modern}
$$\mathcal{A}_j=k_{\B}T \log  \frac{\Phi_j^+(\vec{x})} {\Phi_j^{-}(\vec{x})},$$
where $k_{\B}$ is the Boltzmann constant.
For the equilibrium reaction, $\mathcal{A}_j$ reduces to the difference of the Gibbs free energy along $j$-th reaction pathway \eqref{tm3.35}, so this is a natural extension from equilibrium reactions.
Then  the total entropy production rate is
\begin{equation}\label{Stot}
T\dot{S}_{\tot}:=\sum_j  \bbs{\Phi_j^+(\vec{x}) -\Phi_j^-(\vec{x})} \mathcal{A}_j  = k_{\B}T\sum_j \bbs{\Phi_j^+(\vec{x}) -\Phi_j^-(\vec{x})}  \log  \frac{\Phi_j^+(\vec{x})} {\Phi_j^{-}(\vec{x})}.
\end{equation}

We will first decompose the total entropy production rate into adiabatic and nonadiabatic contributions. Then we apply it to equilibrium reactions to check the consistency with classical thermodynamic relations based on the Gibbs theory. Particularly, as $t\to+\8$,  at non-equilibrium steady states (NESS), the non-equilibrium dynamics still maintain a positive total entropy production rate.

Recall that the energy landscape $\psi^{ss}$ gives dissipation relation \eqref{Ly} in Proposition \ref{prop1}.
The nonadiabatic entropy production rate representing the dissipation of the energy landscape $\psi^{ss}$ is defined as
  \begin{align*}
  T \dot{S}_{na}:= -k_{\B}T \frac{\ud}{\ud t} \psi^{ss}(\vx(t)) =  k_{\B}T \la K(\vx)   \nabla  \psi^{ss}(\vx),\,  \nabla \psi^{ss}(\vx) \ra\geq 0.
  \end{align*}
  As $t\to+\8$, this nonadiabatic entropy production rate goes to zero. However, as one of the most important features for non-equilibrium reaction, at NESS, the total 
  entropy production rate is positive. Therefore, apart from the nonadiabatic entropy production rate $T \dot{S}_{na}$, the remaining part in $T\dot{S}_{\tot}$ is usually called 
 the adiabatic entropy production rate 
  \begin{equation}\label{WWW}
  \begin{aligned}
   T \dot{S}_{a} =& T\dot{S}_{\tot} - T \dot{S}_{na}\\
   =& k_{\B}T \sum_j \bbs{\Phi^+_j (\vec{x}(t)) - \Phi^-_j(\vec{x}(t))}\log \bbs{ \frac{\Phi^+_j (\vx (t)) }{ \Phi^-_j(\vx (t))}  e^{\vec{\nu}_j \cdot \nabla \psi^{ss}}}\\
 =&  k_{\B}T \sum_j \bbs{ \KL(\Phi^+_j (\vec{x}(t))||\Phi^-_j (\vec{x}(t))e^{-\vec{\nu}_j \cdot \nabla \psi^{ss}}) + \KL(\Phi^-_j (\vec{x}(t))||\Phi^+_j (\vec{x}(t))e^{\vec{\nu}_j \cdot \nabla \psi^{ss}})}\geq 0,
   \end{aligned}
  \end{equation}
where in the last equality, we used $H(\nabla \psi^{ss}(\vx),\vx)=0.$
When $t\to+\8$,  we have
\begin{equation}
T \dot{S}_{a} \to k_{\B}T \sum_j \bbs{\Phi^+_j(\peq) - \Phi^-_j(\peq)}\log\frac{\Phi^+_j(\peq)}{\Phi^-_j(\peq)} 
\end{equation} 
which is strictly positive for $\peq$ being NESS.

In summery, we have
\begin{prop}\label{prop_Stot}
The thermodynamic decomposition for the total entropy production rate \eqref{Stot} of a non-equilibrium chemical reaction is given by
\begin{equation}
\begin{aligned}
T \dot{S}_{\tot}&=T \dot{S}_{na} + T \dot{S}_{a} \geq 0,\\
T \dot{S}_{na}&=  k_{\B}T \la K(\vx)   \nabla  \psi^{ss}(\vx),\,  \nabla \psi^{ss}(\vx)\ra\geq 0,
\\
T \dot{S}_{a}&=k_{\B}T \sum_j \bbs{ \KL(\Phi^+_j (\vec{x}(t))||\Phi^-_j (\vec{x}(t))e^{-\vec{\nu}_j \cdot \nabla \psi^{ss}}) + \KL(\Phi^-_j (\vec{x}(t))||\Phi^+_j (\vec{x}(t))e^{\vec{\nu}_j \cdot \nabla \psi^{ss}})} \geq 0,
\end{aligned}
\end{equation}
where $K(\vx)$ is the nonnegative definite operator in \eqref{K12} and $\psi^{ss}(\vx)$ is the stationary solution to HJE \eqref{HJE2psi}.
\end{prop}

Now we  review the thermodynamic decomposition for the equilibrium  chemical reactions to see the  above discussion is a natural generalization to non-equilibrium reactions.  
Assume the chemical potential $\mu_i$ satisfies the thermodynamic relation
$$\mu_i = \mu_i^G+k_{\B}T\log x_i, \quad \mu_i^G = \mu^0_i-k_{\B}T \log x^0_i$$
where $\mu_i^G$ is a reference Gibbs free energy (physically called standard-state Gibbs free energy) and $\log x_i - \log x^0_i$ together is a dimensionless quantity.
From LMA \eqref{lma}, we have
%\begin{equation}\label{PhiDB}
%\log \frac{\Phi_j^+(\vec{x})}{\Phi_j^{+}(\peq)} = \sum_i \nu^+_{ji}\bbs{ \log x_i -\log x^{s}_i} =  \sum_i \nu^+_{ji} \frac{\mu_i -\mu_i^{eq}}{k_{\B}T}.
%\end{equation}
%This, together with a similar relation for $\Phi^-$, gives the relation
\begin{equation}\label{tm3.35}
\log \bbs{ \frac{\Phi_j^-(\vec{x})}{\Phi_j^{-}(\peq)}   \frac{\Phi_j^+(\peq)}{\Phi_j^{+}(\vec{x})} }= \sum_i \nu_{ji}\bbs{ \log x_i -\log x^{s}_i} =  \sum_i \nu_{ji} \frac{\mu_i -\mu_i^{eq}}{k_{\B}T}.
\end{equation}

In the RRE detailed/complex balanced case,
  \textsc{Rao, Esposito} \cite{Rao_Esposito_2016} introduced  a Lyapunov function called  Shear Lyapunov 
function 
$\psi^{ss}(\vx)=  G_{eq} + k_{\B}T \KL(\vec{x}(t)||\peq)$, where $G_{eq}$ is the equilibrium Gibbs free energy with an additional linear combination of conservative quantities. 
Particularly, \cite{Rao_Esposito_2016} also extend this relation to an open chemical reaction network, where  chemostat species interacting with both internal species and environment  
 are included. With the special Shear Lyapunov 
function, the decomposition \eqref{WWW} is reduced to
 \begin{equation}\label{xxx}
\begin{aligned}
T \dot{S}_{a} = & k_{\B}T \sum_j \bbs{\Phi^+_j (\vec{x}(t)) - \Phi^-_j(\vec{x}(t))}\log \bbs{ \frac{\Phi^+_j (\vec{x}(t)) }{ \Phi^-_j(\vec{x}(t))} e^{\vec{\nu}_j \cdot  \log \frac{x_i}{\xss_i} }} \\
= & k_{\B}T \sum_j \bbs{\Phi^+_j (\vec{x}(t)) - \Phi^-_j(\vec{x}(t))}\log \bbs{ \frac{\Phi^+_j (\peq (t)) }{ \Phi^-_j(\peq (t))} }\\
=& T \dot{S}_{\tot} - T \dot{S}_{na} = T \dot{S}_{\tot} + 
k_{\B}T \frac{\ud}{\ud t} \KL(\vec{x}(t)||\peq), 
\end{aligned}
\end{equation}
where the entropy production rate from adiabatic contribution $T\dot{S}_{a}$
represents the  chemical work rate performed by the chemostats when interacting with the environment; see   \cite[eq. (84)]{Rao_Esposito_2016}. 

 We point out for general non-equilibrium reactions, $\psi^{ss}(\vx)$ is an asymptotically effective energy in the large number limit, which is different from the   thermodynamic free energy (Kirkwood potential \cite{kirkwood1935statistical}). On the other hand, we also discuss the passage from the mesoscopic thermodynamics to the macroscopic thermodynamics in Section \ref{sec_meso2} below, where $\psi^{ss}$ can be regarded as the large number limit of the mesoscopic relative entropy.

\subsection{Energy dissipation law and passage from mesoscopic to macroscopic dynamics}\label{sec_meso2}
 
For a general non-equilibrium RRE and the corresponding CME,  we first derive the $\phi$-divergence energy dissipation law  based on the $Q_V$-matrix structure and a Bregman's divergence. This type of $\phi$-divergence energy dissipation law was previously derived by \cite{maas2020modeling} under the RRE detailed balance condition.

 Based on this energy dissipation law for general non-equilibrium reactions,  we take $\phi(\rho)=\rho\log \frac{\rho}{\pi_V}$, then as $V\to +\8$, the corresponding mesoscopic energy dissipation relation converges to the macroscopic energy dissipation relation in terms of the energy landscape $\psi^{ss}(\vx)$. This shows the passage from a mesoscopic convex functional to a macroscopic non-convex function for general non-equilibrium chemical reactions.   We also remark that under the RRE detailed balance condition, \cite{maas2020modeling} rigorously proved the  evolutionary $\Gamma$-convergence in the generalized gradient flow setting for the passage from mesoscopic to macroscopic dynamics. However, in the RRE detailed balance, there is no such a transition from convex functional to a non-convex function; see Remark \ref{rem_nonconvex}.
\begin{prop}\label{prop_meso_limit}
 Assume there exists a positive invariant measure $\pi_V(\vxv)$ and the limit  $\psi^{ss}(\vx):= \lim_{V\to +\8} - \frac{\log \pi_V(\vxv)}{V}$ exists. Then %for $p_i=p(\vxv,t), \pi_i=\pi_V(\vxv)$, 
 \begin{enumerate}[(i)]
 \item  for any convex function $\phi$, we have the mesoscopic energy dissipation relation 
  \begin{equation}\label{dissi_mic_ge}
\frac{\ud }{\ud t} \sum_{\vxv}  \phi \bbs{\frac{p(\vxv)}{\pi(\vxv)}} \pi(\vxv) = -\sum_{\vxv, \vyv} Q(\vyv,\vxv) \pi(\vyv) D_\phi\bbs{ \frac{p(\vyv)}{\pi(\vyv)}, \frac{p(\vxv)}{\pi(\vxv)}}   \leq 0,
\end{equation}
where $
 D_\phi(y,x):= (y-x)^2\int_0^1 (1-\theta)\phi''(x+\theta(y-x))\ud \theta \geq 0.
$
Particularly, taking $\phi(x)=x\log x - x +1 \geq 0$
 \begin{equation}\label{dissi_mic}
\frac{\ud }{\ud t} \sum_{\vxv} p(\vxv) \log \frac{p(\vxv)}{\pi(\vxv)}  = -\sum_{\vxv,\vyv} Q(\vyv,\vxv) p(\vyv)  \log \frac{p(\vyv) \pi(\vxv)}{\pi(\vyv) p(\vxv)}  \leq 0,
\end{equation}
where $Q$ is the $Q_V$-matrix in \eqref{rp_eq} for the mesoscopic jumping process $C_V$;
\item Formally, as $V\to +\8$, the mesoscopic dissipation law \eqref{dissi_mic} converges to the macroscopic dissipation law \eqref{Ly} in the sense that
\begin{align}\label{meso_limit}
\frac{1}{V}\sum_{\vxv} p(\vxv) \log \frac{p(\vxv)}{\pi(\vxv)} \to  \psi^{ss}(\vec{x}^*),\quad 
 \frac{1}{V}\sum_{\vxv,\vyv} Q(\vyv,\vxv) p(\vyv)  \log \frac{p(\vyv) \pi(\vxv)}{\pi(\vyv) p(\vxv)} \to \la K(\vx^*)   \nabla  \psi^{ss}(\vx^*), \nabla\psi^{ss}(\vx^*) \ra,
\end{align} 
 where $\vx^*(t)=\argmin_{\vx} \psi(\vx,t)$ is the mean path obtained in Proposition \ref{prop1} such that $C^V(t)$ converges in law to $\vx^*(t).$ 
 \end{enumerate}
 \end{prop}
 \begin{proof}
 First, recast the master equation for chemical reaction \eqref{rp_eq}  as $\frac{\ud p^T}{\ud t}=p^T Q$.
Since $\sum_{\vxv}Q(\vyv,\vxv)=0$ and $\sum_{\vyv}Q(\vyv,\vxv)\pi(\vyv)=0$, then for any  functions $\phi(x), \psi(x)$, 
\begin{equation}
\begin{aligned}
&\frac{\ud }{\ud t} \sum_{\vxv}  \phi \bbs{\frac{p(\vxv)}{\pi(\vxv)}} \pi(\vxv) = \sum_{\vxv,\vyv}  Q(\vyv,\vxv) p(\vyv) \phi'\bbs{\frac{p(\vxv)}{\pi(\vxv)}} \\
=& \sum_{ \vxv,\vyv} Q(\vyv,\vxv) \pi(\vyv) \frac{p(\vyv)}{\pi(\vyv)} \bbs{\phi' \bbs{\frac{p(\vxv)}{\pi(\vxv)}}-\phi'\bbs{\frac{p(\vyv)}{\pi(\vyv)}}}\\
 =&\sum_{\vxv,\vyv} Q(\vyv,\vxv) \pi(\vyv) \frac{p(\vyv)}{\pi(\vyv)} \bbs{\phi' \bbs{\frac{p(\vxv)}{\pi(\vxv)}}-\phi'\bbs{\frac{p(\vyv)}{\pi(\vyv)}}} -  \sum_{\vxv,\vyv} Q(\vyv,\vxv) \pi(\vyv) \bbs{\psi\bbs{\frac{p(\vxv)}{\pi(\vxv)}} -\psi\bbs{\frac{p(\vyv)}{\pi(\vyv)}} }\\
  =&\sum_{\vxv,\vyv} Q(\vyv,\vxv) \pi(\vyv) \bbs{ \frac{p(\vyv)}{\pi(\vyv)} \bbs{\phi' \bbs{\frac{p(\vxv)}{\pi(\vxv)}}-\phi'\bbs{\frac{p(\vyv)}{\pi(\vyv)}}} -  \bbs{\psi\bbs{\frac{p(\vxv)}{\pi(\vxv)}} -\psi\bbs{\frac{p(\vyv)}{\pi(\vyv)}} } }.
\end{aligned}
\end{equation}
Furthermore, 
take $\psi(x) := x\phi'(x)-\phi(x)$.
Denote $y=\frac{p(\vyv)}{\pi(\vyv)}$ and $x=\frac{p(\vxv)}{\pi(\vxv)}$, then 
the dissipation can be rewritten as a Bregman's divergence
\begin{equation}
y [\phi'(x) - \phi'(y)] - [\psi(x) - \psi(y)] = (y-x)\phi'(x) +\phi(x) - \phi(y)=: - D_\phi(y,x).
\end{equation}
Using the integral form of the reminder in Taylor expansion, 
\begin{equation}
 D_\phi(y,x)= (y-x)^2\int_0^1 (1-\theta)\phi''(x+\theta(y-x))\ud \theta \geq 0.
\end{equation}
This concludes \eqref{dissi_mic_ge}.  
Take $\phi(x)=x\log x - x +1 \geq 0$, 
the dissipation relation becomes \eqref{dissi_mic}.

Second, recall the change of variables in WKB expansion 
\begin{equation}
\psi_V(\vxv,t) = - \frac{\log p(\vxv,t)}{V}, \quad \psi_V^{ss}(\vxv) = - \frac{\log \pi(\vxv)}{V}. 
\end{equation}
Then 
\begin{equation}
\begin{aligned}
\log \frac{p(\vxv-\frac{\vec{\nu}_{j}}{V},t) }{ p(\vxv,t)} = - V \bbs{\psi(\vxv-\frac{\vec{\nu}_{j}}{V},t) - \psi(\vxv)} = \vec{\nu_j} \cdot \int_0^1 \nabla \psi(\vxv -\theta\frac{\vec{\nu}_{j}}{V},t ) \ud \theta.
\end{aligned}
\end{equation}
Using the definition of $Q$-matrix, the dissipation relation \eqref{dissi_mic} reads
%\begin{equation}
\begin{equation}\label{tm_diss}
\begin{aligned}
&\frac{\ud }{\ud t} \sum_{\vxv} p(\vxv,t)  \bbs{\psi_V^{ss}(\vxv) - \psi_V(\vxv,t)} \\
=&   - \sum_{\vxv}\sum_{ j=1}^M \Big[
 \tilde{\Phi}^+_j(\vxv-\frac{\vec{\nu}_{j}}{V})     p(\vxv-\frac{\vec{\nu}_{j}}{V},t)  \vec{\nu_j} \cdot \int_0^1 \nabla \bbs{\psi_V-\psi^{ss}_V}(\vxv -\theta\frac{\vec{\nu}_{j}}{V})  \ud \theta\\  
&
- \tilde{\Phi}^-_j(\vxv+\frac{\vec{\nu}_{j}}{V})     p(\vxv+\frac{\vec{\nu}_{j}}{V},t)  \vec{\nu_j} \cdot \int_0^1 \nabla \bbs{\psi_V-\psi^{ss}_V}(\vxv +\theta\frac{\vec{\nu}_{j}}{V})  \ud \theta  
\Big]
\end{aligned}
\end{equation}
Taking limit $V\to +\8$, from Proposition \ref{prop1}, $C^V(t)$ converges in law to $\vx^*(t)=\argmin \psi(\vx,t)$. Using the fact that $\psi(\vx^*(t),t) = 0$, $\nabla \psi(\vx^*(t),t) = 0$, 
the left-hand-side of \eqref{tm_diss} satisfies
\begin{equation}\label{tm3.47}
\frac{1}{V}\sum_{\vxv} p(\vxv) \log \frac{p(\vxv)}{\pi(\vxv)} = \sum_{\vxv} p(\vxv,t)  \bbs{\psi_V^{ss}(\vxv) - \psi_V(\vxv,t)} \to \psi^{ss}(\vx^*(t)) - \psi(\vx^*(t),t) = \psi^{ss}(\vx^*(t)).
\end{equation}
Similarly, the right-hand-side of \eqref{tm_diss} satisfies
\begin{equation}
\begin{aligned}
 &- \sum_{\vxv}\sum_{ j=1}^M  
 \tilde{\Phi}^+_j(\vxv-\frac{\vec{\nu}_{j}}{V})     p(\vxv-\frac{\vec{\nu}_{j}}{V},t)  \vec{\nu_j} \cdot \int_0^1 \nabla \bbs{\psi_V-\psi^{ss}_V}(\vxv -\theta\frac{\vec{\nu}_{j}}{V})  \ud \theta\\
  \to &  -\sum_{ j=1}^M \Phi^+_j(\vx^*(t)) \vec{\nu}_j\cdot \bbs{\nabla\psi(\vx^*(t),t) - \nabla \psi^{ss}(\vx^*(t))} = \sum_{ j=1}^M \Phi^+_j(\vx^*(t)) \vec{\nu}_j\cdot\nabla \psi^{ss}(\vx^*(t)),
\end{aligned}
\end{equation}
thus
 we arrive at
 \begin{equation}\label{tm3.49}
\begin{aligned}
&-\frac{1}{V}\sum_{\vxv,\vyv} Q(\vyv,\vxv) p(\vyv)  \log \frac{p(\vyv) \pi(\vxv)}{\pi(\vyv) p(\vxv)} \\
& \to  
 \sum_{j=1}^M(\Phi^+_j(\vx^*(t))  - \Phi^-_j(\vx^*(t))) \,  \vec{\nu_j} \cdot  \nabla \psi^{ss}(\vx^*(t)) = -\la K(\vx) \nabla \psi^{ss}(\vx^*(t)), \nabla \psi^{ss}(\vx^*(t)) \ra. 
\end{aligned}
 \end{equation}
Notice the uniqueness of weak convergence. Combining \eqref{dissi_mic}, \eqref{tm3.47} and \eqref{tm3.49}, we conclude $\frac{\ud }{\ud t} \psi^{ss}(\vx^*(t)) = -\la K(\vx) \nabla \psi^{ss}(\vx^*(t)), \nabla \psi^{ss}(\vx^*(t)) \ra$, 
which is exactly the Lyapunov estimate \eqref{Ly} for RRE in Theorem \ref{thm_decom}.
\end{proof}
\begin{rem}\label{rem_nonconvex}
At the mesoscopic level, the energy functional  $F(p)=\sum_{\vxv}  \phi \bbs{\frac{p(\vxv)}{\pi(\vxv)}} \pi(\vxv)$ is also convex w.r.t. $p$. However, since $\phi(u)$ is convex, the nonlinear weight $\pi_{\V}(\vxv)$ in $F(p)$ drastically pick up the complicated non-convex energy landscape for chemical reactions from $\pi_{\V}(\vxv)\approx e^{-V \psi^{ss}(\vx)}$ in the large number limit. Therefore, it is natural that after the concentration of the measure in the large number limit, a non-convex energy landscape emerges. Notice also there is no such a transition from convex functional to a nonconvex function under the RRE detailed balance assumption because the corresponding probability flux is only monomial. {\blue This grouped probability flux including polynomials, which leads to non-convex energy landscape, are only linked to nonequilibrium in the specific context of chemical reactions. For   general equilibrium models in statistical physics, non convex energy landscape is common, for instance the Lagenvin dynamics with non-convex potential and Ising model of ferromagnetism.}    
\end{rem}

%\begin{rem}
%We remark for any $s\in G$, the rate function $L$ can also be characterized as
%\begin{equation}
%L(\vs,\vx)=\min_{\vr^\pm\in \bR^M,\,\,  \nu^T (\vr^+-\vr^-) =s} \sum_{j,\pm} \Phi_j^\pm(\vx) - r_j^\pm + r_j^\pm \log\frac{r_j^\pm}{\Phi_j^\pm(\vx)}.
%\end{equation}
%Indeed, with a Lagrangian multiplier $\alpha=\{\alpha_i\}_{i=1:N}$ for the constraint $\nu^T (\vr^+-\vr^-) =\vs$, the optimal $\vr^\pm$ is given by
%\begin{equation}
%\vr^+_j = \Phi^+_j(\vx) e^{\alpha \cdot \vec{\nu}_j},\quad \vr^-_j = \Phi^-_j(\vx) e^{-\alpha \cdot \vec{\nu}_j} \,\, \text{ with } \nu^T \bbs{\vr^+ - \vr^-}=\vs.
%\end{equation}
%Then one can directly verify $\alpha=\vp^*$ defined in \eqref{ts1}.
%\end{rem}

\section{Symmetric Hamiltonian: strong gradient flow, reversed least action curve, non-equilibrium enzyme reactions }\label{sec_reversal}
 {\blue  In this section, we explore the   symmetry in the mesoscopic CME and its macroscopic consequences. 
 In Section \ref{subsec_sym}, we first clarify that the Markov chain detailed balance implies a symmetric Hamiltonian; see \eqref{newS}. A proper mathematical Markov chain detailed balance condition for CME is a weaker condition than 
 the more constrained chemical version of   detailed balance, and thus includes a class of non-equilibrium enzyme reactions with three distinguished features:    multiple steady states,   nonzero steady state fluxes and positive entropy production rates at non-equilibrium steady states   \cite{kondepudi2014modern}.

   Then we study in detail two consequences of this symmetry. (I) We show  the conservative part $W(\vx)$ vanishes in the   conservative-dissipation decomposition for RRE \eqref{odeDD}; see Section \ref{sec_gf_quasi}. (II) We prove the 'uphill' least action path is   a modified  time reversed curve corresponding to the RRE, i.e., corresponding to a zero action 'downhill' path, and the associated path affinity is given by  the difference of  $\psi^{ss}$; see Proposition \ref{thm2}. We will call $\psi^{ss}$ as the energy landscape since it is a Lyapunov function of RRE and will eventually give the energy barrier of a transition path. We also provide a mesoscopic interpretation of path affinity; see Section \ref{subsec_reverse}. }
 \subsection{Markov chain detailed balance implies a symmetric Hamiltonian} \label{subsec_sym}
  We first observe  the Markov chain detailed balance condition \eqref{master_db_n} for  the mesoscopic CME is different from the more constraint RRE detailed balance \eqref{DB}. The RRE detailed balance is a very strong symmetric condition that implies the  Markov chain detailed balance condition \eqref{master_db_n}. But the latter one is a proper mathematical definition of detailed balance for a Markov process, which leads to a   symmetric Hamiltonian in a reaction system. 
  
 {\blue Observe the jumping process with generator $Q_V$ only distinct the same reaction vector $\vec{\xi}$ and then the summation in $j$ shall be rearranged in terms of all $j$ such that $\vec{\nu}_j=\pm\vec{\xi}$. Therefore, define 
 the probability flux for the same reaction vector $\vec{\xi}$ as
\begin{equation}\label{group_flux}
\Phi^+_{\xi}(\vx):=\sum_{j: \vec{\nu}_j = \vec{\xi}} \Phi^+_j(\vx) + \sum_{j: \vec{\nu}_j =- \vec{\xi}} \Phi^-_j(\vx), \quad \Phi^-_{\xi}(\vx):= \sum_{j: \vec{\nu}_j =  \vec{\xi}} \Phi^-_j(\vx) + \sum_{j: \vec{\nu}_j =  -\vec{\xi}} \Phi^+_j(\vx).
\end{equation}
}
With the grouped probability flux,  CME \eqref{rp_eq} can be recast as
\begin{equation}\label{CME_nn}
\begin{aligned}
 \frac{\ud }{\ud t} p(\vxv, t)  =  &V\sum_{\xi, \vxv- \frac{\vec{\xi}}{V}\geq 0}    \bbs{\tilde{\Phi}^+_{\xi}(\vxv-\frac{\vec{\xi}}{V})     p(\vxv-\frac{\vec{\xi}}{V},t) - \tilde{\Phi}^-_{\xi}(\vxv)     p(\vxv,t) } \\
  &+    V\sum_{\xi, \vxv+ \frac{\vec{\xi}}{V}\geq 0}  \bbs{    \tilde{\Phi}^-_{\xi}(\vxv+\frac{\vec{\xi}}{V})     p(\vxv+\frac{\vec{\xi}_{j}}{V},t) - \tilde{\Phi}^-_{\xi}(\vxv) p(\vxv,t)},
 \end{aligned}
\end{equation}
where the $\tilde{\Phi}^\pm_{\xi}$ has the same definition as $\Phi^\pm_{\xi}$ in \eqref{group_flux} but replacing $\Phi_j^\pm$ by $\tilde{\Phi}_j^\pm$.
  For any $\vxv$ and $\vec{\xi}$,  the proper mathematical definition for the Markov chain detailed balance for CME means that there exists a positive invariant measure $\pi(\vxv)$ to CME \eqref{rp_eq} such that the total forward probability steady flux from $\vxv$ to $\vxv+\frac{\vec{\xi}}{V}\geq 0$ equals the total backward one
\begin{equation}\label{master_db_n}
\begin{aligned}
\tilde{\Phi}^-_{\xi}(\vxv+\frac{\vec{\xi}}{V})     \pi(\vxv+\frac{\vec{\xi}}{V}) = \tilde{\Phi}^+_{\xi}(\vxv)\pi(\vxv), \quad \forall \vec{\xi}.
\end{aligned}
\end{equation}

On the other hand, a commonly used detailed balance condition  in biochemistry is the more constrained chemical version of detailed balance for each reaction channel, c.f., 
 \cite[Ch7, Lemma 3.1]{Whittle}, \cite[Theorem 4.5]{Anderson_Craciun_Kurtz_2010}, \cite[(7.30)]{qian_book},  
\begin{equation}\label{master_db}
\tilde{\Phi}^-_j(\vxv+\frac{\vec{\nu}_{j}}{V})     \pi(\vxv+\frac{\vec{\nu}_{j}}{V}) = \tilde{\Phi}^+_j(\vxv)\pi(\vxv), \quad \forall j.
\end{equation}
This is also known as  Whittle's Markov chain detailed balance \cite{Joshi}.
It is well known that the mesoscopic Whittle's Markov chain detailed balance \eqref{master_db} is equivalent to the macroscopic RRE detailed balance \eqref{DB}. Indeed, \textsc{Whittle} use the product of Poisson distributions with intensity $V \peq$ to  construct a detailed balanced invariant measure $\pi_V(\vxv)$ satisfying \eqref{master_db}. While from \eqref{master_db}, it is nontrivial to obtain a detailed balance steady state $\peq$;  see  \cite[Ch7, Lemma 3.1]{Whittle}. 
%For a macroscopic RRE satisfying complex balance \eqref{cDB}, \cite{Anderson_Craciun_Kurtz_2010} also use the product of Poisson distributions to construct a invariant measure $\pi_V(\vxv)$; see also Lemma \ref{lem_cDB}. 
%However, we emphasis complex balance  implies neither Whittle's Markov chain detailed balance \eqref{master_db}, nor the Markov chain detailed balance \eqref{master_db_n}. A simple counterexample is a three isomerization reactions in a cycle $A\ce{<=>}B\ce{<=>}C\ce{<=>}A$ with different forward/backward cycle rates and additional exchange with the environment $A\ce{<=>}\emptyset\ce{<=>}B$. This chemical network is complex balance due to deficiency zero but doesn't satisfy \eqref{master_db} nor \eqref{master_db_n}.

  The  proposition below shows that Markov chain detailed balance condition \eqref{master_db_n} gives raise a symmetry in the Hamiltonian    $H(\vp,\vx)$ (see \eqref{newS})
\begin{equation}
  H(\vp, \vx) = H(\nabla \psi^{ss} (\vx) -\vp, \vx), \quad \forall \vx, \vp.
  \end{equation}
   Taking $\vp=\vec{0}$, we see $\psi^{ss}$ is a steady solution to HJE \eqref{HJE2psi}.
  \begin{prop}\label{prop_mcdb}
Assume there exists a positive invariant measure $\pi_V(\vxv)$ satisfying the Markov chain detailed balance condition \eqref{master_db_n} for mesoscopic effective stochastic process. Assume the macroscopic energy landscape $\psi^{ss}(\vx):= \lim_{V\to +\8} - \frac{\log \pi_V(\vxv)}{V}$ exists, then the Hamiltonian $H(\vp,\vx)$ for macroscopic RRE satisfies the symmetry \eqref{newS} w.r.t. $\psi^{ss}(\vx)$, or equivalently
\begin{equation}\label{iff1}
e^{\vec{\xi}  \cdot \nabla\psi^{ss}(\vx)}  \Phi^+_{\xi}(\vx) = \Phi^-_{\xi}(\vx).  
\end{equation} 
\end{prop}
\begin{proof}

Let the equilibrium $\pi_V(\vxv)$ to mesoscopic effective stochastic process  be $\pi_V(\vxv) = e^{-V \psi_V^{ss}(\vxv)}$. Then the Markov chain detailed balance condition \eqref{master_db_n} implies
\begin{equation}\label{grouping1}
\frac{\pi_V(\vxv)}{\pi_V(\vxv+\frac{\vec{\xi}}{V})}  \tilde{\Phi}^+_{\xi}(\vxv) =  \tilde{\Phi}^-_{\xi}(\vxv+\frac{\vec{\xi}}{V}).
\end{equation}
Since as $V\to +\8$, $\vxv \to \vx$ and $\psi^{ss}_V(\vxv) \to \psi^{ss}(\vx)$, then
\begin{equation}
\frac{\pi_V(\vxv)}{\pi_V(\vxv+\frac{\vec{\xi}}{V})} = e^{ \vec{\xi} \cdot \int_0^1 \nabla \psi^{ss}_V(\vxv + \theta \frac{\vec{\xi}}{V}) \ud \theta} \to e^{\vec{\xi} \cdot \nabla \psi^{ss} (\vx)}.
\end{equation}
Then taking limit in \eqref{grouping1}, we obtain   \eqref{iff1}.
Using the notation for probability flux $\Phi^\pm_{\xi}(\vx)$ in \eqref{group_flux}, the Hamiltonian becomes
\begin{equation}\label{H_nn}
H(\vec{p},\vec{x}):= \sum_{\vec{\xi}}     \bbs{ \Phi^+_{\xi}(\vec{x})e^{\vec{\xi} \cdot \vec{p}}   -  \Phi^+_{\xi}(\vec{x}) +  \Phi^-_{\xi}(\vec{x})e^{-\vec{\xi} \cdot\vec{p}}   -  \Phi^-_{\xi}(\vec{x}) }.
\end{equation}
Since $\vp$ in \eqref{newS} is arbitrary, we  rearrange and take out common factor $e^{\vec{\xi}\cdot \vp}$. Then the even symmetry of $H$ in  \eqref{geneH} is equivalent to for any $\vp$
\begin{equation}
\begin{aligned}
&H(\nabla\psi^{ss}(\vx)-\vp,\vx) - H(\vp,\vx)\\
 =&  \sum_{ \vec{\xi}}  \bbs{ \Phi^+_{\xi}(\vec{x})e^{\vec{\xi} \cdot \nabla\psi^{ss}(\vec{x})} e^{-\vec{\xi} \cdot \vp}    +  \Phi^-_{\xi}(\vec{x})e^{-\vec{\xi} \cdot\nabla\psi^{ss}(\vec{x})} e^{\vec{\xi} \cdot \vp}    - \Phi^+_{\xi}(\vec{x})e^{\vec{\xi} \cdot \vp}     -  \Phi^-_{\xi}(\vec{x})e^{-\vec{\xi} \cdot \vp }   } \\
=& \sum_{ \vec{\xi} }     \bbs{ \Phi^+_{\xi}(\vec{x})e^{\vec{\xi} \cdot \nabla\psi^{ss}(\vec{x})}      -  \Phi^-_{\xi}(\vec{x})}    \bbs{e^{-\vec{\xi}   \cdot \vp}    -  e^{-\vec{\xi} \cdot \nabla\psi^{ss}(\vec{x})} e^{\vec{\xi} \cdot \vp}       }.
 \end{aligned}
\end{equation}
This means the coefficients of exponential function $e^{\vec{\xi}\cdot \vp}$ for each distinct $ \vec{\xi}$ must be same, hence \eqref{iff1} is equivalent to \eqref{newS}.
\end{proof} 
{\blue We remark the existence of $\psi^{ss}(\vx)$ constructed from a positive invariant measure $\pi_V(\vxv)$ satisfying \eqref{master_db_n} was proved in \cite{GL_vis} in the sense of an upper semicontinuous (USC) viscosity solution to stationary HJE following Barron-Jensen’s definition \cite{Barron_Jensen_1990} for USC viscosity solution. However, the uniqueness and selection principle for those USC viscosity solution is still open.}

\begin{cor}\label{thm3}
Let $H$ be the Hamiltonian for a chemical reaction defined in \eqref{H} and $\peq$ be any steady states for RRE \eqref{odex}. Then a necessary condition for the symmetry of $H$  \eqref{newS} is that  for each reaction vector $\vec{\xi}$ in the chemical reaction 
\begin{equation}\label{balance_nu}
   \Phi^+_{\xi}(\peq) = \Phi^-_{\xi}(\peq). 
\end{equation}
%This represents flux grouping within the degeneracy due to sharing a same reaction vector. We call \eqref{balance_nu} a balance within the same reaction vector.
\end{cor}
\begin{proof}
If $\peq$ is a steady state for RRE \eqref{odex}, we know $\nu \nabla \psi^{ss}(\peq)=\vec{0}$. Otherwise, take $\vx_0=\peq$ as initial data, from the estimate \eqref{Ly},
\begin{equation}
0\equiv \frac{\ud \psi^{ss}(\vx(t))}{\ud t} = -\la \nabla \psi^{ss}(\vx), K(\vx)\nabla\psi^{ss} \ra <0,
\end{equation}
due to $H$ is strictly convex in $G$ (see Lemma \ref{Hconvex}). This contradiction shows $\nabla \psi^{ss}(\peq)\in\kk(\nu)$ thus $\nu \nabla \psi^{ss}(\peq)=\vec{0}$. Therefore, evaluating \eqref{iff1} at $\peq$ yields \eqref{balance_nu}.
\end{proof}
\begin{rem} 
  As a slight generalization,  an even-symmetry of the Hamiltonian 
  w.r.t $\frac{\vq}{2}$ is
\begin{equation}\label{geneH}
H(\vp, \vx) = H(\vq(\vx)-\vp, \vx), \quad \forall \vx, \vp
\end{equation}
for some   function $\vq(\vx)$. This is equivalent to
\begin{equation} 
L(\vs,\vx) - L(-\vs, \vx) =   \vs \cdot \vq(\vx).
\end{equation}
A Hamiltonian which is quadratic in terms of the momentum $\vp$ is a special cases of \eqref{newS}. For an irreversible drift-diffusion process, $\ud x = - \vec{q} \ud t + \sqrt{2\eps}\ud B$, the corresponding Hamiltonian $H(\vp,\vx)=\vp \cdot (\vp-\vec{q})$ satisfies even-symmetry \eqref{geneH}. Another example in electromagnetism is the even-symmetry for momentum $\vp$ in Hamiltonian w.r.t the magnetic vector potential \cite{morpurgo1954time}.
\end{rem}

  In the following subsections, we study   two  consequences for a symmetric Hamiltonian.
   
  (I) Under   symmetric assumption \eqref{newS},  we provide an Onsager's strong form of gradient flow structure in terms of the energy landscape $\psi^{ss}$ in Section \ref{sec_gf_quasi}. That is to say, the conservative part $W(\vx)$ vanish in the previous conservative-dissipation decomposition for RRE \eqref{odeDD}.

    (II)  The symmetric Hamiltonian \eqref{newS} is equivalent to the time reversal symmetry in the Lagrangian $L$ upto a null Lagrangian (see \eqref{LLL})
\begin{equation} 
L(\vs,\vx) - L(-\vs, \vx) =   \vs \cdot \nabla \psi^{ss}(\vx), \quad \forall \vx, \vs.
\end{equation}
Here  $\vs \cdot \nabla \psi^{ss}(\vx)$ is a null Lagrangian, whose Euler-Lagrange equation vanishes.
Denote the time reversed curve of $\vx(t)$ as $\vxr(t) = \vx(T-t)$ with $\vxr(T)=\vxr_T=\xc$ and $\vxr(0)=\vxr_0=\xa$. Then take $\vs=\dot{\vx}^{\text{R}}(t)$ in \eqref{LLL} and integrate w.r.t time $t$ from $0$ to $T$ leads to 
the  action cost identity 
\begin{equation}
\ce{Act}(\vxr(\cdot)) - \ac(\vx(\cdot)) = \psi^{ss}(\vxr_T) - \psi^{ss}(\vxr_0);
\end{equation}
see Proposition \ref{thm2}.
We point out  the time reversed least action path is an application of the Freidlin-Wentzell theory \cite{Freidlin_Wentzell_2012} for a general exit problem to the chemical reactions while it also gives the most probable path connecting two steady states $\xa,\xb$. 
   As a well-known application of  due to symmetric Hamiltonian, for a Langevin dynamics with a potential form drift $-\nabla U$, the Freidlin-Wentzell theory \cite{Freidlin_Wentzell_2012} shows the 'uphill'  least action curve with nonzero action is exactly the time reversal of the `downhill' least action curve. The associated Hamiltonian for this Langevin dynamics is
$H(\vp,\vx) = \vp \cdot (\vp -\nabla U)$, which is symmetric 
$
H(\vp, \vx) = H(\nabla U - \vp, \vx), \quad \forall \vx, \vp.
$ 

    The idea of using this kind of symmetric Hamiltonian to find the time reversed least action curve  for some classical mechanics was first discovered by \textsc{Morpurgo} et.al. \cite{morpurgo1954time}.  
In the RRE detailed balanced case, the symmetric property w.r.t $\frac12\nabla\KL(\vx,\peq)$ of the Hamiltonian was first studied in  \cite{Dykman_Mori_Ross_Hunt_1994}, \textsc{Dykman} et.al..  With a symmetric Hamiltonian, the corresponding generalized gradient flow was first studied in \textsc{Mielke} et.al \cite{Mielke_Renger_Peletier_2014}, where the residual of the gradient flow was connected with the rate function in the large deviation principle. 
 The symmetric Hamiltonian was also used in \textsc{Bertin} \cite{bertini2002macroscopic} to study the fluctuation symmetry;   see also a comprehensive review   \cite{Bertini15} on the macroscopic fluctuation theory and recent development in   \cite{Renger_2018,  Kraaij_Lazarescu_Maes_Peletier_2020, Patterson_Renger_Sharma_2021}.    
% also discussed the time reversal, symmetry of Hamiltonian and fluctuation theory at a macroscopic scale. \cite{Bertini15} used a linear response relation with a susceptibility $\chi(\rho)$ between the current and the external field generating the fluctuation, instead of a fully nonlinear Hamiltonian discussed in this paper.

\subsection{Onsager's strong form of gradient flow in terms of energy landscape $\psi^{ss}$}\label{sec_gf_quasi}
  In this section, under the symmetric assumption \eqref{newS} for Hamiltonian, we derive a strong form of gradient flow formulation, where the steady solution $\psi^{ss}$ to the HJE serves as a free energy. This gradient flow immediately gives  vanishing of the conservative part $W(\vx)=0$ in RRE decomposition \eqref{odeDD}. 
\begin{prop}\label{prop_sgf}
Under the symmetric assumption \eqref{newS}, the RRE \eqref{odex} becomes a strong gradient flow in terms of $\psi^{ss}(\vx)$
 \begin{equation}\label{strong_GF}
 \dot{\vx} =- K(\vx) \nabla \psi^{ss}(\vx),  \quad K(\vx) = \int_0^1 \frac12\nabla^2_{pp} H(\theta \nabla \psi^{ss}(\vx)) \ud \theta.
 \end{equation}
 Particularly, for chemical reaction with RRE detailed balance \eqref{DB}, \eqref{strong_GF} reduces to
 \begin{equation}\label{strong_GF_DB}
 \dot{\vx} =- K(\vx) \nabla \KL(\vx||\peq), \quad K(\vx)=\sum_{j=1}^M \Lambda\bbs{ \Phi^+_j(\vec{x}), \Phi^-_j(\vx)} \bbs{\vec{\nu}_j \otimes \vec{\nu}_j },
 \end{equation}
 where  $\Lambda(x,y):=\frac{x-y}{\log x -\log y}$ is the logarithmic mean. 
\end{prop}
\begin{proof}
Recall the decomposition for RRE in \eqref{odeDD}.
When $H(\vp,\vx)$ satisfies the symmetric condition \eqref{newS}, we have 
 \begin{equation}\label{tms1}
 \nabla_p H(\nabla \psi^{ss}(\vx)-\vp, \vx) = - \nabla_p H(\vp,\vx), \quad \forall \vp.
 \end{equation}
 Taking $\vp=\theta \nabla \psi^{ss}$, then
  \begin{equation}
 \nabla_p H((1-\theta)\nabla \psi^{ss}(\vx), \vx) = - \nabla_p H(\theta \nabla \psi,\vx).
 \end{equation}
 Then integrating w.r.t $\theta$ implies
 $$\int_0^1 \nabla_p H(\theta \nabla \psi^{ss}(\vx),\vx) \ud \theta=0$$
 and thus the RRE is simply a strong gradient flow \eqref{strong_GF}.
Furthermore, recall symmetric nonnegative operator   $K(\vx)= \int_0^1 (1-\theta) \nabla^2_{pp} H(\theta \nabla \psi^{ss}(\vx)) \ud \theta$.  From \eqref{tms1}, we have the symmetry 
  $$\int_0^1 (1-\theta) \nabla^2_{pp} H(\theta \nabla \psi^{ss}(\vx)) \ud \theta = \int_0^1  \theta  \nabla^2_{pp} H(\theta \nabla \psi^{ss}(\vx)) \ud \theta.$$
   Thus
 \begin{equation}\label{tmK}
 K(\vx)= \int_0^1 (1-\theta) \nabla^2_{pp} H(\theta \nabla \psi^{ss}(\vx)) \ud \theta = \int_0^1 \frac12\nabla^2_{pp} H(\theta \nabla \psi^{ss}(\vx)) \ud \theta.
 \end{equation}
 
 Particularly,  
 let $\peq$ be a steady solution to \eqref{odex} satisfying RRE detailed balance condition \eqref{DB}. Then we know $\psi^{ss}(\vx)=\KL(\vx||\peq)$ and $\vec{\nu}_j \cdot \log \frac{\vx}{\peq} = \log \frac{\Phi^-_j(\vec{x})}{\Phi^+_j(\vx)}$. Thus the $K$-matrix in \eqref{tmK}  reduces to
\begin{equation}
K(\vx)= \sum_{j=1}^M     \frac{\bbs{ \Phi^-_j(\vx)-\Phi^+_j(\vx)}\vec{\nu}_j\otimes \vec{\nu}_j}{\vec{\nu_j} \cdot \nabla \psi^{ss}(\vx)} =  \sum_{j=1}^M \Lambda\bbs{ \Phi^+_j(\vec{x}), \Phi^-_j(\vx)} \bbs{\vec{\nu}_j \otimes \vec{\nu}_j } .
\end{equation}
 \end{proof}
The above formula \eqref{strong_GF_DB}  can also be written as $\sum_{j=1}^M \Phi_j(\peq) \Lambda\bbs{\frac{\Phi^+_j(\vec{x})}{\Phi_j(\peq)}, \frac{\Phi^-_j(\vec{x})}{\Phi_j(\peq)}  } \bbs{\vec{\nu}_j \otimes \vec{\nu}_j }$, which is known as biochemical conductance in biochemistry \cite{qian2005thermodynamics}. This exactly recovers
  the  well-known strong gradient flow represented by the logarithmic mean $\Lambda(x,y)=\frac{x-y}{\log x -\log y}$   \cite{Hanggi_Grabert_Talkner_Thomas_1984} (\cite[Theorem 2.2]{maas2020modeling}) for RRE detailed balance case.

\begin{rem}
To fit into  more general biochemical  reactions such as gene switch \cite{roma2005optimal},  we 
give a slightly more general symmetric condition for $H$ so that the strong gradient flow structure still holds.
Assume there exists $\alpha(\vx)>0$ such that
\begin{equation}
H(\vp, \vx)= H(\nabla \psi^{ss}(\vx)-\alpha(\vx) \vp, \vx)
\end{equation}
and $H$ still satisfies $H(\vec{0}, \vx)=H(\nabla \psi^{ss}(\vx), \vx)=0$.
Then we have
\begin{equation}\label{tm_gf3}
\frac{\ud}{\ud t} \vx = \nabla_p H(\vec{0}, \vx) = - \alpha(\vx) \nabla_p H(\nabla \psi^{ss}(\vx(t)), \vx(t) ) =-\alpha(\vx) \bbs{\nabla_p H( \vec{0} , \vx(t) ) + 2K(\vx) \nabla \psi^{ss}(\vx)},
\end{equation}
which yields a gradient flow structure
\begin{equation}
\frac{\ud}{\ud t} \vx =  \nabla_p H(\vec{0}, \vx) = -\frac{2\alpha(\vx)}{1+\alpha(\vx)}  K(\vx) \nabla \psi^{ss}(\vx).
\end{equation}
\end{rem}

\subsection{Reversed least action curve and path affinity described by energy landscape $\psi^{ss}$}\label{subsec_reverse}

In order to  study the transition path between two   states $ \xa$ and $ \xb$, we now characterize the time reversed solution to RRE \eqref{odex}. 
Let $\vx(t)$ be the 'downhill' solution to RRE \eqref{odex} with $\vx(0)=\xc$ and $\vx(T)=\xa$ for some finite time $T$. Notice this requires $\xa, \xc$ are not steady states to RRE \eqref{odex}. However, $\xa, \xc$ can be in a small neighborhood of steady states and then taking time goes to infinity gives the transition path between two stable states passing through a saddle point.  Then by Proposition \ref{lem:least}, we know $\vx(t)$ is a least action solution with action cost $\ac(\vxr(\cdot))=0$ in \eqref{A}. 
 We define the time reversed curve for $\vx(t),\vp(t)$ by
\begin{equation}\label{vxr}
\vxr(t) = \vx(T-t), \quad \vpr(t) = \vp(T-t), \quad 0\leq t\leq T.
\end{equation}
Then we know $\vxr$ satisfies
\begin{equation}\label{vxr_ode-K}
\dot{\vx}^{\R} = \sum_{j=1}^M \vec{\nu}_j \bbs{\Phi^-_j(\vxr) - \Phi^+_j(\vxr)}, \quad \vxr_0=\vx(T)=\xa,\,\, \vxr_T=\vx(0)=\xc.
\end{equation}

The following Proposition \ref{thm2} states that   the time reversed solution $\vxr$ with a modified reversed momentum 
\begin{equation}\label{mom}
\vprr(t) = \nabla\psi^{ss}(\vxr(t)) -\vpr(t)
\end{equation}
 is  a 'uphill' least action solution from $\vxr_0=\xa$ to $\vxr_T=\xc$ but with a non-zero  action $\ac(\vxr(\cdot))$.

\begin{prop}\label{thm2}
Given a Hamiltonian $H(\vp,\vx)$ satisfying \eqref{newS}, suppose $L(\vs,\vx)$ is its convex conjugate.
Let $\vx(t),\vp(t)$ be a least action solution for the action functional $\ac(\vx(\cdot))=\int_0^T L(\dot{\vx},\vx) \ud t$ starting from $\vx(0)=\xc$ and ending at $\vx(T)=\xa$. Then for the time reversed  solution $\vxr(t), \vpr(t)$ defined in \eqref{vxr}, we know
\begin{enumerate}[(i)]
\item  the modified time reversed solution $\vxr(t)$, $\vprr(t) = \nabla\psi^{ss}(\vxr(t)) -\vpr(t) $ is  a least action curve starting from $\vxr_0=\vx(T)=\xa$, ending at $\vxr_T=\vx(0)=\xc$ and satisfies the Hamiltonian dynamics
\begin{equation}
\frac{\ud}{\ud t} \vxr = \nabla_p H(\vprr, \vxr), \quad \frac{\ud}{\ud t} \vprr = - \nabla_x H(\vprr , \vxr);
\end{equation}
\item   the corresponding action cost for the least action curve $ \vxr(t) $ is given by
\begin{equation}\label{Kaa}
\ac(\vxr(\cdot))  = \ac(\vx(\cdot)) +  \psi^{ss} (\vxr_T) - \psi^{ss} (\vxr_0).
\end{equation}
\end{enumerate}
\end{prop}
 
\begin{proof}
First, recall the definition of $H$ in \eqref{H} which satisfies \eqref{newS}. Then regarding $\vp,\vx$ as independent variables in $H(\nabla \psi^{ss}(\vx)-\vp, \vx) \equiv H(\vp,\vx)$,  taking derivatives, we directly have following identities
\begin{align}
&\nabla_p H(\vp,\vx) \equiv -\nabla_p H(\nabla \psi^{ss}(\vx)-\vp, \vx), \label{HH2}
 \\
& \nabla_x H(\vp,\vx) \equiv \nabla_x H (\nabla \psi^{ss}(\vx)-\vp,\vxr) +  \nabla^2 \psi^{ss}(\vx){\nabla_p H(\nabla \psi^{ss}(\vx)-\vp,\vx)}. \label{HH3}
\end{align}

Second,  from \eqref{HH2}, we have
  \begin{equation}\label{hl1}
\frac{\ud}{\ud t} \vxr(t) = -\nabla_p H(\vpr(t),\vxr(t))= \nabla_p H(\vprr(t), \vxr(t)).
\end{equation}

Third, 
by the definition of modified reversed momentum  $\vprr$ and  \eqref{HH3}, we have
 \begin{equation}\label{hl2}
 \begin{aligned}
 \frac{\ud}{\ud t} \vprr(t) =& -  \frac{\ud}{\ud t} \vpr(t) +  \nabla^2 \psi^{ss}(\vxr(t)){\dot{\vx}^{\R}(t)}    =  \dot{\vp}\Big|_{T-t} + \nabla^2 \psi^{ss}(\vxr(t)) {\nabla_p H(\vprr(t),\vxr(t))} \\  =& -\nabla_x H(\vpr(t),\vxr(t)) +  \nabla^2 \psi^{ss}(\vxr(t)){\nabla_p H(\vprr(t),\vxr(t))}  =   - \nabla_x H(\vprr(t) , \vxr(t)),
 \end{aligned}
 \end{equation}

Fourth, combining \eqref{hl1} and \eqref{hl2}, we know $(\vxr(t), \vprr(t))$ solves a Hamiltonian dynamics. Notice for any Hamiltonian trajectory $(\vx(t),\vp(t))$, the Lagrangian can be expressed as
\begin{equation}
L(\dot{\vx}(t), \vx(t)) = \vp(t)\cdot \dot{\vx}(t) -H(\vp(t),\vx(t)).
\end{equation}
From \eqref{newS}, one can directly compute the action cost along $\vxr$
 \begin{equation}
 \begin{aligned}
\ac(\vxr(\cdot)) =& \int_0^T   L(\dot{\vx}^{\R}(t),\vxr(t))  \ud t= \int_0^T \bbs{\vprr(t) \cdot \dot{\vx}^{\R}(t) - H(\vprr(t),\vxr(t))  }\ud t\\
=&\int_0^T \bbs{ \nabla \psi^{ss}(\vxr) \cdot \frac{\ud}{\ud t} \vxr(t) -\vpr(t) \cdot  \nabla_p H(\vprr,\vxr)    - H(\vprr(t),\vxr(t)) } \ud t.
\end{aligned}
 \end{equation}
where we used \eqref{mom}. Then by  \eqref{HH2}, we obtain
 \begin{equation}\label{tmAA}
 \begin{aligned}
\ac(\vxr(\cdot))=&\int_0^T \bbs{  \nabla\psi^{ss}(\vxr) \cdot \frac{\ud}{\ud t} \vxr(t) + \vpr(t) \cdot  \nabla_p H(\vpr,\vxr)  - H(\vpr(t),\vxr(t)) } \ud t \\
=& \int_0^T L(\dot{\vx}(T-t), \vx(T-t) ) \ud t + \int_0^T  \frac{\ud}{\ud t} \psi^{ss}(\vxr(t)) \ud t\\
=& \ac(\vx(\cdot)) +\psi^{ss} (\vxr_T) - \psi^{ss} (\vxr_0),
 \end{aligned}
 \end{equation}
which concludes \eqref{Kaa}.
\end{proof}
The proof for the minimum cost  only relies on the observation for null Lagrange $L(\dot{\vx},\vx) - L(-\dot{\vx}, \vx)$ in \eqref{LLL}.
The statement (ii) for the reversed  action cost can be understood as a path affinity describing in which direction the chemical reaction (or a general nonlinear dynamics) proceed. Precisely, this affinity is  given by  $\ac(\vxr(\cdot)) - \ac(\vx(\cdot)) =  \psi^{ss} (\vxr_T) - \psi^{ss} (\vxr_0)$.
 In the case of $\vpr\equiv \vec{0}$, the 'downhill' path corresponds to the  solution to RRE \eqref{odex} with action cost $\ac(\vx(\cdot))=0$.  In this case, the reversed action cost is
\begin{equation}\label{Ebb} 
 \ac(\vxr(\cdot))  =  \psi^{ss} (\vxr_T) - \psi^{ss} (\vxr_0) =\psi^{ss} (\xc) - \psi^{ss} (\xa). 
 \end{equation}

Usually, the steady solution $\psi^{ss}(\vx)$ to HJE  is known as the quasipotential $V(\xc;\xa)$ upto a constant for the exit problem in the Freidlin-Wentzell theory \cite{Freidlin_Wentzell_2012} in the sense that for $\xc$ in the basin of attraction
 \begin{equation}\label{quasipotential}
\begin{aligned}
\psi^{ss}(\xc) - \psi^{ss}(\xa)= V(\xc;\xa)  =&\inf_{T>0,\,\, \vxr(0)=\xa, \, \vxr(T)=\xc}   \int_0^T L(\dot{\vx}^{\R}(t),\vxr(t)) \ud t.
 \end{aligned}
\end{equation}

 Below we explain why the modified least action path  with associated action cost in \eqref{Ebb} for a given fixed time $T$ coincides with the quasipotential $\psi^{ss}(\vx)$ in \eqref{quasipotential}, where the time $T$ is also a variable to minimize. On the one hand,  the assumption in Proposition \ref{thm2} that there exists a forward least action curve from $\vx(0)=\xc$ to $\vx(T)=\xa$ already gives the curve trajectory with a fixed reaching time $T$. On the other hand, as proved above, the symmetric property ensures that the modified time reversal yields exactly the 'uphill' least action curve with the same trajectory and the same reaching time $T$, but only along a reversed time order. That is to say, the associated action cost in \eqref{Ebb} with the reaching time $T$ also implies the optimal time in \eqref{quasipotential} is exactly the reaching time $T$ in the 'downhill' solution to the RRE. 
 The minimization problem \eqref{quasipotential} is also called Maupertuis's principle of least action in an undefined time horizon  and we will formulate it as an optimal control problem in a undefined (infinite) time horizon in Section \ref{sec_control}  without the detailed balance.

The quasipotential is a generalized potential function to quantify the energy barrier for the transitions starting from one stable state to another one. 
The concept of energy barrier (a.k.a activation energy) was initialed by \textsc{Arrhenius} in 1889 who related  the transition rate $\mathcal{K}$ to the  free energy difference via Arrhenius's law $\mathcal{K}\propto e^{-\frac{\delta E}{\eps}}$ with a noise parameter $\eps$ indicating the thermal energy. Using a Langevin dynamics starting from $\vx$ in a basin of attraction, Kramers   estimated the   mean first passage time $\tau(\vx)=\bE^{\vx}(\tau_C)$, i.e., the expectation of the stopping time  defined as the first hitting time  on the boundary of the basin of attraction, which gives an explicit formula for transition rate $\mathcal{K}=\frac{1}{\tau(\vx)}\propto e^{-\frac{\delta E}{\eps}}$. At a rigorously mathematical level,   the large deviation theory for general stochastic processes gives the estimate for the reaction rate by computing the probability of the exit problem from the basin of attraction via the good rate functional $\bP\{C^V\in \Gamma\} \approx e^{-\frac{\inf_{\vx\in\Gamma} \ac(\vx(\cdot))}{\eps}}$; see comprehensive studies in the Freidlin-Wentzell theory \cite{Freidlin_Wentzell_2012} and precise statement in \cite[Theorem 1.6]{Dembo18} for chemical reactions. Therefore, the least action cost computed in \eqref{Ebb} gives the energy barrier for the transition path problem in chemical reactions and this is why we call $\psi^{ss}$ the energy landscape for RRE \eqref{odex}.  
  
    In Proposition \ref{thm2}, the assumption that there exists a forward least action curve from $\vx(0)=\xc$ to $\vx(T)=\xa$ already limits the curve within the stable basin of $\xa$ and not passing beyond the separatrix (boundary of the basin). Within the basin of attraction of stable state $\xa$,  the globally defined energy landscape $\psi^{ss}$ coincides with the quasipotential upto a constant.  
  On the other hand,  the transition path connecting $\xa$ to $\xb$ and passing through some saddle point $\xc$ is one of the most important scientific questions. In this case, the most probable path is piecewisely defined by finding the least action curve from $\xa$ to $\xc$ and then from $\xc$ to $\xb$; see Section \ref{subsec_oc}. The energy barrier for the rare transition shall be computed piecewisely, for instance, the energy barrier for transition $\xa$ to $\xb$ is given by the 'uphill' action cost $\psi^{ss} (\xc) - \psi^{ss} (\xa)$ plus zero action cost for the 'downhill' curve. In practice, given the energy landscape $\psi^{ss}$, there are many methods such as string method \cite{string} to find the saddle point $\xc$.

\subsubsection{Mesoscopic interpretation of path affinity}\label{subsec_aff}
 Denote $(\Omega, \mathcal{F}, \bP)$ as a probobility space. For the large number process $C^V(t)$, define  a new random variable on the Skorokhod space $D([0,T]; \bR^N_+)$ through the froward trajactory $\vx^V(\cdot)$ as
 \begin{equation}\label{zzz}
 Z_{\V}  := \log \frac{\ud \bP^V_{[0,T]}}{\ud \bP^V_{[T,0]}}(\vx^V(\cdot )) 
 %=  \log \frac{ \bP^V_{[0,T]}(\vx^V(\cdot))}{ \bP^V_{[0,T]}(\vx_R^V(\cdot ))},
 \end{equation}
 where $\vx^V_R(t):=\vx^V(T-t)=: (\mathcal{R} \circ \vx^V)(t) $ is the time reversed trajectory and $\mathcal{R}$ is the reversed operator.
 Here $\bP^V_{[0,T]}$ is the probability measure on path space $D([0,T]; \bR^n_+)$ defined via pushforward of $\bP$ by $\vx^V(\cdot)$, i.e., $\vx^V(\cdot)_\#\bP$ and
with the reversed operator $\mathcal{R}$, $\bP^V_{[T,0]}$ is the probability measure on path space $D([T,0]; \bR^N_+)$ defined via pushforward $(\mathcal{R}\circ \vx^V)_\#\bP$. 
The ratio $\frac{\ud \bP^V_{[0,T]}}{\ud \bP^V_{[T,0]}}(\vx^V(\cdot ))$  is the Radon-Nikodym derivative. 
This $Z_{\V}$ is known as the fluctuating entropy production rate of any forward trajectory $\vx^V(\cdot)$, with respect to its reversed trajectory. $Z_{\V}$ was also historically introduced by Onsager as a dissipation function \cite{evans2002fluctuation, Yang_Qian_2020}.

For simplicity in presentation, we assume there is probability density (with the same notations as the distribution) for $\ud \bP^V_{[0,T]}$ then $\ud \bP^V_{[0,T]}(\vx^V )=\bP^V_{[0,T]}(\vx^V ) \ud x$  and similarly we have $\ud \bP^V_{[T,0]}(\vx^V )=\bP^V_{[0,T]}(\mathcal{R}\circ \vx^V ) \ud x$.
Fix starting point $\vx_0=\xb$ and ending point  $\vx_T=\xa$. We define a subset $\Gamma\subset D([0,T]; \bR^N_+)$ as all trajectories $\vx(\cdot)$ starting from $\vx_0$ and ending at $\vx_T$. Then by the large deviation principle in  \cite[Theorem 1.6]{Dembo18} (Theorem 2.6 above), we have
\begin{equation}\label{la1}
\lim_{V\to +\8}\frac{1}{V}\log { \bP^V_{[0,T]}(\vx^V(\cdot)\in \Gamma)}  = -\min_{\vx(\cdot)\in\Gamma\cap AC([0,T]; \bR^N_+)} \int_0^T L(\dot{\vx}(t), \vx(t)) \ud t = \ac(\vx^*(\cdot)),
\end{equation}
where $\min$ is achieved at an interior point $\vx^*(\cdot)$, i.e., the least action curve among $\Gamma$ satisfying Euler-Lagrange equation \eqref{EL}.
On the other hand, for $\bP^V_{[T,0]}$ defined on reversed trajectory above, we have
\begin{equation}
\bP^V_{[T,0]} \{ \vx^V \in \Gamma \} = \bP^V_{[0,T]} \{\vx^V_R \in \Gamma\} =  \bP^V_{[0,T]} \{\vx^V \in \Gamma^R\},
\end{equation}
where $\Gamma^R\subset D([T,0]; \bR^n_+)$ is the set of any trajectories $\vx(\cdot)$ starting from $\vx_T$ and ending at $\vx_0$. Then we have
 \begin{equation}\label{la2}
 \begin{aligned}
\lim_{V\to +\8}\frac{1}{V}\log &{ \bP^V_{[T,0]}(\vx^V(\cdot)\in \Gamma)}= \lim_{V\to +\8}\frac{1}{V}\log { \bP^V_{[0,T]}(\vx^V(\cdot)\in \Gamma^R)}  = -\min_{\vx(\cdot)\in\Gamma^R\cap AC} \int_0^T L(\dot{\vx}(t), \vx(t)) \ud t\\
= &-\min_{\vx(\cdot)\in\Gamma \cap AC} \int_0^T L(\dot{\vx}^{\R}(t), \vxr(t)) \ud t = -\min_{\vx(\cdot)\in\Gamma \cap AC} \int_0^T L(-\dot{\vx}(t), \vx(t)) \ud t = \ac(\vx^*_R),
\end{aligned}
\end{equation}
where $\vx^*_R(\cdot)$ is the least action curve among $\Gamma^R$ satisfying Euler-Lagrange equation \eqref{EL}.

Under symmetric assumption \eqref{geneH}  for $H$,  Proposition \ref{thm2} tells us $\vx^*_R$ is exactly the time reversal of $\vx^*$, i.e., $\vx^*_R(t)=\vx^*(T-t)$. Thus plugging the least action curve $\vx^*$ into  \eqref{la1} and \eqref{la2} and taking difference,   we use the relation \eqref{LLL} to derive
\begin{equation}
\begin{aligned}
  &\lim_{V\to +\8}\frac{1}{V}\log { \bP^V_{[0,T]}(\vx^V(\cdot)\in \Gamma)} - \lim_{V\to +\8}\frac{1}{V}\log { \bP^V_{[T,0]}(\vx^V(\cdot)\in \Gamma)} \\
 =&  \int_0^T \bbs{ L(-\dot{\vx}^*(t), \vx^*(t)) -L(\dot{\vx}^*(t), \vx^*(t)) } \ud t = - \int_0^T  \dot{\vx}^* \cdot \nabla \psi^{ss} (\vx^*) \ud t = \psi^{ss}(\vx_0) -\psi^{ss}(\vx_T).
 \end{aligned}
\end{equation}
Therefore, the symmetric Hamiltonian  implies that in the large number limit,  the fluctuating entropy production rate $Z_{\V}$ defined in \eqref{zzz} only depends on the given initial $\vx_0$, end states    $\vx_T$ and its value is given by the path affinity $\ac(\vxr(\cdot)) - \ac(\vx(\cdot)) =  \psi^{ss} (\vx_0) - \psi^{ss} (\vx_T).$

As a special example, 
when RRE \eqref{odex} is   detailed balanced, $H$ satisfies the symmetry \eqref{newS} with $\psi^{ss}(\vx)=\KL(\vx||\peq ).$
Then the modified reversed momentum $\vprr(t)$ is
$
\vprr(t) =  \log \frac{\vxr(t)}{\peq} - \vpr(t).
$ 
The corresponding   minimum action can be calculated as
 \begin{equation}
 \ac(\vxr(\cdot)) = \ac(\vx(\cdot)) + \int_0^T \frac{\ud}{\ud t} \KL(\vxr(t)||\peq) \ud t =\ac(\vx(\cdot)) +   \KL(\vxr_T||\peq) - \KL(\vxr_0||\peq).
 \end{equation}
For a special case that
  $\vx(t)$ being the solution to RRE \eqref{odex} with $\vx(0)=\xb$ and $\vx(T)=\xa$ for some finite time $T$, then $\vx(t)$ is a least action solution   with zero action cost  (the 'downhill' path). The time reversed curve $\vxr(t)=\vx(T-t)$ is the most probable path (the 'uphill' path) from $\xa$ to $\xb$ with action cost $\KL(\xb||\peq) - \KL(\xa||\peq)$.
 As time evolves, the  solution to the RRE   and the reversed one stay at the same level set of the Hamiltonian $H\equiv 0$.

\subsection{Non-equilibrium example: a bistable Schl\"ogl catalysis model}\label{subsub_sch}
In this section, we show the symmetric Hamiltonian, brought by the  Markov chain detailed balance, does include a class of non-equilibrium enzyme reactions due to the flux grouping property \eqref{group_flux}.  This type of non-equilibrium enzyme reactions plays important roles in a living cell, for instance in the phosphorylation-dephosphorylation with 2-autocatalysis described in Appendix \ref{app_2d}, the enzyme plays as an intermediate  species but dramatically lower the energy barrier.   The flux grouping  within a same reaction vector leads to multiple steady states and   nonzero steady state fluxes that maintains a ecosystem.
We will  illustrate the idea of constructing the optimally controlled time reserved solution in a well-known bistable reaction in an open system.  Consider Schl\"ogl catalysis model \cite{Schlogl_1972} with environment $\emptyset$, chemostats $A$, $B$ and internal specie $X$
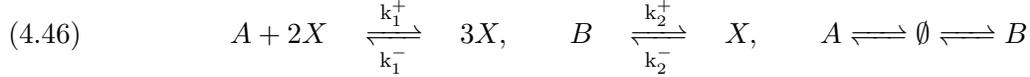
\begin{equation}
A + 2X  \quad \ce{<=>[k^+_1][k^-_1]} \quad   3 X, \qquad  B \quad \ce{<=>[k^+_2][k^-_2]} \quad    X, \qquad A \ce{<=>} \emptyset \ce{<=>}  B
\end{equation}
where $k_1^+,k_1^-,k_2^+,k_2^->0$ are reaction rates. $X$ plays a role of enzyme in biological system and is usually called an intermediate or an autocatalyst \cite{bierman1954studies, van1992stochastic}.  
Denote the concentration of $X$ as $x$ and the concentration of $A,B$ as $a,b$. Here $a,b$ are  assumed to be constants that are sustained by the environment. Below,  we will see there is a bifurcation ratio of $a,b$ to classify the reaction rate system as a non-equilibrium reaction system, except for a special value of $\frac{a}{b}$.  We will use the Schl\"ogl model to describe non-equilibrium steady state behaviors, which have three typical   features: (i) multiple steady states; (ii) nonzero steady state fluxes; (iii) positive entropy production rates at non-equilibrium steady states. We first observe  the flux grouping degeneracy in the same reaction vector,  which usually exists in enzyme reactions such as the Michaelis-Menten kinetics; see Appendix \ref{app_2d} for a realistic Phosphorylation-dephosphorylation model. This is the main reason leading to coexistence of multiple steady states. The RRE for the Schl\"ogl model indeed can be formulated as a Ginzburg-Landau model with double well potential. Based on   Proposition \ref{thm2}, we study the transition path between two stable steady states passing through an unstable state with an associated energy barrier. We will see the energy barrier is not computed from the Ginzburg-Landau double well potential but rather the energy landscape given by the HJE steady solution. That is to say, a simple diffusion approximation can not be used to compute transition path problem in the large deviation regime; c.f. \cite{Assaf_Meerson_2017}.

In detail,
given $a,b>0$, the forward/backward fluxes of these two reactions are
\begin{equation}\label{fluxS}
\Phi_1^+(x) = k_1^+ a x^2, \quad \Phi_1^-(x) = k_1^- x^3, \quad \Phi_2^+(x) = k_2^+ b, \quad \Phi_2^-(x) = k_2^- x.
\end{equation}
Since  $\nu_{11}=1, \nu_{21}=1$ for internal species $X$, we have the macroscopic RRE for $x$
\begin{equation}\label{odeS}
\dot{x} = k_1^+ ax^2 - k_1^- x^3 + k_2^+ b - k_2^- x =: f(x).
\end{equation}
Given $a,b>0$, the steady states of \eqref{odeS} is solved by $f(x_s)=0$, while the RRE detailed balance condition reads
\begin{equation}
k_1^+ ax_s^2 = k_1^- x_s^3, \quad k_2^+ b = k_2^- x_s.
\end{equation}
This means only when the ratio $\frac{a}{b}=\frac{k_1^- k_2^+}{k_1^+k_2^-}$, there exists a unique detailed balanced equilibrium $x_s$. Indeed, in this case, $f(x)=\bbs{\frac{k_1^-}{k_2^-} x^2 +1 }\bbs{k_2^+ b -k_2^- x}$.  Especially, there is no external flux between chemostats $A, B$ and the environment and thus at the detailed balanced equilibrium, the system can be regarded as a closed system. Notice in this simple example, the complex  balance condition \eqref{cDB} is same as detailed balance condition \eqref{DB}, so Schl\"ogl model \eqref{odeS} is a non-detailed/complex balanced RRE system. One can also check the deficiency of this model (as defined in Section \eqref{subsec_de}) is $\delta= 4 - 2-1=1.$ 

In general, assume $\frac{a}{b}\neq \frac{k_1^- k_2^+}{k_1^+k_2^-}$ and $f(x)$ has three zero points. Two of them are stable steady states while one of them is an unstable steady state. We denote the corresponding antiderivative of $-f(x)$ as $\Upsilon(x)$. $\Upsilon(x)$ is a Ginzburg-Landau double well potential with two stable states, which determines the bifurcation and the first order phase transition.  At a non-equilibrium steady states $x^s$, by elementary calculations, there is nonzero steady flux 
$f_1(x^s):=\Phi_1^+(x^s)-\Phi_1^-(x^s) < 0 < f_2(x^s) := \Phi_2^+(x^s)-\Phi_2^-(x^s)$. The nonzero steady flux maintains a source-production circulation $\emptyset \lra B \lra X \lra A \lra \emptyset$ in this open ecosystem at either one of the non-equilibrium steady states. This ecosystem continues exchanging both  chemical energy and   materials with its environment. Compared with equilibrium, the open system continues converting the chemical energy into heat at  non-equilibrium steady states.   The positive entropy production rate is 
$$T\dot{S}=k_{\B}T(\Phi_1^+(x^s)-\Phi_1^-(x^s))\log \frac{\Phi_1^+(x^s)}{\Phi_1^-(x^s)} + k_{\B}T(\Phi_2^+(x^s)-\Phi_2^-(x^s))\log \frac{\Phi_2^+(x^s)}{\Phi_2^-(x^s)} > 0,$$
 which characterizes the irreversible process.

With this simple non-equilibrium reaction system, we illustrate how to find a transition path as a least action curve (in the sense of the large deviation theory in \cite[Theorem 1.6]{Dembo18}) connecting two non-equilibrium steady states. In general, $\Upsilon(x)$ is not symmetric but without loss of generality we simply assume the following symmetric form
\begin{equation}
\Upsilon(x) = \frac{k_1^-}{4} [(x-\theta)^2 - r^2]^2.
\end{equation}
This typical symmetric double well potential has two stable local minimums $\theta\pm r$ and an unstable critical point $\theta$ provided $(k_1^+)^2 a^2 - 3 k_1^- k_2^->0$.
Then $f(x) = - \pt_x \Upsilon(x)$ implies $\theta = \frac{a k_1^+}{2 k_1^-}$, $r=\frac{\sqrt{a^2(k_1^+)^2 -3k_1^- k_2^-}}{\sqrt{3}k_1^-}$. Here $k_2^+$ will be a slaver parameter to ensure symmetry.
With the fluxes in \eqref{fluxS}, the Hamiltonian is
\begin{equation}
\begin{aligned}
H(p,x)=&[\Phi_1^+(x) + \Phi_2^+(x)](e^p -1)+[\Phi_1^-(x)+ \Phi_2^-(x)](e^{-p}-1)\\
=&(k_1^+ ax^2 + k_2^+ b)(e^p -1) + (k_1^- x^3 + k_2^- x)(e^{-p}-1),
\end{aligned}
\end{equation}
which is strictly  convex w.r.t $\vp$. This degeneracy is due to the same reaction vector yields a flux grouping. Denote 
$$\alpha(x):= \frac{\Phi_1^-(x)+ \Phi_2^-(x)}{\Phi_1^+(x) + \Phi_2^+(x)} = \frac{k_1^- x^3 + k_2^- x}{k_1^+ ax^2 + k_2^+ b}.$$ Then we know $p=0$ or $p=\log \alpha(x)$ are solutions to $H(p,x)=0$. By elementary calculations, one can verify
\begin{equation}
\nabla_p H(p,x)\big|_{p=0} = f(x) = - \nabla_p H(p,x) \big|_{p=\log \alpha}, 
\end{equation}
and the symmetry w.r.t $\frac{\log \alpha}{2}$
\begin{equation}
H(p,x)=H(\log \alpha -p, x), \quad \forall x, p.
\end{equation}
Assume the 'downhill' RRE starts from some initial state $\vx(0)$ and then goes to one stable state $\vx(T_\eps)\to\theta-r$. Here $\eps>0$ is the $\eps$-neighborhood of $\theta-r$ and $T_\eps\to +\8$ as $\eps\to 0^+$. 
Then by Proposition \ref{thm2}, the modified time reversed solution $\vxr(t)$ and $\vprr(t) = -\vpr(t)+\vq(\vxr(t))=\log \alpha $ is  still a least action solution starting from $\vxr_0=\vx(T_\eps)\approx \theta-r$ and ending at $\vxr_{T_\eps}=\vx(0)$ and satisfy the same Hamiltonian dynamics.

In this 1D example, there always exists a potential function $\psi^{ss}(x)$ such that $\log \alpha(x) =  \pt_x \psi^{ss}(x)$. This $\psi^{ss}$ is the steady solution to the HJE, and is a Lyapunov function to RRE \eqref{odeS}.
From Proposition \ref{thm2},   the least action value  is given by
 the path affinity  
 \begin{equation}
 \ac(\vxr(\cdot))  =  \psi^{ss} (\vxr_T) - \psi^{ss} (\vxr_0) = \psi^{ss}(x(0))-\psi^{ss}(\theta-r).
 \end{equation}

 However, we point out energy landscape $\psi^{ss}(x)$ computed from  $\log \alpha =\pt_x \psi^{ss}$ is not same as the double well potential $\Upsilon(x)$ in RRE \eqref{odeS}.  They are   two different Lyapunov functions but have same increasing/decreasing regimes. Indeed, Since $\pt_x \psi^{ss}(x)=\log \alpha$ is the steady solution to the HJE, so by Proposition \ref{prop1}, $\psi^{ss}(\vx)$ is a Lyapunov function satisfying
 \begin{equation}
 \frac{\ud \psi^{ss}} {\ud t} = \pt_x \psi^{ss} \dot{x} = - \pt_x \Upsilon(x) \pt_x \psi^{ss}(x) \leq 0.
 \end{equation}
  Therefore, $\pt_x \Upsilon$ and $\pt_x \psi^{ss}$ has the same monotonicity. Although $\psi^{ss}$ is also a double well potential with the same stable/unstable points as $\Upsilon$,   the affinity of the path is given by the difference in terms of $\psi^{ss}$ instead of $\Upsilon$. At each basin of attraction of stable states, $\psi^{ss}$ coincides with the so-called quasipotential, as explained in \eqref{Ebb}.

Below, we also study the effects of perturbations of chemostats in this sustained non-equilibrium system, specifically, the linear response of the energy landscape $\psi^{ss}$ to a perturbation of the external flux represented by chemostats. Denote the concentration of chemostats as a generic parameter $b$ with perturbation $\eps \tilde{b}$, $\eps\ll 1$. Then the energy landscape satisfies the steady HJE with parameter $b$
\begin{equation}
H(\nabla \psi^{ss}(\vx), \vx, b)=0 = H(\nabla (\psi^{ss}(\vx)+\eps \tilde{\psi}^{ss}(\vx)), \vx, b+\eps \tilde{b}  ).
\end{equation}
Here $\psi^{ss}(\vx)+\eps \tilde{\psi}^{ss}(\vx)$ is the new energy landscape under perturbed chemostats.
Then Taylor's expansion w.r.t. $\eps$ gives the leading order equation
\begin{equation}
\nabla_p H (\nabla \psi^{ss}(\vx), \vx, b) \cdot \nabla \tilde{\psi}^{ss}(\vx) + \nabla_b H(\nabla \psi^{ss}(\vx), \vx, b) \cdot \tilde{b} =0.
\end{equation}
If further assume the symmetry \eqref{newS} for Hamiltonian, the response energy landscape perturbation is given by
\begin{equation}
\frac{\ud \tilde{\psi}^{ss}(\vx(t))}{\ud t} = \nabla \tilde{\psi}^{ss}(\vx(t)) \cdot \dot{\vx} = \nabla_b H(\nabla \psi^{ss}(\vx), \vx, b) \cdot \tilde{b}.
\end{equation}
The rigorous justification for this linear response relation can follow  the method in \textsc{Hairer, Majda}  
\cite{Hairer_Majda_2010}.

\section{Existence of $\psi^{ss}$ and diffusion approximation for transition paths}\label{sec_control}
We have shown in previous sections that the  stationary solution $\psi^{ss}(\vx)$ to HJE \eqref{HJE2psi}  serves as the energy landscape of chemical reactions, facilitates the conservative-dissipative decomposition for RRE, and also  determines both the energy barrier and thermodynamics of chemical reactions. For a detailed/complex balanced RRE, we simply have a convex stationary solution
 $\psi^{ss}(\vx)=\KL(\vx||\peq)$; see Lemma \ref{lem_cDB}. {\blue For general chemical reactions, the existence of $\psi^{ss}$ and obtaining $\psi^{ss}$ via optimal control problem   will be discussed in this section. Based on the strong gradient formulation \eqref{strong_GF} under detailed balance assumption, a    drift-diffusion approximation, which shares the same energy landscape and same symmetric Hamiltonian structure, gives a good quadratic approximation near not only the 'downhill' solution to the macroscopic RRE but also the 'uphill' least action curve; see Section \ref{subsec_diffusion}. 
 
 If a positive steady state $\pi(\vxv)$ to the mesoscopic master equation \eqref{rp_eq} exists, then one way to obtain $\psi^{ss}$  is from the WKB expansion
$
\psi^{ss}(\vx)=\lim_{V\to +\8} \frac{- \log \pi_V(\vx_V)}{V}.
$ 
Rigorously, under the assumption that there exists a positive detailed balanced $\pi_{\vxv}$,
an USC viscosity solution to the stationary HJE \eqref{HJE2psi} in the Barron-Jensen’s sense \cite{Barron_Jensen_1990}  was constructed in \cite{GL_vis} by using this invariant measure $\pi(\vxv)$.
 For the general case without detailed balance, the existence of viscosity solutions to HJE can also be obtained using the dynamic programming method. In  \cite[Theorem 2.2, Theorem 2.41]{Tran21},   viscosity solutions are constructed via the minimization of the action functional $\inf_{T, \vx(\cdot)}\ac_{x_0,T}$ defined in Section \ref{sec4.0}.  So we use Hamiltonian $H(\vp,\vx)$ to   reformulate the transition path problem as a dual problem of the Maupertuis's principle of the least action problem \eqref{quasipotential}.  
That is to say, we regard $\vp$ as a control variable, then the least action problem is a constrained  optimal control problem in a undefined  time horizon (a.k.a infinite time horizon with an optimal terminal time \cite{fleming06}); see Section \ref{subsec_oc}.  Then the energy landscape $\psi^{ss}$ is then represented as the unique weak KAM solution to HJE satisfying given boundary data on the projected Aubry set since the projected
 Aubry set is a uniqueness set for weak KAM solutions \cite{ishii2020vanishing, Gao_2022}.
}

\subsection{Existence of the stationary solution $\psi^{ss}$ via optimal control and viscosity solution}\label{subsec_oc}
Assume $\xa$ and $\xb$ are two steady  states to RRE belonging to  the same  stoichiometric compatibility class such that $\xa-\xb\in G$.
For the most probable path described by a least action problem with $L$ defined in \eqref{L}, recall the minimization problem \eqref{Ebb} based on Maupertuis's principle. Without   the explicit formula of $L$, using the Hamiltonian $H(\vp,\vx)$,   we first reformulate \eqref{Ebb} as the following control problem. Regarding $\vp$ as a control variable, we minimize the running cost described by the action functional in an undefined time horizon  
 \begin{equation}\label{ocn}
 \begin{aligned}
 &v(\vy; \, \xa, c)=\inf_{T,\vp} \int_0^T \bbs{ \vp \cdot \nabla_p H(\vp,\vx) - H(\vp,\vx) + c  }\ud t,\\
 & \text{s.t.}  \,\, \dot{\vx} = \nabla_p H(\vp,\vx),\,\, t\in(0,T),\quad \vx_0=\xa,\,\, \vx_T =\vy.
 \end{aligned}
 \end{equation}
  Here $v(\vy; \,\xa,c )$ is called the value function, $c \geq c_0$ is an energy level and $c_0$ is a critical minimum energy level  such that 
\begin{equation}
\inf_{T, \vx(\cdot)} \int_0^T (L(\dot{\vx}(t),\vx(t))+c_0) \ud t \geq 0 
\end{equation}
and if $c<c_0$ this $\inf$ becomes $-\8$.
From the definition of the critical minimum energy level \cite{contreras1999global}
\begin{equation}\label{mane}
c_0= \sup\{c\in \bR; \, \exists \text{ closed curve } \vx(\cdot) \text{ s.t. } \int_0^T (L(\dot{\vx}(t),\vx(t))+c) \ud t<0 \},
\end{equation}
it is easy to see that for the Lagrangian $L$ and Hamiltonian $H$ in chemical reactions, the critical level $c_0=0.$
 Indeed, on the one hand, since $L\geq 0$ due to Lemma \ref{lem:least}, so we know at least  $c_0\leq 0$. On the other hand,  if $c_0<0$, then one can choose a standing curve $\vx(t)\equiv y$ at a steady state $y$ of RRE such that $\dot{\vx} = \vec{R}(y)\equiv 0$.
 Then one have $L(\dot{\vx},\vx)\equiv 0$ while $\int_0^T (L(\dot{\vx}(t),\vx(t))+c_0) \ud t<0$.
 
From \cite[Theorem 2.39 and Theorem 2.47]{Tran21}, we know for any $c\geq c_0$, the value function $v(y;\, \xa,c)$ is a viscosity solution to the following static HJE
\begin{equation}
H(\nabla v (\vy) , \vy) = c, \quad \forall \vy
\end{equation}
due to $\xa$ is a steady state of RRE.
In the chemical reactions,  $c=0 = c_0$, the above $v(y;\, \xa)=v(y;\, \xa,0)$ is a viscosity solution to HJE $H(\nabla v (\vy) , \vy) = 0$. However, these viscosity solutions are not unique.

  Now  we describe a selection principle via the weak KAM theory \cite{ishii2020vanishing} and then the global energy landscape $\psi^{ss}$ can be represented via the following weak KAM solution. Assume there are only finite steady solutions  to RRE, denoted as the Aubry set $\mathcal{A}=\{\xa_i\}_{i=1}^J$. Then  
\begin{equation}\label{kam}
\psi^{ss}(\vx)=\min_{\xa\in \mathcal{A}} \bbs{ \psi^{ss}(\xa)+v(\vx;\, \xa)}
\end{equation}
is the unique weak KAM solution to stationary HJE satisfying given boundary data on the
 projected Aubry set since the projected Aubry set is a uniqueness set for weak KAM solutions \cite{Gao_2022}. 
    %The corresponding value function is called the critical Ma\~n\'e potential \cite[Theorem 2.41]{Tran21}.

\begin{rem}\label{tpt}
  In  \cite{Lazarescu_Cossetto_Falasco_Esposito_2019}, \textsc{Lazarescu} et.al. used  a biased Hamiltonian $H$ with observations for the time-averaged flux and concentration to study the dynamic phase transitions in a long time limit. In an open system without detailed balance, with mixed boundary condition and properly chosen bias for fluxes and concentrations, trajectories converging to a constructed global attractor was obtained in \cite{Lazarescu_Cossetto_Falasco_Esposito_2019} while the optimality of the biased Hamiltonian dynamics  in the optimal control context was still unclear.
Beside the deterministic optimal control problem described above, one can also directly investigate the stochastic optimal control problem from the original large number process $\cv $ with a fixed volume $V$. 
The transition path theory   theory (TPT) was first
proposed by E and Vanden-Eijnden in \cite{weinan2006},  particularly in \cite{MSV09} for Markov jumping process, to obtain   transition paths  and transition
rates at a finite noise level by calculating the committor function, i.e., the stationary solution to the backward equation with two boundary conditions at two stable states $A$ and $B$. In \cite{GLLL20}, an optimally controlled random walk is constructed based on the committor function, which realized Monte Carlo simulations for the transition path almost surely. We refer to \cite{gao2020note, gao2020data, gao2021random} for various applications of reversible/irreversible Fokker-Planck equations and the data-driven random walk approximations.  
\end{rem}

\subsection{Construction of a   drift-diffusion process with the same energy barrier for the transition path}\label{subsec_diffusion}
We have shown the law of large numbers gives the macroscopic RRE however the transition path is in the large deviation regime. In this section, we construct a diffusion approximation, which can also be used to approximate the transition path. The most efficient way for constructing a diffusion approximation is through the  Kramers-Moyal approximation   for the  master equation. We will show it is  exactly equivalent to the quadratic approximation of the Hamiltonian near the solution to the macroscopic RRE. Then using the symmetric Hamiltonian, we give a new construction of diffusion approximation that shares the same energy barrier for transition paths.

Near the minimizer of $\ac(\cdot)$, i.e., the curve solves RRE \eqref{odex}, we have the following quadratic approximation for the running cost.
Denote ${\vs}^*:= \nabla_p H(\vp, \vec{x})\big|_{\vp=0} = \sum_{j=1}^M \vec{\nu}_j \bbs{\Phi^+_j(\vec{x}) - \Phi^-_j(\vec{x})}.$ Then we have
\begin{equation}\label{H222}
H(\vp,\vx) = \vs^* \cdot \vp + \frac12 \vp^T \nabla^2_{pp} H(0,\vx) \vp + o(|\vp|^2)
\end{equation}
and for $\vs\in G$,
\begin{equation}
L(\vs,\vx) = \max_{\vp \in G} (\vs-\vs^*)\cdot \vp - \frac12  \vp^T \nabla^2_{pp} H(\vec{0},\vx) \vp + o(|\vp|^2).
\end{equation}
Then approximately we have $\vs-\vs^* =  \nabla^2_{pp} H(0,\vx) \vp^*$ and
\begin{equation}
L(\vs,\vx) \approx    \frac12 \vp^{*T} \nabla_{pp}^2 H(\vec{0},\vx) \vp^*.
\end{equation}
One way of constructing a  Langevin equation with the corresponding  quadratic Hamiltonian \eqref{H222} is
\begin{equation}\label{clr}
\ud  \vx  = \nabla_p H(\vec{0}, \vec{x})  \ud t + \sqrt{\frac{1}{V}\nabla^2_{pp} H(0,\vx)} \ud B.
\end{equation}
Particularly, for our chemical reaction Hamiltonian, $\nabla^2_{pp} H(0,\vx)= \sum_j \bbs{\Phi_j^+(\vx)+\Phi_j^-(\vx)}\vec{\nu}_j \otimes \vec{\nu}_j.$  The above equation is known as the chemical Langevin equation \cite{GG00}.

 We now explain the above quadratic approximation exactly corresponds to  the Kramers-Moyal approximation   for the CME \eqref{rp_eq}.
 The CME \eqref{rp_eq} can be regarded as a monotone scheme for the RRE \eqref{odex}; see \cite{GL_vis}. The leading 
  Taylor expansion for \eqref{rp_eq} upto the second order yields a diffusion approximation  
  \begin{equation}\label{ttt}
  \pt_t p = -\nabla\cdot \bbs{p \sum_{j=1}^M \vec{\nu}_j \bbs{\Phi_j^+(\vx)-\Phi^-_j(\vx)}} + \frac{1}{2V} \sum_{j=1}^M \la \nabla^2\bbs{p(\Phi_j^+(\vx)+\Phi^-_j(\vx))} \vec{\nu}_j, \vec{\nu}_j \ra + O(\frac{1}{V^2}).
  \end{equation}
   This is known as the Kramers-Moyal expansion for the CME.
  This was also used as `system size expansion'   by \textsc{van Kampen}  in \cite{van1992stochastic} and in numerical analysis, it is also called   a modified equation.
  The corresponding Hamiltonian of \eqref{ttt} via the WKB expansion is
  \begin{equation}
  H(\vp,\vx) =  \sum_{j=1}^M \bbs{\Phi_j^+(\vx)-\Phi^-_j(\vx)} \vp \cdot \vec{\nu}_j  + \frac12 \sum_{j=1}^M \bbs{\Phi_j^+(\vx)+\Phi^-_j(\vx)} (\vp\cdot \vec{\nu}_j)^2.  
  \end{equation}
  This is exactly the same as the quadratic approximation \eqref{H222} of the original Hamiltonian  at $\vp=\vec{0}$.

 However, as illustrated in the Schl\"ogl catalysis model in Section \ref{subsub_sch}, we point out the above quadratic approximation for the Hamiltonian  works only for a region  close to solutions to the 'downhill' macroscopic RRE.  On the contrary, the 'uphill' transition path  is apparently a rare transition path in the large deviation regime that is not closed to solutions to the RRE.   \cite{Doering_Sargsyan_Sander_2005} also quantified the failure of the simple diffusion approximation via the Kramers-Moyal expansion  when studying  the extinction problem for stochastic population model, which is also an exit problem in the large deviation regime.

Below,  {\blue we follow the standard procedure for  achieving the fluctuation-dissipation relation to   construct a   diffusion approximation such that (i) the diffusion model satisfies a fluctuation-dissipation relation and yields the same energy landscape as the original chemical reaction process; (ii) the corresponding quadratic Hamiltonian has the same symmetric property; (iii) the diffusion approximation valid near both the 'downhill' RRE solution and the 'uphill' most probable path.  }

 Under symmetric Hamiltonian condition \eqref{newS}, recall the strong gradient flow in terms of energy landscape 
$$\frac{\ud}{\ud t} \vx  = -   K(\vx) {\nabla \psi^{ss}(\vx)},$$
$K(\vx) = \frac12\int_0^1 \nabla^2_{pp} H(\theta \nabla \psi^{ss}(\vx)) \ud \theta$.
Following the standard technique for achieving the fluctuation-dissipation relation, we use the backward Ito's integral to construct a drift-diffusion process 
\begin{equation}\label{df11}
\ud  \vx  =    K(\vx) {\nabla \psi^{ss}(\vx)}   \ud t + \sqrt{\frac{2}{V} K} \, \widehat{\ud} B,
\end{equation}
where $\widehat{\ud} B$ means the multiplicative noise in the backward Ito's integral sense \cite{kunitha1982backward}. In the standard forward Ito's integral sense, this reads
\begin{equation}\label{528}
\ud  \vx  = -   K(\vx) {\nabla \psi^{ss}(\vx)} \ud t + \frac{1}{V}\nabla \cdot K \ud t + \sqrt{\frac{2}{V} K} \ud B.
\end{equation}
Then the Fokker-Planck equation
is
 $$\frac{\pt \rho}{\pt t} = \frac1V \nabla \cdot \bbs{e^{-V \psi^{ss} } K \nabla \bbs{\rho e^{V\psi^{ss}}} }.$$
(i)  This equation   has an  invariant measure $\pi=e^{-V \psi^{ss} }$; (ii) The invariant measure satisfies the detailed balance condition, since the Fokker-Planck operator $\nabla \cdot \bbs{\pi K \nabla \bbs{\frac{\rho}{\pi}} }$ is self-adjoint in $L^2(\pi^{-1})$; (iii) For any convex function $\phi(x)$, the dissipation relation holds
\begin{equation}
\frac{\ud}{\ud t} \int \pi \phi\bbs{\frac{\rho}{\pi}} \ud x = \la \phi'\bbs{\frac{\rho}{\pi}}, \pt_t \rho \ra = -\frac{1}{V} \la K \nabla  \frac{\rho}{\pi}\, , \, \phi''\bbs{\frac{\rho}{\pi}}\nabla \frac{\rho}{\pi} \ra\leq 0.
\end{equation}
 Here the invariant measure  yields exactly the same energy landscape as the original chemical large number process $\psi^{ss}=-\frac{\log \pi}{V}$.
 The corresponding quadratic Hamiltonian is symmetric w.r.t. $\nabla\psi^{ss}(\vx)$
\begin{equation}
H(\vp, \vx) = \bbs{\vp -\nabla \psi^{ss}(\vx)} \cdot K \vp = H({\nabla \psi^{ss}(\vx)}-\vp, \vx).
\end{equation}
We point out this diffusion approximation \eqref{df11} has a covariance $\int_0^1 \nabla^2_{pp} H(\theta \nabla \psi^{ss}(\vx), \vx) \ud \theta$ but the diffusion approximation using chemical Langevin equation \eqref{clr} has a different covariance $ \nabla^2_{pp} H(0,\vx)  $   is different from the previous diffusion approximation \eqref{clr} near RRE at the central limit regime. {\blue We remark the diffusion approximation \eqref{df11} satisfying fluctuation-dissipation relation is also used in \cite{Gingrich_Horowitz_Perunov_England_2016, Polettini_Lazarescu_Esposito_2016, Horowitz_Gingrich_2017} to study the stochastic uncertainty relation for a general process.}

\section{Discussion}
  In this paper, we revisit the macroscopic dynamics for some non-equilibrium chemical reactions from a Hamiltonian viewpoint. The concentration of chemical species  is modeled by the nonlinear RRE system, which is the thermodynamic limiting equation from the law of large numbers for   the random time-changed Poisson representation of chemical reactions. The Hamiltonian defined from the WKB expansion determines a HJE, and the minimizer of the dynamic solution recovers the solution to the RRE. 
  The stationary solution $\psi^{ss}$ to HJE serves as the   energy landscape for general non-equilibrium reactions. The existence of $\psi^{ss}$ is represented as an optimal control problem in an undefined time horizon, which can be represented as a weak KAM solution to HJE. More importantly, we use $\psi^{ss}$ to decompose RRE into a conservative part and dissipative part, which, together with the additional mass conservation law, gives raise a GENERIC formalism for RRE. Through $\psi^{ss}$, the thermodynamics for non-equilibrium reactions can also be decomposed as   nonadiabatic and  adiabatic  parts, where the later one maintains a positive entropy production rate at NESS. 
  %This positive entropy production rate is the most important feature for a non-equilibrium reaction  and maintains an ecosystem in living cells. 
  We then study the energy dissipation relation at both mesoscopic and macroscopic levels and prove the passage from the mesoscopic one to the other. A non-convex energy landscape $\psi^{ss}$ emerges from the convex mesoscopic relative entropy functional $\KL(\rho_{\V}||\pi_{\V})$ in the large number limit, which picks up the non-equilibrium features.
  {\blue This mean-field limit passage also applies to the symmetric property in a chemical reaction. Particularly, the mesoscopic Markov chain detailed balance leads to a symmetric Hamiltonian, while the Markov chain detailed
 balance is  not  equivalent to the more constrained chemical version of detailed balance. The non-convexity of the macroscopic energy landscape $\psi^{ss}$, naturally brought by a grouped polynomial probability flux, enables us to study   a class of non-equilibrium chemical reaction with multiple steady states, for instance  the bistable Schl\"ogl model. However, we point out multiple steady states and non-convex energy are also common in other equilibrium statistical physics  such as  the Lagenvin dynamics with non-convex potential and Ising model of ferromagnetism.} 
  %Indeed, the symmetry of Hamiltonian w.r.t. $\psi^{ss}$ directly reduces general RRE to a strong gradient flow but can still describe a class of non-equilibrium chemical reaction with multiple steady states and positive entropy production rate. 
   We then focus on finding transition paths between coexistent stable steady states in some non-equilibrium biochemical reactions using a symmetric Hamiltonian w.r.t. the stationary solution $\nabla\psi^{ss}$.  Under this symmetric condition, the transition path is explicitly given by piecewise least action curves, where the 'uphill' curve is   a $\nabla\psi^{ss}$-modified time reversal of the 'downhill' least action curve, where $\psi^{ss}$ also gives the energy barriers and path affinities. 
   %We  define a new balance within the same reaction vector to characterize the reversibility condition for the Hamiltonian, which applies to a class of  non-equilibrium catalysis reactions due to flux grouping degeneracy, .
 The bistability and bifurcation in Schl\"ogl's model appear in many general forms, such as the Stuart-Landau equation for general sustained   nonlinear oscillating system with application for the Belousov–Zhabotinsky reaction.  When including spatial variation in the reaction-diffusion equation for spontaneous spatial pattern formation, the double well bistability generates the Turing pattern while the Fisher-KPP bistability generates traveling waves.   We also study  a quadratic approximation for the Hamiltonian near the   RRE solution, i.e., the mean path in the sense of the law of large numbers. However,  we point out the transition path problem in chemical reaction is in the  large deviation regime and the associated energy barrier can not be computed by a simple quadratic approximation. Instead, based on the strong form of gradient flow in terms of free energy $\psi^{ss}$, we construct anther drift-diffusion approximation which shares   the same symmetric Hamiltonian and energy barrier for the most probable path connecting two non-equilibrium steady states.

\section*{Acknowledgements}
The authors would like to thank Jin Feng and Hong Qian for valuable discussions and thank  Alexander Mielke, Mark Peletier and  Michiel Renger    for some insightful suggestions. Yuan Gao was supported by NSF under Award DMS-2204288. J.-G. Liu was supported   by NSF under award DMS-2106988. 

\section*{Data availability statement}
All data generated or analysed during this study are included in this published article.

\appendix

\section{Master equation and generator}\label{app:master}

\subsection{Master equation derivation}
We will only compute the generator for the portion of the forward reactions with the forward Poisson process $Y_j =Y^+_j$ in \eqref{CR}, because the backward portion is exactly same.    Consider 
\begin{equation}
\begin{aligned}
%X_i(t) = X_i(0) + \sum_{j=1}^M \nu_{ji} Y^+_j  \bbs{\int_0^t \varphi^+_j(X(s))\ud s}.\\
\vec{X}(t) = \vec{X}(0) + \sum_{j=1}^M \vec{\nu}_{j}  \mathbbm{1}_{\{\vec{X}(t_-) + \vec{\nu}_j \geq 0 \}}  Y^+_j \bbs{\int_0^t \varphi^+_j(X(s))\ud s}
\end{aligned}
\end{equation}
For any test function $f\in C_b$,  since   $R^+_j(t):=Y^+_j  \bbs{\int_0^t \varphi^+_j(\vec{X}(s))\ud s}$ is a counting process representing the $j$-th reaction, so
\begin{equation}\label{testX}
f(\vec{X}(t))=f(\vec{X}(0)) + \sum_{j=1}^M  \int_0^t \mathbbm{1}_{\{\vec{X}(s_-) + \vec{\nu}_j \geq 0 \}} \bbs{f(\vec{X}(s_-)+\vec{\nu}_{j}) - f(\vec{X}(s_-))} \ud R^+_j(s).
\end{equation}

From \cite[Thm 1.10]{Kurtz15},
\begin{equation}\label{Mg1}
M^+_j(t):= Y^+_j  \bbs{\int_0^t \varphi^+_j(\vec{X}(s))\ud s} - \int_0^t \varphi^+_j(\vec{X}(s))\ud s = R_j^+(t) - \int_0^t \varphi^+_j(\vec{X}(s))\ud s
\end{equation}
is a Martingale. Thus \eqref{testX} becomes
\begin{equation}\label{ito}
\begin{aligned}
f(\vec{X}(t))=&f(\vec{X}(0))  
+ \sum_{j=1}^M  \int_0^t \mathbbm{1}_{\{\vec{X}(s) + \vec{\nu}_j \geq 0 \}} \varphi^+_j( \vec{X}(s) )  \bbs{f(\vec{X}(s)+\vec{\nu}_{j}) - f(\vec{X}(s))}  \ud s\\
&+ \sum_{j=1}^M  \int_0^t \mathbbm{1}_{\{\vec{X}(s_-) + \vec{\nu}_j \geq 0 \}} \bbs{f(\vec{X}(s_-)+\vec{\nu}_{j}) - f(\vec{X}(s_-))} \ud M^+_j(s).
\end{aligned}
\end{equation}

Now we derive the master equation for $\vec{X}(t)\in \mathbb{N}^N$. Denote the (time marginal) law of $\vec{X}(t)$ as
\begin{align}
p(\vec{n}, t) = \bE\bbs{\mathbbm{1}_{\vec{X}(t)}(\vec{n})},
\end{align}
where $\mathbbm{1}$ is the indicator function.
For any $f: \bZ^N \to \bR$, 
$
f(\vec{X}) = \sum_{\vec{n}} f(\vec{n}) \mathbbm{1}_{\vec{X}}(\vec{n}),
$ and
\begin{equation}
\bE (f(\vec{X}) ) = \sum_{\vec{n}} f(\vec{n}) \bE(\mathbbm{1}_{\vec{X}}(\vec{n})) = \sum_{\vec{n}} f(\vec{n})p(\vec{n},t).
\end{equation}
Taking expectation for \eqref{ito}, we have the Dynkin's formula
\begin{equation}\label{Eito}
\begin{aligned}
\bE f(\vec{X}(t))=\bE f(\vec{X}(0)) 
+ \sum_{j=1}^M  \int_0^t \bE \bbs{\mathbbm{1}_{\{\vec{X}(s) + \vec{\nu}_j \geq 0 \}}\varphi^+_j( \vec{X}(s) )  \bbs{f(\vec{X}(s)+\vec{\nu}_{j}) - f(\vec{X}(s))} } \ud s.
\end{aligned}
\end{equation}
Taking derivative yields
\begin{equation}
\begin{aligned}
\frac{\ud }{\ud t}\sum_{\vec{n}}  f(\vec{n}) p(\vec{n}, t) = & \sum_{j=1}^M \sum_{\vec{n}\geq 0,\, \vec{n}+ \vec{\nu}_j \geq 0} \varphi^+_j(\vec{n}) \bbs{f(\vec{n} + \vec{\nu}_{j}) - f(\vec{n}) } p(\vec{n},t)\\
= &\sum_{j=1}^M \sum_{\vec{n}\geq 0,\, \vec{n}+ \vec{\nu}_j \geq 0} \varphi^+_j(\vec{n})  f(\vec{n} + \vec{\nu}_{j})   p(\vec{n},t) - \sum_{j=1}^M \sum_{\vec{n}\geq 0,\, \vec{n}+ \vec{\nu}_j \geq 0} \varphi^+_j(\vec{n})  f(\vec{n} )   p(\vec{n},t)\\
=&  \sum_{j=1}^M \sum_{\vec{n}\geq 0,\, \vec{n}- \vec{\nu}_j \geq 0}  \varphi^+_j(\vec{n}-\vec{\nu}_{j})  f(\vec{n})  p(\vec{n}-\vec{\nu}_{j},t) -  \sum_{j=1}^M \sum_{\vec{n}\geq 0,\, \vec{n}+ \vec{\nu}_j \geq 0} \varphi^+_j(\vec{n})  f(\vec{n})  p(\vec{n},t)\\
=&\sum_{\vec{n}\geq 0  }f(\vec{n}) \bbs{ \sum_{j=1, \, \vec{n}- \vec{\nu}_j \geq 0}^M   \varphi^+_j(\vec{n}-\vec{\nu}_{j})    p(\vec{n}-\vec{\nu}_{j},t) -  \sum_{j=1,\, \vec{n}+ \vec{\nu}_j \geq 0}^M   \varphi^+_j(\vec{n})     p(\vec{n},t)}.
\end{aligned}
\end{equation}
%\begin{equation}\label{Eito}
%\begin{aligned}
%\bE f(\vec{X}(t))=\bE f(X_i(0)) 
%+ \sum_{j=1}^M  \int_0^t \bE \bbs{\varphi_j( X_i(s) )  \bbs{f(X_i(s)+\nu_{ji}) - f(X_i(s))} } \ud s.
%\end{aligned}
%\end{equation}
%Taking derivative implies
%\begin{equation}
%\begin{aligned}
%\frac{\ud }{\ud t}\sum  f(\vec{n}) p(\vec{n}, t) = & \sum_{j=1}^M \sum \varphi_j(\vec{n}) \bbs{f(\vec{n} + \nu_{ji}) - f(\vec{n}) } p(\vec{n},t)\\
%=&  \sum_{j=1}^M \sum \varphi_j(\vec{n}-\nu_{ji})  f(\vec{n})  p(\vec{n}-\nu_{ji},t) -  \sum_{j=1}^M \sum \varphi_j(\vec{n})  f(\vec{n})  p(\vec{n},t).
%\end{aligned}
%\end{equation}
Then the master equation for $p(\vec{n},t)$ is
\begin{equation}
\frac{\ud }{\ud t} p(\vec{n}, t)  =   \sum_{j=1, \, \vec{n}- \vec{\nu}_j \geq 0}^M   \varphi^+_j(\vec{n}-\vec{\nu}_{j})    p(\vec{n}-\vec{\nu}_{j},t) -  \sum_{j=1,\, \vec{n}+ \vec{\nu}_j \geq 0}^M   \varphi^+_j(\vec{n})     p(\vec{n},t).
\end{equation}
After including the backward reactions with $Y^-$,
\begin{equation}
\begin{aligned}
\frac{\ud }{\ud t} p(\vec{n}, t)  =   \sum_{j=1, \, \vec{n}- \vec{\nu}_j \geq 0}^M   \varphi^+_j(\vec{n}-\vec{\nu}_{j})    p(\vec{n}-\vec{\nu}_{j},t) -  \sum_{j=1,\, \vec{n}+ \vec{\nu}_j \geq 0}^M   \varphi^+_j(\vec{n})     p(\vec{n},t)\\
+  \sum_{j=1, \, \vec{n}+ \vec{\nu}_j \geq 0}^M   \varphi^-_j(\vec{n}+\vec{\nu}_{j})    p(\vec{n}+\vec{\nu}_{j},t) -  \sum_{j=1,\, \vec{n}- \vec{\nu}_j \geq 0}^M   \varphi^+_j(\vec{n})     p(\vec{n},t)
\end{aligned}
\end{equation}
Therefore, for the chemical reaction described by \eqref{CR}, the master equation is
\begin{equation}
\begin{aligned}
\frac{\ud }{\ud t} p(\vec{n}, t)  = & \sum_{j=1,  \, \vec{n}- \vec{\nu}_j \geq 0}^M   \bbs{\varphi^+_j(\vec{n}-\vec{\nu}_{j})     p(\vec{n}-\vec{\nu}_{j},t) - \varphi_j^-(\vec{n})   p(\vec{n},t) }  \\
&+   \sum_{j=1,  \, \vec{n}+ \vec{\nu}_j \geq 0}^M \bbs{  \varphi^-_j(\vec{n}+\vec{\nu}_{j})     p(\vec{n}+\vec{\nu}_{j},t)  - \varphi^+_j(\vec{n})    p(\vec{n},t)}.
\end{aligned}
\end{equation}

Similarly, one can derive the master equation for the rescaled large number jumping process $\cv(t)$.

We   only compute the generator for the portion of the forward reactions.
Notice $R^+_j(t)=Y_j  \bbs{V\int_0^t \tilde{\Phi}^+_j(\cv (s))\ud s}$ is a counting process and 
\begin{equation}\label{Mg2}
M^+_j(t) = \frac{1}{V} Y^+_j  \bbs{V \int_0^t \tilde{\Phi}^+_j(\cv (s))\ud s} - \int_0^t \tilde{\Phi}^+_j(\cv (s))\ud s
\end{equation}
is a martingale. Similar to \eqref{ito}, we obtain 
for any $f\in C_b$,  
\begin{equation}\label{itoLL}
\begin{aligned}
f(\cv(t))=&f(\cv(0)) + \sum_{j=1}^M  \int_0^t \mathbbm{1}_{\{\cv(s_-)+\frac{\vec{\nu}_{j}}{V} \geq 0 \}} \bbs{f(\cv (s_-)+\frac{\vec{\nu}_{j}}{V}) - f(\cv (s_-))} \ud R^+_j(s)\\ 
=&f(\cv (0))  
+ \sum_{j=1}^M  \int_0^t V \mathbbm{1}_{\{\cv (s )+\frac{\vec{\nu}_{j}}{V} \geq 0 \}}   \tilde{\Phi}_j( \cv (s) )  \bbs{f(\cv (s)+\frac{\vec{\nu}_{j}}{V}) - f(\cv (s))}  \ud s\\
&+ \sum_{j=1}^M  \int_0^t V \mathbbm{1}_{\{\cv (s_-)+\frac{\vec{\nu}_{j}}{V} \geq 0 \}} \bbs{f(\cv (s_-)+\frac{\vec{\nu}_{j}}{V}) - f(\cv (s_-))} \ud M^+_j(s).
\end{aligned}
\end{equation}
Then using 
$\bE f(\cv (t)) =\frac1V\sum_{\vxv} f(\vxv) p(\vxv, t),$ we obtain the generator $Q_{\V}$ for the large number process $\cv (t)$ for fixed $V$
\begin{equation}\label{p_eq}
\begin{aligned}
 \frac{\ud}{\ud t}\sum_{\vxv}f(\vxv) p(\vxv, t)   =&V\sum_{\vxv\geq 0}  \Big[ \sum_{j=1, \, \vxv+\frac{\vec{\nu_{j}}}{V}\geq 0}^M \tilde{\Phi}^+_j(\vxv)\bbs{ f(\vxv+\frac{\vec{\nu_{j}}}{V}) - f(\vxv)} p(\vxv, t)   \\
&+  \sum_{j=1, \, \vxv-\frac{\vec{\nu_{j}}}{V}\geq 0}^M  \tilde{\Phi}_j^-(\vxv)\bbs{ f(\vxv-\frac{\vec{\nu_{j}}}{V}) - f(\vxv)} p(\vxv, t)   \Big]=:  \sum_{\vxv} (Q_{\V} f)(\vxv)   p(\vxv,t).
\end{aligned}
\end{equation}
Here in the definition of generator, one can define a zero extension for the region outside $\vxv\geq 0$.

\section{Mean-field limit RRE for $\cv$}\label{app_meanfield}
 Since the original proof for the mean filed equation of chemical reaction $\cv$  in \cite{Kurtz71, Kurtz15} omitted the `no reaction' constraint outside nonnegative region, so we provide a pedagogical  proof after including the constraint $\vxv\pm \frac{\vec{\nu}_j}{V}\geq 0$.

%\begin{equation}
%\begin{aligned}
%\cv(t) - \vx(t) =& \cv(0) - \vx(0) + M_{\V}(t) \\
%&+ \int_0^t \Bigg[\sum_{j=1}^M  \vec{\nu}_j  \Big(     \mathbbm{1}_{\{\cv(s_-)+\frac{\vec{\nu}_j}{V}  \geq 0\}} \tilde{\Phi}^+_j(\cv(s)) - \mathbbm{1}_{\{\cv(s_-)-\frac{\vec{\nu}_j}{V}  \geq 0\}}  \tilde{\Phi}^-_j(\cv(s)) \Big) - \vec{R}(\vx(s))\Bigg] \ud s,
%\end{aligned}
%\end{equation}
%\begin{equation}
%\cv(t) - \vx(t) = \cv(0) - \vx(0) + M_{\V}(t) + \int_0^t [\vec{R}(\cv(s)) - \vec{R}(\vx(s))] \ud s,
%\end{equation}

Assume there exists a   solution $\vxv(\cdot)\in C^1([0,T]; \mathbb{R}^N_+)$ to RRE \eqref{odex} and $x_i(t)>0$ for all $t\in[0,T]$ and each component $i$.
Recall $\vec{R}(\vx)= \sum_{j=1}^M \vec{\nu}_j \bbs{\Phi_j^+(\vx) - \Phi_j^-(\vx)}$   defined in \eqref{RRR}.

Fix $a_0>0$ such that the RRE solution tube $\Omega_{a_0}:=\{\vy;\, \max_{t\in[0,T]} |\vy - \vx(t)|< a_0 \} \subset \mathbb{R}^N_+$. For any $0<a<a_0$,  since $\vec{R}$ is locally Lipschitz, there exists $K_a$ such that
$|\vec{R}(\vx)-\vec{R}(\vy)|\leq K_a |\vx-\vy|$ for $\vx, \vy \in \Omega_{a}$. Then we define a stopping time
\begin{equation}
\tau_{\V,a} = \inf\{ t; |\cv(t)-\vx(t)| > a \}.
\end{equation}
Then for $t\leq \tau_{\V,a}$, $\cv\pm \frac{\vec{\nu}_j}{V}\subset \Omega_{a_0} \subset \mathbb{R}^N_+$ for $V$ large enough, so the `no reaction' constraint in process \eqref{Csde} does not turn on before $\tau_{\V,a}$.  Thus 
 from the martingale decomposition \eqref{Mg2}, by Doob's continuous time optional stopping lemma, 
\begin{equation}
\begin{aligned}
M_{\V}(t\wedge \tau_{\V,a}) :=  &\sum_j    { \vec{\nu}_j}\frac{1}{V}  \bbs{Y^+_j  \bbs{V\int_0^{t\wedge \tau_{\V,a}}  \Phi^+_j(\cv(s))\ud s} + Y^-_j  \bbs{V\int_0^{t\wedge \tau_{\V,a}}  \Phi^-_j(\cv(s))\ud s}}\\
& - \sum_j\vec{\nu}_j\int_0^{t\wedge \tau_{\V,a}}  \Phi^+_j(\cv(s))+ \Phi^-_j(\cv(s))\ud s
\end{aligned}
\end{equation}
is a martingale, and thus
\begin{equation}
\begin{aligned}
 \cv(t\wedge \tau_{\V,a})   =&  \cv(0) - \vx(0) + M_{\V}(t\wedge \tau_{\V,a}) + \int_0^{t\wedge \tau_{\V,a}}  \vec{R}(\cv(s))   \ud s.  
\end{aligned}
\end{equation}
 Here for simplicity, we assume the mesoscopic and macroscopic LMA are same $\tilde{\Phi}^\pm_j(\vx) = \Phi^\pm_j(\vx)$ and then drop tilde.
 Compare the trajectory of SDE \eqref{Csde} with solution to RRE \eqref{odex}
\begin{equation}
\begin{aligned}
 \cv(t\wedge \tau_{\V,a}) - \vx(t\wedge \tau_{\V,a})  =&  \cv(0) - \vx(0) + M_{\V}(t\wedge \tau_{\V,a}) + \int_0^{t\wedge \tau_{\V,a}} [\vec{R}(\cv(s)) - \vec{R}(\vx(s))] \ud s. 
\end{aligned}
\end{equation}

 Thus
\begin{equation}
\begin{aligned}
|\cv(t\wedge \tau_{\V,a}) - \vx(t\wedge \tau_{\V,a}) | 
 \leq& |\cv(0) - \vx(0) |+ |M_{\V}(t\wedge \tau_{\V,a})| + K_a \int_0^{t\wedge \tau_{\V,a}} |\cv(s)-\vx(s)|  \ud s|.
\end{aligned}
\end{equation}
Then by Gronwall's inequality, we have
\begin{equation}\label{B5}
\begin{aligned}
|\cv(t\wedge \tau_{\V,a}) - \vx(t\wedge \tau_{\V,a}) |& \leq \bbs{|\cv(0) - \vx(0) + \sup_{1\leq s\leq t\wedge \tau_{\V,a}}|M_{\V}(s)|}e^{K_a t\wedge \tau_{\V,a} }.
\end{aligned}
\end{equation}
Notice that the process $\cv$ is right continuous, then by definition of $\tau_{\V,a}$, we have
\begin{equation}
\begin{aligned}
\{ \max_{0\leq s\leq t} |\cv(s) - \vx(s)| >a \} &\subset \{  |\cv(t\wedge \tau_{\V,a}) - \vx(t\wedge \tau_{\V,a})| \geq a \}. 
\end{aligned}
\end{equation}
Then by \eqref{B5},
\begin{equation}
\begin{aligned}
\{ \max_{0\leq s\leq t} |\cv(s) - \vx(s)| >a \} 
& \subset   \{ |\cv(0) - \vx(0) + \sup_{1\leq s\leq t\wedge \tau_{\V,a} }|M_{\V}(s)| \geq a e^{-K_a t } \}\\
& \subset \{ |\cv(0) - \vx(0) | \geq \frac{a}{2} e^{-K_a t } \} \cup \{ \sup_{1\leq s\leq t\wedge \tau_{\V,a} }|M_{\V}(s)|^2 \geq \frac{a^2}{4} e^{-2K_a t } \}.
\end{aligned}
\end{equation}
Then by Doob’s maximal inequality for submartingales, we know
\begin{equation}
\bP  \{ \sup_{1\leq s\leq t\wedge \tau_{\V,a} }|M_{\V}(s)|^2 \geq \frac{a^2}{4} e^{-2K_a t }   \} \leq \frac{4e^{2K_a t}}{a^2}{\bE(|M_{\V}(t\wedge \tau_{\V,a})|^2)}  
\end{equation}
Using the estimate of martingale $M_{\V}$
$$\bE(|M_{\V}(t\wedge \tau_{\V,a})|^2) =  \sum_{j=1}^M \frac{|\vec{\nu}_j|^2}{V^2}  \bE \bbs{ V\int_0^{t\wedge \tau_{\V,a}}  [\Phi_j^+(\cv(s)) + \Phi_j^-(\cv(s))]  \ud s }\leq C_a \frac{1}{V},$$
 we know
\begin{equation}
\begin{aligned}
\bP \{ \max_{0\leq s\leq t} |\cv(s) - \vx(s)| >a \} &\leq \bP \{ |\cv(0) - \vx(0) | \geq \frac{a}{2} e^{-K_a t } \} + \frac{4e^{2K_a t}}{a^2}{\bE(|M_{\V}(t\wedge \tau_{\V,a})|^2)}\\
&\leq \bP \{ |\cv(0) - \vx(0) | \geq \frac{a}{2} e^{-K_a t } \} + C_a \frac{1}{V}.
\end{aligned}
\end{equation}
Then for arbitrarily small $a$, we conclude that
if $\cv(0) \to \vx(0) $,   \eqref{lln} holds, i.e.,
\begin{equation}
\lim_{V\to +\8} \bP \{ \max_{0\leq s\leq t} |\cv(s) - \vx(s)| >a \} = 0.
\end{equation}

\section{Proof of Lemma \ref{lem_cDB}}\label{app_lem_cDB}
We give the proof of Lemma \ref{lem_cDB} by  some elementary computations and collecting existing results.
 \begin{proof}[Proof of Lemma \ref{lem_cDB}]
Step 1. We prove the equivalence between (i) and (ii). 

Plugging   identity \eqref{complexIND}, we obtain the identity
\begin{equation}\label{tm216}
\begin{aligned}
H( \log \frac{\vec{x}}{\peq},\vx)=& \sum_j \bbs{ \Phi^-_j(\vx) \bbs{\frac{\Phi^+_j(\peq)}{\Phi_j^-(\peq)} -1} + \Phi_j^+ (\vx) \bbs{\frac{\Phi^-_j(\peq)}{\Phi_j^+(\peq)} -1} }\\
= & \sum_j \bbs{  \frac{\Phi^-_j(\vx)}{\Phi_j^-(\peq)}  \bbs{  {\Phi^+_j(\peq)}-{\Phi_j^-(\peq)} } + \frac{\Phi_j^+ (\vx)}{\Phi^+_j (\peq)} \bbs{{\Phi^-_j(\peq)}-{\Phi_j^+(\peq)} } }\\
= & \sum_j   \bbs{\frac{\vx}{\peq}}^{\vec{\nu}_j^-} \bbs{  {\Phi^+_j(\peq)}-{\Phi_j^-(\peq)}} + \sum_j \bbs{\frac{\vx}{\peq}}^{\vec{\nu}_j^+} \bbs{{\Phi^-_j(\peq)}-{\Phi_j^+(\peq)} } .
\end{aligned}
\end{equation}

 Similar to \eqref{Kflux}, rearranging according to the reactant complex $\vec{\eta}\in \mathcal{C}$, we have
\begin{equation}\label{tm217}
\begin{aligned}
 H( \log \frac{\vec{x}}{\peq}, \vx)=\sum_{\vec{\eta}\in\mathcal{C}}  \bbs{\frac{\vx}{\peq}}^{\vec{\eta}}   \bbs{ \sum_{j: \vec{\nu}_j^-=\vec{\eta}}\bbs{  {\Phi^+_j(\peq)}-{\Phi_j^-(\peq)} } + \sum_{j: \vec{\nu}_j^+=\vec{\eta}} \bbs{{\Phi^-_j(\peq)}-{\Phi_j^+(\peq)} } }.
\end{aligned}
\end{equation}
Then  complex  balance \eqref{cDB} is equivalent to $H( \log \frac{\vec{x}}{\peq}, \vx)=0.$

 Step 2, assume (i), i.e., $\peq$ satisfies complex  balance condition \eqref{cDB}, then one can construct  a  stationary distribution $\pi_V $ via the product of Poisson distributions with intensity $V \peq$  \cite[Theorem 3.7]{Kurtz15} {(see also \cite{Anderson_Craciun_Kurtz_2010})}
\begin{equation}\label{piv}
\log \pi_V(\vxv) = \sum_{i=1}^{N} \Big( n_i \log(V x^s_i) - \log (n_i !) - V x^s_i \Big), \quad \vec{n}:= V \vxv
\end{equation}
 for the chemical master equation \eqref{rp_eq} with $\tilde{\Phi}= {\varphi}/V$ for  a fixed volume $V$. Thus (i) implies (iii).
 %This construction based on the observation that for the same complex,  the corresponding total flux is still given by the law of mass action. 

Step 3, assume (iii), for any $\vx\in \bR^N_+$, let $\vx_V=\frac{\vec{n}}{V}\to \vx$ as $V\to +\8$,   then the limit in WKB approximation for $\pi_V$ exists
\begin{equation}
  \lim_{V\to +\8} \frac{- \log \pi_V(\vx_V)}{V} =\sum_{i=1}^N \bbs{ x_i \log x_i -x_i \log x^s_i   + x^s_i -x_i   }=\KL(\vx||\peq)= \psi^{ss}(\vx).
\end{equation}
Indeed, changing to variable $\vx_V=\frac{\vec{n}}{V}$ and using the Stirling's formula,  we have
 \begin{equation}
 \begin{aligned}
 \frac{\log \pi_V(\vx_V)}{V} =& \frac{1}{V}   \sum_{i=1}^N \bbs{ n_i \log(V x^s_i) - n_i\log (n_i) + n_i - V x^s_i +O(\log n_i)  } \\
 =& \sum_{i=1}^N \bbs{ x_i \log x^s_i   - x_i \log x_i + x_i -x^s_i   } + \frac{O(\sum_i^N \log n_i)}{V}.
 \end{aligned}
 \end{equation}
 Then for any fixed  $\vx\in \bR^N$, $\vx_V=\frac{\vec{n}}{V}\to \vx$ implies $\frac{O(\sum_i^N \log n_i)}{V}\to 0$ as $V\to+\8.$
Thus (ii) follows.  
\end{proof}

\section{Phosphorylation-dephosphorylation with 2-autocatalysis}\label{app_2d}
  The Schl\"ogl model can be regarded as a simple but representative example which keeps the main features of non-equilibrium enzyme reactions.  As the one of the most important enzyme reaction in a single living cell, the phosphorylation-dephosphorylation reaction system (c.f. \cite{qian2007phosphorylation})  also fits into the symmetric Hamiltonian framework.

Here we  briefly revisit the phosphorylation-dephosphorylation with 2-autocatalysis proposed by  Fischer-Krebs in 50's.
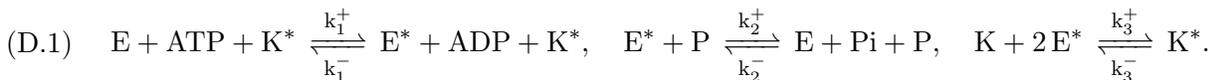
\begin{equation}
\begin{aligned}
\ce{E + ATP + {K^*}} \,  \ce{<=>[k^+_1][k^-_1]} \,    \ce{E^* + ADP + {K^*}}, 
\quad  
\ce{ E^* + {P}}   \, \ce{<=>[k^+_2][k^-_2]} \, \ce{ E + Pi + {P}},
\quad 
\ce{{K} + 2E^*} \, \ce{<=>[k^+_3][k^-_3]} \,    \ce{{K^*}}. 
\end{aligned}
\end{equation}
Here the concentration of a protein in its open state (the phosphorylated \ce{E}) is denoted as  $x(t) = \ce{[E^*]}$ while the  concentration of a protein in its close state is  $y(t) = \ce{[E]}$.
The third reaction equation representing the reversible binding is rapid and thus  is assumed to be quasi-static. Under this quasi-static assumption, the active kinase \ce{K^*} in the first reaction equation has a positive feedback from 2\ce{E^*}, which is known as 2-autocatalysis.  We also regard the concentrations of the inactive kinase \ce{K}, phosphatase \ce{P}, adenosine triphosphate \ce{ATP},  adenosine diphosphate \ce{ADP} and phosphate group \ce{Pi} as constant that sustained by environment.
 
From
$\dot{y} = - \dot{x}$, we know the conservation of total mass of two proteins and thus $y(t) = \ce{[E_{tot}] - x(t)}$. The RRE is given by
\begin{equation}\label{PP}
\dot{x} = (a^+_1\ce{[K]} x^2y + a^-_2y )
- ( a^-_1 \ce{[K]} x^3 + a^+_2x ) =: [\Phi_1^+(x) + \Phi_2^+(x)] - [\Phi_1^-(x)+ \Phi_2^-(x)]
\end{equation}
where we lumped chemostats into rates and used quasi-static relation:
$$
a^+_1=k^+_1\ce{[ATP]} \frac{k^+_3}{k^-_3}, \quad
a^-_1=k^-_1\ce{[ADP]} \frac{k^+_3}{k^-_3}, \quad
a^+_2=k^-+2\ce{[P]}, \quad
a^-_2=k^-_2\ce{[Pi][P]}. 
$$
We take \ce{[K]} as a bifurcation parameter for the first order phase transition. The right-hand-side of \eqref{PP} is a double well potential raising from the flux grouping property. This 2-autocatalysis model is basically same as the Schl\"ogl model after effectively eliminating the quasi-static third reaction equation, so the mathematical analysis are same.

\bibliographystyle{alpha}
\bibliography{LDbib}

\end{document}